\documentclass[12pt,a4paper]{amsart}
\usepackage{amsfonts,amsmath,amssymb}
\usepackage{hyperref}
\usepackage{latexsym,fullpage,amsfonts,amssymb,amsmath,amscd,graphics,epic}
\usepackage[all]{xy}
\usepackage{amssymb,amsbsy,amsthm,amsxtra}
\usepackage[usenames]{color}
\usepackage{amscd}
\usepackage{amsthm}
\usepackage{amsfonts}
\usepackage{amssymb}
\usepackage{mathrsfs}
\usepackage{url}
\usepackage{bbm}
\usepackage{wasysym}
\usepackage{enumitem}
\usepackage{framed}
\usepackage{nicematrix}
\usepackage{setspace}
\usepackage{comment}
\usepackage[normalem]{ulem}

\usepackage[capitalise]{cleveref}

\usepackage{ytableau}%
\ytableausetup{smalltableaux} 
\usepackage{dynkin-diagrams}

\usepackage{ctable}

\theoremstyle{plain}% default
\newtheorem*{theorem*}{Theorem}
\newtheorem{lemma}{Lemma}[subsection]
\newtheorem{proposition}[lemma]{Proposition}
\newtheorem{corollary}[lemma]{Corollary}
\newtheorem{theorem}[lemma]{Theorem}

\newtheorem*{conjecture*}{Conjecture}

\newtheorem{introthm}{Theorem}

\newtheorem{introcorollary}[introthm]{Corollary}
\newtheorem{introprop}[introthm]{Proposition}

\theoremstyle{definition}
\newtheorem{definition}[lemma]{Definition}

\newtheorem{example}[lemma]{Example}

\newtheorem*{example*}{Example}

\theoremstyle{remark}
\newtheorem*{remark*}{Remark}
\newtheorem{remark}[lemma]{Remark}

\newtheorem{notation}[lemma]{Notation}

\newcommand{\Hom}{\operatorname{Hom}}

\newcommand{\diag}{\operatorname{diag}}

\newcommand{\triv}{{\mathbbm 1}}

\newcommand{\id}{\operatorname{Id}}

\renewcommand{\Im}{\operatorname{Im}}
\newcommand{\Ker}{\operatorname{Ker}}

\newcommand{\Aut}{{\operatorname{Aut}}}

\newcommand{\End}{\operatorname{End}}

\newcommand{\sdim}{\operatorname{sdim}}

\newcommand{\C}{{\kk}}
\newcommand{\Z}{{\mathbb Z}}

\newcommand{\fm}{{\mathfrak m}}

\newcommand{\eps}{{\varepsilon}}

\newcommand{\lam}{{\lambda}}

\newcommand{\g}{{\mathfrak g}}
\newcommand{\gl}{{\mathfrak{gl}}}
\renewcommand{\sl}{{\mathfrak{sl}}}

\newcommand{\osp}{{\mathfrak{osp}}}

\newcommand{\p}{\mathfrak{p}}
\newcommand{\q}{\mathfrak{q}}
\newcommand{\s}{\mathfrak{s}}
\newcommand{\h}{\mathfrak{h}}
\newcommand{\fc}{\mathfrak{c}}
\newcommand{\rr}{\mathfrak{r}}
\newcommand{\fl}{\mathfrak{l}}
\newcommand{\ft}{\mathfrak{t}}

\newcommand{\Lie}{\operatorname{Lie}}
\newcommand{\ad}{\operatorname{ad}}
\newcommand{\Ad}{\operatorname{Ad}}

\newcommand{\abs}[1]{\left|{#1}\right|}

\newcommand{\MM}{\pmb{\rm{M}}}

\newcommand{\sVect}{\mathtt{sVect}}

\newcommand{\Der}{\mathrm{Der}}

\newcommand{\Rep}{\mathrm{Rep}}

\newcommand{\G}{\mathbb{G}_a}
\newcommand{\kk}{\mathbbm{k}}
\renewcommand{\ss}{\mathbf{ss}}
\newcommand{\Inna}[1]{\begin{framed} {\tt{\color{blue}{#1}}} \end{framed}} 
\newcommand{\Innas}[1]{\fbox{\tt{\color{blue}{#1}}}}

\newcommand{\InnaA}[1]{#1}

\def\quotient#1#2{%
    \raise1ex\hbox{$#1$}\Big/\lower1ex\hbox{$#2$}%
}

\begin{document}

\date{\today}
\title{Balanced and neat elements in quasi-reductive Lie superalgebras}
\author{Inna Entova-Aizenbud, Vera Serganova}

 \begin{abstract}
 
Let $G$ be a quasi-reductive supergroup (so its underlying algebraic group $G_{\bar 0}$ is reductive). 
To understand better odd elements in quasi-reductive Lie superalgebras, we consider two trivially intersecting classes of odd elements, which turn out to be particularly convenient to work with: {\it neat} odd elements and {\it balanced} odd elements. 

Neat elements are always $\ad$-nilpotent. Neat elements are the best analogues one can consider for nilpotent elements in a semisimple Lie algebra, since they may be embedded into subalgebras that are isomorphic to $\osp(1|2)$, a simple Lie superalgebra whose underlying Lie algebra is $\sl_2$. Balanced odd elements, on the other hand, are a natural generalization of the notion of a self-commuting element (an element $x\in \mathrm{Lie}(G)_{\bar 1}$ for which $[x,x]=0$). Balanced elements are used to define homology-type functors on the category of representations of $G$.

We study the properties of balanced and neat odd elements in quasi-reductive Lie superalgebras and show that any element $x\in  \mathrm{Lie}(G)_{\bar 1}$ may be written as a sum of a neat and a balanced odd element which commute with each other. 

This theorem has a very nice categorical application. 

Let $\g^{(1|1)}$ be the $(1|1)$-dimensional Lie superalgebra generated by an odd element $x$. The semisimplification of the category of finite-dimensional super-representations of $\g^{(1|1)}$ is a symmetric monoidal functor $S:  \mathrm{Rep}(\g^{(1|1)}) \to \mathrm{Rep}(SOSp(1|2))$.

Any $x\in 
\mathrm{Lie}(G)_{\bar 1}$ induces a homomorphism $ i_x:\g^{(1|1)}\to \mathrm{Lie}(G)$. Let $$\Phi_x=S\circ (-)\downarrow_{i_x}:\mathrm{Rep}(G)\to \mathrm{Rep}(SOSp(1|2))$$ be the composition of the restriction functor $(-)\downarrow_{i_x}$ and the semisimplification functor $S $. 

We show that the functor $\Phi_x$ may be described explicitly using the homology-type functor $\Phi_{x_{bal}}$ corresponding to the balanced part of $x$ in the above decomposition. These homology-type functors are known as Duflo-Serganova functors.

Finally, we provide a full classification of distinguished odd elements in simple quasi-reductive Lie superalgebras and show that in all cases except $\mathfrak{spe}(n)$, such elements are either balanced or neat.
 \end{abstract}

%\keywords{}

\maketitle
\setcounter{tocdepth}{1}
\tableofcontents

\section{Introduction}

\subsection{}
The study of nilpotent elements and their orbits is a central theme in the representation theory of semisimple Lie algebras. In this paper, we propose a new approach to the study of odd elements in quasi-reductive Lie superalgebras. 

An algebraic supergroup $G$ is called quasi-reductive if its underlying algebraic group $G_{\bar{0}}$ is reductive. 
A Lie superalgebra $\mathfrak{g} = \mathfrak{g}_{\bar{0}} \oplus \mathfrak{g}_{\bar 1}$ is quasi-reductive if its even part $\mathfrak{g}_{\bar{0}}$ is reductive and $\mathfrak{g}_{\bar{0}}$ acts semisimply on $\mathfrak{g}_{\bar 1}$. Examples of such structures include the classical families $\mathfrak{gl}(m|n)$, $\mathfrak{osp}(m|2n)$ and the strange Lie superalgebras $\mathfrak{pe}(n)$ and $\mathfrak{q}(n)$.

Let $G$ be a quasi-reductive supergroup and $\g:=\mathrm{Lie}~G$ its Lie superalgebra. We will consider orbits of elements in $\g_{\bar 1}$ under the adjoint action of the underlying reductive algebraic group $G_{\bar{0}}$.

The orbits of $\ad$-nilpotent odd elements in a quasi-reductive Lie superalgebra may exhibit a behaviour very different from the nilpotent orbits in a semisimple Lie algebra: for example, orbits in a Lie superalgebra do not have to be conical and there might be infinitely many such orbits.

\subsection{The additive and orthosymplectic supergroups}

%We will say that $x\in \g_{\bar 1}$ is {\it $\ad$-nilpotent} if $[x,x]$ is a nilpotent element in the reductive Lie algebra $\g_{\bar 0}$ (that is, $[x,x]\in [\g_{\bar 0}, \g_{\bar 0}]$ and $\ad_x\in \End(\g)$ is a nilpotent operator). 
Let $\kk$ be an algebraically closed field of characteristic zero.

Let $OSp(1|2)$ be the (connected) orthosymplectic supergroup with the Lie superalgebra
 $$\osp(1|2)=\sl_2\oplus \kk^2, \;\;\; \sl_2=\mathrm{span}\{e,h,f\},\;\;\; \kk^2=\mathrm{span}\{X, Y\}$$
 where $[X,X]=f, ~[Y, Y]=e,~ [Y, X]=2h $. Let $\G^{(1|1)}$ be the $(1|1)$-dimensional additive supergroup, whose Lie superalgebra is $ \g_a=\mathrm{span}\{f\}\oplus \mathrm{span}\{x\}$ with $[f,x]=0,~ [x,x]=f$. We have a natural embedding $\iota: \G^{(1|1)}\hookrightarrow OSp(1|2)$ corresponding to $x\mapsto X, ~ f\mapsto f$. 

 Given a quasi-reductive supergroup $G$, the $\ad$-nilpotent odd elements $x\in \g_{\bar 1}$ are in bijection with  homomorphisms $ \G^{(1|1)}\to G$, so the supergroup $\G^{(1|1)} $ is naturally one of the main objects of study in the theory of nilpotent elements in supergroups.

 \subsection{Two types of nilpotent operators}
 
Consider a vector superspace $V$ and an odd nilpotent endomorphism $x$ of $V$ (equivalently, an action of $\G^{(1|1)}$ on $V$). Then there exists an increasing filtration $V=\bigcup_i V^i$, the {\it Deligne filtration for an odd nilpotent endomorphism} such that 

\begin{itemize}
  \item $x( V^{k})\subset \Pi V^{k-2}$;
  \item $\forall~k\geq 0$, $x^{{k}}: Gr^{k}(V)\xrightarrow{\sim} \Pi^kGr^{-k}(V) $, where $Gr^k(V):=V^{k}/ V^{k-1}$.
  \end{itemize}

 One may then define an action of $OSp(1|2)$ on $Gr^{ev}(V):=\bigoplus_k Gr^{2k}(V)$ so that the restriction of this action via $\iota:\G^{(1|1)}\hookrightarrow OSp(1|2)$ gives the action of $x$ on $Gr^{ev}(V)$.
   
In down-to-earth terms, the vector superspace $Gr^{ev}(V)$ is isomorphic to the direct sum of the odd-sized Jordan blocks of $x$.

Let us consider two special types of odd nilpotent operators $x$.
    \begin{itemize}
        \item If all the non-trivial Jordan blocks of $x:V\to V$ have even sizes then $OSp(1|2) $ acts trivially on $Gr^{ev}(V) = \Ker(x)/(\Im(x)\cap \Ker(x))$.
        Such $x$ is called {\bf balanced}. A special subclass of balanced operators are {\bf self-commuting} operators: these are $x$ for which $x\circ x=0 $, so $Gr^{ev}(V)=Gr^0(V)=\Ker(x)/\Im(x)$.
        \item If the Jordan blocks of $x$ have {\it odd sizes} then $Gr^{ev}(V) \cong V$, so we obtain an action $OSp(1|2)\curvearrowright V$. Such $x$ is called {\bf neat}.
        \end{itemize}

The self-commuting operators are clearly balanced and the sets of neat and balanced operators intersect only at zero. Moreover, one can easily see that any odd operator $x$ can be presented as a sum of a neat operator and a balanced operator, by decomposing $x$ into a sum of Jordan blocks of odd and even sizes. 

\mbox{}

The first aim of this paper is to generalize this picture to arbitrary odd elements in quasi-reductive Lie superalgebras.

Let $G$ be a quasi-reductive group, $\g:=\Lie(G)$. Given an element $x\in \g_{\bar 1}$, we may consider the Jordan decomposition of $\frac{1}{2}[x,x]=s+f \in \g_{\bar 0}$ into semisimple and nilpotent parts. If $s=0$ we will say that $x$ is {\bf $\ad$-nilpotent}; in this case, $x\rvert_V$ is a nilpotent operator for any $V\in \Rep(G)$.

An $\ad$-nilpotent element $x\in \g_{\bar 1}$ is called {\bf balanced} (respectively, {\bf neat}) if $x\rvert_{V}$ is a balanced (respectively, neat) odd nilpotent operator for any $V\in \Rep(G)$. A {\bf self-commuting} element in $\g_{\bar 1}$ is an element $x$ such that $[x,x]=0$; it is automatically balanced.

\begin{example*}
An $\ad$-nilpotent element $x\in \gl(m|n)_{\bar 1}$ is given by an odd nilpotent operator $\phi_x\in \End^{\bullet}(\kk^{m|n})$. It turns out that such an element $x$ is balanced (respectively, neat) if and only if $\phi_x$ is a balanced (respectively, neat) operator. 
\end{example*}

In our previous work (\cite{entova2022jacobson}), we proved the following criterion, which shows that neat elements enjoy a Jacobson-Morozov property similar to nilpotent elements in reductive Lie algebras.
\begin{theorem*}[Jacobson-Morozov theorem for supergroups, see \cite{entova2022jacobson}]
    Let $G$ be a quasi-reductive supergroup, $\g:=\Lie(G)$ and let $x\in \g_{\bar 1}$ be $\ad$-nilpotent. 
    
    Let $i:\G^{(1|1)}\to G$ be a homomorphism corresponding to $x$. Then $x$ is neat if and only if there exists a homomorphism $\overline{i}:OSp(1|2)\to G$ such that $i=\overline{i}\circ \iota$, where $$\iota:\G^{(1|1)}\hookrightarrow OSp(1|2)$$ is the embedding we fixed before.
\end{theorem*}

We may further generalize the notions of a balanced or a self-commuting element to any odd element $x\in \g_{\bar 1}$: denoting by $s$ the semisimple part of $\frac{1}{2}[x,x]\in \g_{\bar 0}$, we say that $x$ is {\bf balanced} if and only if $x\rvert_{V^s}$ is a balanced odd nilpotent operator for any $V\in \Rep(G)$. We say that $x$ is {\bf homological} if and only if $ \frac{1}{2}[x,x]=s$. The properties of being self-commuting, homological, balanced and neat are clearly invariant under $G_{\bar 0}$-conjugation.
\begin{remark}
    Any self-commuting element is homological and any homological element is balanced.
\end{remark}

The following theorem, proved in a slightly stronger form (see \cref{thm:2-step_exist}), is one of the main results in this paper:
\begin{introthm}\label{introthrm:2-step}
    Let $G$ be a quasi-reductive supergroup, $\g:=\Lie(G)$ and let $x\in\g_{\bar 1}$. 
There exists a decomposition 
$x=x_{bal}+x_{neat}$, such that 
$x_{bal}, x_{neat} \in \g_{\bar 1}$, $x_{bal} $ is balanced, $x_{neat} $ is neat and $[x_{bal},x_{neat}]=0$. 
   
\end{introthm}
This decomposition is called a {\bf $2$-step decomposition of $x$}. Such a decomposition is not necessarily unique.

To better understand neat and balanced elements, we provide intrinsic criteria for an element $x\in \g_{\bar 1}$ to be neat or balanced. 

Let $x\in\g_{\bar 1}$ and let $\frac{1}{2}[x,x]=s+f$ be the Jordan decomposition of $\frac{1}{2}[x,x]$, where $s$ is the semisimple part and $f$ is the nilpotent part. 

Assume that $f\neq 0$. Fix an $\sl_2$-triple $(e,h,f)$ such that $[s, h]=0$ and let us write 
\begin{equation}\label{introeq:h_eigenvectors}
x=x_0+x_{-1}+\dots+x_{-k}
\end{equation} where $[h,x_i]=ix_i$, $[s,x_i]=0$ and $[f, x_i]=0$ for all $i$. The element $x_0$ is a homological element: $\frac{1}{2}[x_0,x_0]$ is the semisimple part in the Jordan decomposition of $\frac{1}{2}[x,x]\in \g_{\bar 0}$. We call $x_0$ an {\it attractor} of $x$. 

For a different choice of the Cartan element $h$, one obtains a conjugate attractor. We denote the set of attractors of $x$ by $att(x)$.
If $\frac{1}{2}[x,x]$ is semisimple (so $f=0$), we call $x_0:=x$ its own attractor.

We may now formulate the intrinsic criteria for an odd element to be neat or balanced (see \cref{lem:criterion_neatness_h_decomp}, \cref{ssec:strong_bal_implies_bal} and \cref{cor:decomp_for_neat_or_bal_elems}):
\begin{introprop}\label{introprop:neat_bal_criteria}
Let $x\in \g_{\bar 1}$ and let $\frac{1}{2}[x,x]=s+f$ be the Jordan decomposition as before. 
\begin{itemize}
     \item The element $x$ is neat iff $att(x)=\{0\}$. In that case, $x_{-1}$ is conjugate to $x$.
    \item The element $x$ is balanced iff it is {\it strongly balanced}: that is, either $x$ is homological, or there exists a non-trivial $\sl_2$-triple $(e,h,f)$ as above so that $x_{-1}=0$.
\end{itemize}
\end{introprop}
 The two-step decomposition shown in \cref{introthrm:2-step} interacts nicely with attractors: one can choose it so that $x, x_{bal}$ have a common attractor $x_0$ and so that $x_0$ commutes with the $\osp(1|2)$-type subalgebra of $\g$ corresponding to $x_{neat}$.
\subsection{Attractors of odd elements}

The attractors of an odd element $x$ play an important role: they lie in the closure of the orbit $G_{\bar 0}.x$ and convey important information about the element $x$ and its action on $G$-modules, as we show in \cref{introprop:neat_bal_criteria} and \cref{introthrm:functors}.

Furthermore, we prove the following statement (see \cref{thm:finattractor}): 
\begin{introthm}\label{introthm:finite_fibers}
Let $\g^{hom}$ denote the set of homological elements.
    The correspondence $$att:\g_{\bar 1}/G_{\bar 0}\; \longrightarrow \;\g^{hom}/G_{\bar 0}, \;\;\;\;\;\;G_{\bar 0}.x \; \longmapsto \InnaA{G_{\bar 0}.}att(x)$$ between $G_{\bar 0}$-orbits in $\g_{\bar 1}$ and their attractors is a surjective map with finite fibers.
\end{introthm}
\InnaA{
\begin{introcorollary}
Let $\g^{\mathbf{sc}}:=\{x\in \g_{\bar 1}: [x,x]=0\} $ and $ \mathcal{N}^{\ad}(\g_{\bar 1})$ denote the sets of self-commuting (respectively, $\ad$-nilpotent) odd elements in a quasi-reductive Lie superalgebra $\g$. Then $\mathcal{N}^{\ad}(\g_{\bar 1})$ has finitely many $G_{\bar 0}$-orbits if and only if $\g^{\mathbf{sc}}$ has finitely many $G_{\bar 0}$-orbits.
\end{introcorollary}}

As one can see from \cref{introprop:neat_bal_criteria}, the fiber $att^{-1}(\{0\})$ consists of neat orbits in $\g_{\bar 1}$. It was proved in \cref{cor:fin_many_neat_orb} that every quasi-reductive Lie superalgebra has finitely many neat orbits and \cref{introthm:finite_fibers} is a direct generalization of that statement.

Unlike nilpotent orbits in finite-dimensional Lie algebras, there exist Lie superalgebras such as $\sl(n|n)$ which have infinitely many odd $\ad$-nilpotent orbits. \cref{introthm:finite_fibers} allows one to reduce problems concerning the number and the structure of $G_{\bar 0}$-orbits in $\g_{\bar 1}$ to the study of the set $\g^{hom}$, which has already been studied extensively (see \cite{duflo2005associated, gorelik2022duflo}).

\subsection{Distinguished odd elements in quasi-reductive Lie superalgebras}

Let $G$ be a quasi-reductive supergroup and $\g=\Lie(G)$ be its Lie superalgebra. We may define Levi subsuperalgebras in $\g$ in the same way one defines Levi subalgebras in reductive Lie algebras: as centralizers of (purely even) toral subalgebras of $\g_{\bar 0}$.

As in the classical theory of nilpotent orbits, we say that 
$x$ is {\bf distinguished} iff $x $ is not contained in any proper Levi sub-superalgebra of $\g$.

We classify all the distinguished orbits in simple quasi-reductive Lie superalgebras and in Takiff Lie superalgebras (this is done in \cref{sec:orbits_in_classical}). Using this classification, we prove the following dichotomy (see \cref{thrm:dist_is_neat_or_balanced}; also used to prove \cref{introthrm:2-step}):
\begin{introthm}
    Let $\g$ be a simple quasi-reductive Lie superalgebra or a Takiff Lie superalgebra, such that $\g\not \cong \mathfrak{spe}(n)$ for any $n$. Let $x\in \g_{\bar 1}$ be a distinguished odd element. Then $x$ is either strongly balanced or neat.
\end{introthm}

\subsection{The categorical picture}
Finally, let us describe the categorical side of the above definitions.

Let us consider the functor $Gr^{ev}: \Rep(\G^{(1|1)}) \to \Rep(OSp(1|2))$. This functor is a symmetric monoidal functor, though not coming from a homomorphism of supergroups. In fact, it is a semisimplification functor (see \cite{entova2022jacobson, etingof2021semisimplification}), annihilating the indecomposable $\G^{(1|1)}$-modules of zero superdimension.

Let $G$ be a quasi-reductive supergroup, $\g:=\Lie(G)$ and $x\in \g_{\bar 1}$, with $s$ denoting the semisimple part in the Jordan decomposition of $\frac{1}{2}[x,x]$.

We denote by $\Rep(G)$ the category of finite-dimensional super-representations of $G$ and similarly for $\Rep(\G^{(1|1)}), \Rep(OSp(1|2)).$ \InnaA{Given any $M\in \Rep(G)$, we may consider the subspace of $s$-invariants $M^s$, on which $x$ acts nilpotently. This endows $M^s$ with action of $\G^{(1|1)}$ corresponding to $x$, and we obtain a functor $\Rep(G)\to\Rep(\G^{(1|1)})$.}

Consider the categorical picture corresponding to an odd element $x\in \g_{\bar 1}$. We denote by $(-)\downarrow$ the corresponding restriction functors, \InnaA{and by $(-)^s$ the functor of $s$-invariants:}
 $$\xymatrix{&\Rep(G) \ar[rr]^{(-)^s\downarrow_{\G^{(1|1)}}} \ar[rrd]_{\;\;\;\;  Gr^{ev}\circ (-)^s\downarrow_{\G^{(1|1)}} \;\;\;\; } &{} &\Rep(\G^{(1|1)})\ar@/_0.5pc/[d]_{Gr^{ev}} \\ &{} &{} &\Rep(OSp(1|2))  \ar@/_0.5pc/[u]_{(-)\downarrow^{OSp(1|2)}_{\G^{(1|1)}}}  }$$

%Note that $ Gr^{ev}\circ (-)\downarrow^{OSp(1|2)}_{\G^{(1|1)}}\cong \id$. 

Consider the symmetric monoidal functor $\Phi_x:=Gr^{ev}\circ (-)^s\downarrow_{\G^{(1|1)}}$ from $\Rep(G)$ to $\Rep(OSp(1|2))$. 

When $x$ is neat, the Jacobson-Morozov theorem for supergroups implies that this functor is isomorphic to a restriction functor with respect to a homomorphism $OSp(1|2)\to G$ (see \cite{entova2022jacobson}); in particular, it is a faithful and exact symmetric monoidal functor. On the other hand, if $x$ is balanced then the action of $OSp(1|2)$ on $\Phi_x(M)$ is trivial for any $M\in \Rep(G)$ and the functor $\Phi_X$ is not faithful if $x\neq 0$.

An important special case is when $x$ is a homological element, i.e. $s:=\frac{1}{2}[x,x] \in \g_{\bar 0}$ is semisimple. In that case, the functor $$\Phi_x(-)=\frac{\Ker(x)\cap \Ker(s)}{\Im(x)\cap \Ker(s)}$$ is called the {\it Duflo-Serganova functor} $DS_x$. Such functors have been defined in \cite{duflo2005associated} for self-commuting elements $x$ (see \cite{gorelik2022duflo} for a detailed survey and more recent results). The Duflo-Serganova functors have been extensively studied and used to construct support theory for $G$-modules. 

Our main categorical result states that the functor $\Phi_x$ for any odd element $x\in \g_{\bar 1}$ can be described via the Duflo-Serganova functor $DS_{x_0}$ associated with an attractor $x_0$ of $x$.

The stronger version of \cref{introthrm:2-step}, proved in \cref{thm:2-step_exist}, allows us to choose an attractor $x_0$ and a $2$-step decomposition $x=x_{bal}+x_{neat}$ so that $x_{neat}$ lies in an $\osp(1|2)$-type Lie sub-superalgebra which commutes with $x_0$ (such a Lie sub-superalgebra exists by the super Jacobson-Morozov theorem). Let us denote the group homomorphism corresponding to this sub-superalgebra by 
$\varphi_{x_{neat}}:OSp(1|2)\to G$. 

Denote by $(-)\downarrow_{\varphi_{x_{neat}}}: \Rep(G)\to \Rep(OSp(1|2)\times G')$ the restriction functor to the sub-supergroup $OSp(1|2)\times G'$, where $G'$ is the quasi-reductive subgroup of $G$ centralizing the subgroup $\Im(\varphi_{x_{neat}})$.

The following result is a corollary of \cref{prop:tensor_functors_2_step}, \cref{cor:criterion_balanced}.

\begin{introthm}\label{introthrm:functors}
Let $x\in \g_{\bar 1}$. Choose $x_0\in att(x)$ and $\varphi_{x_{neat}} $ as above.
Then we have an isomorphism of functors
making the following diagram commutative:
$$ \xymatrix{ &\Rep(G) \ar[rrrd]_{\Phi_x} \ar[rrr]^-{(-)\downarrow_{\varphi_{x_{neat}}}} &&&{\Rep(OSp(1|2)\times G')} \ar[d]^{\id\boxtimes DS_{x_{bal}}} \\
& & &&\Rep(OSp(1|2))}.$$
In particular, if $x$ is balanced then there exists a natural isomorphism $\Phi_x \xrightarrow{\sim} DS_{x_0}.$
\end{introthm}

Intuitively, this theorem means that the functor $\Phi_x$ behaves essentially like the homological functor $DS_{x_0}=\Phi_{x_0}$.

\subsection{Structure of the paper}
We give the general preliminaries in \cref{sec:prelim} and give the required background on neat elements and the super Jacobson-Morozov theorem in \cref{sec:neat_elems}. We then define attractors of an odd element in \cref{sec:attractor} and study their properties, proving the first part of \cref{introprop:neat_bal_criteria}. In \cref{sec:balanced}, we define balanced and strongly balanced elements, explain why the latter implies the former and explain why the correspondence between balanced orbits and their attractors has finite fibers (as a preparation for \cref{introthm:finite_fibers}). In \cref{sec:orbits_in_classical}, we classify balanced orbits in simple quasi-reductive Lie superalgebras, stating explicitly which are neat and which are balanced. Finally, we use this classification to prove Theorems \ref{introthrm:2-step}, \ref{introthm:finite_fibers}, \ref{introthrm:functors} in \cref{sec:2step}. Additionally, in \cref{app:detecting_neat_and_bal} we discuss how the properties of being neat or balanced are inherited in Levi sub-superalgebras. 
\subsection{Acknowledgments}
The authors thank A. Elashvili and P. Etingof for stimulating discussions. I.E.-A. was supported by the Israel Science Foundation (grant no. 1362/23). V.S. was supported by Simons Foundation grant 346300.

\section{Preliminaries}\label{sec:prelim}

Our base field will be an algebraically closed field $\kk$ with $char(\kk) = 0$.

All our categories will be $\kk$-linear rigid symmetric monoidal categories, with the bifunctor $-\otimes-$ being bilinear. All the functors will be symmetric monoidal (we will write SM for short) and $\kk$-linear.

\subsection{Semisimplification}\label{ssec:symmetric monoidal_cat}

\begin{definition}
    The {\it semisimplification} of a rigid symmetric monoidal $\kk$-linear category $\mathcal{U}$ is the pair $(S, \overline{\mathcal{U}})$, where $\overline{\mathcal{U}} = \mathcal{U} /\mathcal{N} $ and $S: \mathcal{U} \to \overline{\mathcal{U}}$ is the quotient functor. 
\end{definition} 
The category
$\overline{\mathcal{U}}$ is a semisimple rigid symmetric $\kk$-linear category and $S: \mathcal{U} \to \overline{\mathcal{U}}$ is a full symmetric monoidal $\kk$-linear functor. More details can be found in \cite{etingof2021semisimplification}.

\subsection{Vector superspaces and supergroups}\label{ssec:supervec}

A {\it vector superspace} is a $\mathbb{Z}/2\mathbb{Z}$-graded $\kk$-vector space 
$V=V_{\bar 0}\oplus V_{\bar 1}$; for $v\in V_{\eps}$, $\eps\in  
\mathbb{Z}/2\mathbb{Z}$, the {\it parity} of $v$ is $\bar{v} := \eps$. 

The objects in the category of vector superspaces $\sVect$ are 
finite-dimensional vector superspaces and the morphisms are linear 
morphisms preserving the grading. When mentioning maps $V\to W$ between two vector superspaces, we will denote by $\Hom(V, W)$ the space of grading-preserving maps and by $\Hom^{\bullet}(V, W)$ the space of all linear maps. The latter is naturally a vector superspace in its own right and $\Hom^{\bullet}(V, W)_{\bar 0}=\Hom(V, W)$. The elements of $\Hom^{\bullet}(V, W)_{\bar 1}$ will be referred to as {\it odd maps}.

When mentioning maps $V\to W$ we will by default assume that these maps are grading-preserving, unless explicitly stated otherwise.

Let $V\in \sVect$. The (categorical) trace of an endomorphism $\phi\in \End^{\bullet}(V)$, called the {\it supertrace} of $\phi$, is given by $str(\phi):=tr(\phi^{(0)})-tr(\phi^{(1)})$ where $\phi^{(p)}:V_{\bar p}\to V_{\bar p}$ is the corresponding component of the map $\phi$, for $p\in \{\bar 0, \bar 1\}$. The categorical dimension of $V \in \sVect$, also called the
{\it superdimension}, is defined as
$$\sdim V := str(\id_V)=\dim V_{\bar{0}} - 
\dim V_{\bar 1}.$$ We also denote the usual dimension by $\dim V:=\dim V_{\bar{0}} + 
\dim V_{\bar 1}$.

We will denote by $\Pi$ the change of parity endofunctor on $\sVect$: namely, $\Pi \kk^{m|n} \cong \kk^{n|m}$.

A supergroup $G$ may be defined  via the supercommutative Hopf superalgebra of functions on $G$. The Lie superalgebra $\mathrm{Lie}(G)$ corresponding to an algebraic supergroup $G$ is defined as a subquotient of this Hopf superalgebra in the usual way.

An alternative approach to algebraic supergroups is via Harish-Chandra pairs (see \cite{masuoka2012harish}). A Harish-Chandra pair is a pair $(G_{\bar 0}, ~\g)$, where $\g$ is a Lie superalgebra, $G_{\bar 0}$ is an affine algebraic group with $\mathrm{Lie} ~G_{\bar 0}=\g_{\bar 0}$ and the adjoint action of $\g_{\bar 0}$ on $\g_{\bar 1}$ integrates to an action of $G_{\bar 0}$.
The category of affine algebraic supergroups is equivalent to the category of Harish-Chandra pairs. 
Given an algebraic supergroup $G$, its corresponding Harish-Chandra pair is $(G_{\bar 0}, ~\g)$, where $\g:=\mathrm{Lie}(G)$ and $G_{\bar 0}$ is defined as the group of the $\kk$-points of $G$. The latter is called the {\it underlying algebraic group} of $G$.

\begin{definition}
 Given a Lie superalgebra $\g$, the category of finite-dimensional representations of $\g$ is denoted by $\Rep(\g)$. Its objects are $(V, \rho)$ where $V \in \sVect$ and $\rho: \g \to \gl(V)$ is a homomorphism of Lie algebra objects in $\sVect$. The maps in $\Rep(\g)$ are $\g$-equivariant, parity-preserving maps.
\end{definition}

Let $G$ be an algebraic supergroup corresponding to a Harish-Chandra pair $(G_{\bar 0}, ~\g)$.

\begin{definition}
We define $\Rep(G)$ as the category of representations of $G$ in $\sVect$, with parity-preserving maps. 
\end{definition}

In the Harish-Chandra pairs approach, the category $\Rep(G)$ can be viewed as the category of finite-dimensional Harish-Chandra $(G_{\bar 0}, ~\g)$-modules: that is, representations of $\g$ where the action of the Lie algebra $\g_{\bar 0}$ integrates to an action of $G_{\bar 0}$.

\begin{definition}
 A supergroup $G$ is called {\it quasi-reductive} if $G_{\bar{0}}$ is reductive\footnote{By a reductive algebraic group we mean an algebraic group whose finite-dimensional (rational) representations form a semisimple category.}. 
\end{definition}

\subsection{Centralizers, orbits, semisimple and homological elements}\label{ssec:centralizers}
Let $G$ be an algebraic supergroup and $\g:=\mathrm{Lie}(G)$.
\begin{comment}
\begin{notation}

We denote by $\Aut^{\circ}(\g)$ the connected component of the identity in the algebraic group of automorphisms of the Lie superalgebra $\g$ and by $$\Der(\g):=Lie(\Aut^{\circ}(\g))=\{\partial:\g\to\g ~\rvert~ \forall ~ a,b\in \g, \; \partial([a,b])=[\partial(a),b]+[a,\partial(b)]~\}$$
the Lie algebra of (even) derivations of $\g$.
The inner derivations $$\mathrm{Inn}(\g):=\{\ad_a|a\in \g_{\bar 0}\}\cong \g_{\bar 0}/Z(\g)_{\bar 0}$$ form an ideal in $\Der(\g)$.

\end{notation}
Sometimes we will be interested in the  Lie superalgebra $\Der^{\bullet}(\g)$ of all derivations of $\g$, both even and odd:
\begin{notation}
We denote:
  $$\Der^{\bullet}(\g):=\{\phi\in \End^{\bullet}(\g)~\rvert ~ \forall ~a,b\in \g, ~ \phi([a,b])=[\phi(a), b]+(-1)^{p(a)p(\phi)}[a, \phi(b)]\}. $$
  
This is a vector superspace and elements of $\Der^{\bullet}(\g)_{\bar 1}$ are called {\it odd derivations } of $\g$.
\end{notation}
Clearly, $\Der^{\bullet}(\g)_{\bar 0} = \Der(\g)$.
\begin{notation}
For any $a\in \g$, let $C(a)$ be the subgroup of $\Aut^{\circ}(\g)$ preserving $a$. Similarly, for any subset $S\subset \g$ let $C(S):=\cap_{a\in S} C(a)$ be the subgroup of $\Aut^{\circ}(\g)$ preserving each element of $S$.
\end{notation}
\end{comment}
\begin{notation}
    For any $a\in \g$, we will denote by $\mathfrak{c}_{\g_{\bar 0}}(a):=\Ker \ad_a\rvert_{\g_{\bar 0}}$ the centralizer of $a$ in $\g_{\bar 0}$ and by $G_{\bar 0}.a$ its orbit under the adjoint action of $G_{\bar 0}$. Similarly, for any subset $S\subset \g$, let $\mathfrak{c}_{\g_{\bar 0}}(S) := \cap_{a\in S} \mathfrak{c}_{\g_{\bar 0}}(a)$.

    When considering the centralizer of the $a$ (respectively, $S$) in $G_{\bar 0}$, we will denote by $C_{G_{\bar 0}}(a)$ (respectively, $C_{G_{\bar 0}}(S)$) the connected component of the unit; this is a connected algebraic subgroup of $G_{\bar 0}$ and $\mathrm{Lie} ~C_{G_{\bar 0}}(a)=\mathfrak{c}_{\g_{\bar 0}}(a)$.
\end{notation}

\begin{notation}
    We denote by $Z(\g):=\{a\in \g~\rvert~ \forall ~b\in \g, [a,b]=0\}$ the center of $\g$.
\end{notation}

\begin{proposition}\label{prop:centralizer_reductive} Let $\g$ be a quasi-reductive Lie superalgebra and $\h$ be a subalgebra which acts semisimply on $\g$ via the adjoint action. Then $\fl=\fc_{\g}(\h)$ is quasi-reductive.    
\end{proposition}
\begin{proof}
Let $\mathfrak{m}$ be the direct sum of all simple non-trivial $\h$-submodules in $\g$.
Then $\g=\fl\oplus\mathfrak{m}$ and $[\fl,\mathfrak m]\subset\mathfrak m$. Let $\mathfrak n$ denote the nil radical of $\fl_{\bar 0}$. Assume $\mathfrak n\neq 0$ and choose $x\in \mathfrak n$. Since $\g_{\bar 0}$ is reductive there exists an $\mathfrak{sl}_2$ triple $(x,h,y)$ in $\g_{\bar 0}$. Let $\bar h$ and $\bar y$ denote the images of $h, y$ under the projection $\g\twoheadrightarrow\fl$
with kernel $\mathfrak m$. Then $[\bar h,x]=2x$ and $[\bar y,x]=-\bar h$ (here we use:  $[\fl,\mathfrak m]\subset\mathfrak m$). So $\bar h$ lies in $\mathfrak n$ but is not nilpotent. We obtain a contradiction.
\end{proof}
\begin{remark}
An analogous statement holds for supergroups as well.
\end{remark}

The following fact will be used frequently throughout the paper.
\begin{lemma}\label{lem:centralizer_Jordan_closed}
Let $G$ be quasi-reductive and $\g:=\Lie(G)$.
    Let $\mathfrak{c}_{\g_{\bar 0}}(a)\subset  \g_{\bar 0}$ be the (even) centralizer of any element $x\in \g$. Then $\mathfrak{c}_{\g_{\bar 0}}(a)$ is closed under Jordan decomposition: that is, it contains the semisimple and nilpotent part of any element of $\mathfrak{c}_{\g_{\bar 0}}(a)$.
    
\end{lemma}

\begin{proof}
    Given an even element $t\in \mathfrak{c}_{\g_{\bar 0}}(a)$, consider its Jordan decomposition $t=s+n$, where $s$ is semisimple, $n$ is nilpotent. Then $\ad_s$ is a polynomial in $\ad_t$, so $x \in \Ker(\ad_t) \subset \Ker(\ad_s)$ and thus $s\in \mathfrak{c}_{\g_{\bar 0}}(a)$. Hence $\mathfrak{c}_{\g_{\bar 0}}(a)$ is closed under Jordan decomposition.
\end{proof}

\begin{notation}\label{notn:ss}
    Let $\mathfrak{r}$ be a reductive Lie algebra (purely even). We denote by $\mathfrak{r}^{\mathbf{ss}}$ the set of its semisimple elements.
\end{notation}
Note that in this setting, $Z(\mathfrak{r})\subset \mathfrak{r}^{\mathbf{ss}}$.

\begin{notation}
Assume that $G$ is quasi-reductive.
  \begin{itemize}
      \item We will denote by $\g^{\mathbf{ss}}$ the set of semisimple elements in $\g_{\bar 0}$.
      \item We will denote $\g^{hom}:=\{x\in \g_{\bar 1}~|~ [x,x]\in \g^{\mathbf{ss}}\}.$
      Such elements $x$ will be called {\it homological elements of $\g$}.
  \end{itemize} 
\end{notation}

\subsection{Simple Lie superalgebras}\label{ssec:classical_list}
The superalgebras below are examples of quasi-reductive Lie superalgebras corresponding to quasi-reductive Lie supergroups\footnote{We will write $OSp(m|2n)$ for the supergroup whose underlying algebraic group is connected; this supergroup is often denoted $SOSp(m|2n)$.}: 
%\begin{align*}  & \gl(m|n), \;\sl(m|n), \p\gl(m|n), \p\sl(n|n), \;\osp(m|2n),\; D(2|1; a),\; F(1,3), \;G(1,2), \\&\mathfrak{pe}(n), \;\mathfrak{spe}(n), \;\mathfrak{q}(n),\;\mathfrak{pq}(n),\;\mathfrak{sq}(n),\;\mathfrak{psq}(n)
%\end{align*}

%We remind the reader briefly the definitions of some of these Lie superalgebras:
\begin{itemize}
\item $\gl(m|n)=\End^{\bullet}(\kk^{m|n})$ with $[A', A'']=AA'-(-1)^{\overline{A'}~\overline{A''}}A''A'$ for homogeneous $A', A''\in \End^{\bullet}(\kk^{m|n})$. We define $\sl(m|n)=\{A\in \gl(m|n)~|~str(A)=0\}$.
\item Let $B\in \Hom^{\bullet}(S^2\C^{m|n'}, \kk)$ be a non-degenerate (homogeneous) symmetric form on $\C^{m|n}$. Let $\g$ be the Lie sub-superalgebra of $\gl(m|n)$ preserving this form. If $\overline{B}=\bar{0}$ then $n'= 2n\in 2\Z$ and $\g$ is called the orthosymplectic Lie superalgebra $\osp(m|2n)$. If $\overline{B}=\bar{1}$ then $n'= m$ and $\g$ is called the periplectic Lie superalgebra $\mathfrak{pe}(m|m)$. 
\item  $\q(n)\subset \gl(n|n)$ is the Lie sub-superalgebra centralizing the parity-switching isomorphism $ \kk^{n|n}\to \kk^{n|n}$. The Lie sub-superalgebra $\mathfrak{sq}(n)\subset \q(n)$ has $\mathfrak{sq}(n)_{\bar 0} =\gl_n$,  $\mathfrak{sq}(n)_{\bar 1} =\sl_n$.
    \item $\p\gl(m|n)$, $\p\sl(n|n)$, $\mathfrak{pq}(n)$ and $\mathfrak{psq}(n)$ are the respective quotients of $\gl(m|n)$, $\sl(n|n)$, $\mathfrak{q}(n)$ and $\mathfrak{sq}(n)$ by their $1$-dimensional centers and $\mathfrak{spe}(n) :=\mathfrak{pe}(n)\cap \sl(n|n)$. 
\end{itemize}

The classification of simple quasi-reductive Lie superalgebras can be derived from \cite{kac1977lie} and \cite{serganova2011quasireductive}. These are:
\begin{itemize}
\item Simple finite-dimensional Lie algebras,
    \item $\sl(m|n)$ (for $m\neq n$, $m, n\in\Z_{\geq 1}$), $\p\sl(n|n)$ ($n\in\Z_{\geq 1}$), $\osp(m|2n)$ ($m, n\in\Z_{\geq 1}$),
    \item Exceptional Lie superalgebras $D(2|1; a),\; F(1,3), \;G(1,2)$ with $a\neq 0$ (see  \cite{kac1977lie} for explicit constructions),
    \item Strange Lie superalgebras $\mathfrak{spe}(n)$, $\mathfrak{psq}(n)$ ($n\in\Z_{\geq 1}$).
\end{itemize}

We will also consider Takiff superalgebras, which we describe briefly below.  

Recall that for a Lie superalgebra $\g$, an odd derivation of $\g$ is a map 
  $$\phi:\g\to \Pi \g~ ~\text{ so that }  ~~ \forall ~a,b\in \g, ~ \phi([a,b])=[\phi(a), b]-(-1)^{\bar{a}}[a, \phi(b)]. $$

 Let $\fl$ be a simple Lie algebra and consider the Lie superalgebra $\g:=\fl\oplus\Pi\fl$, where the odd part is an abelian ideal given by the adjoint $\fl$-module. We may identify $\g\cong \fl\otimes \kk \xi$ where $\xi$ is an odd element and $[\xi, \xi]=0$. Let $\partial:=\frac{\partial}{\partial \xi}$; this is an odd derivation of $\g$. Consider the Lie superalgebra $\g\rtimes \kk \partial$, with $[\partial, \partial]=0$ and $[\partial, u]=\partial(u)$ for any $u\in \g$. Then $\g\rtimes \kk \partial$ is called the Takiff Lie superalgebra.  

\subsection{Quasi-reductive Lie superalgebras}\label{ssec:quasired}
Below we give a short summary of the results by Serganova on the structure of quasi-reductive Lie superalgebras (see \cite[Theorem 6.9]{serganova2011quasireductive}). 
Let $G$ be a quasi-reductive algebraic supergroup such that $G_{\bar 0}$ is connected and let $\g:=\mathrm{Lie} (G)$. 

\begin{lemma}\label{lem:aux_odd_central_commutator}
We have: $Z(\g)_{\bar 1} \cap [\g, \g] =\{0\} $. In particular, there exists a splitting of Lie superalgebras $\g=\widetilde{\g}\times Z(\g)_{\bar 1}$.
\end{lemma}
\begin{proof}
The supergroup $G$ is quasi-reductive, so the reductive Lie algebra $\g_{\bar 0}$ acts semisimply on $\g_{\bar 1}$. This implies that we have a $\g_{\bar 0}$-decomposition 
$$\g_{\bar 1}= Z(\g)_{\bar 1}\oplus \mathfrak{s}$$
for some subspace $\mathfrak{s} \subseteq \g_{\bar 1}$.
Thus $[\g_{\bar 0}, \g_{\bar 1}]\subseteq \mathfrak{s}$ and so $[\g, \g]_{\bar 1} \subseteq \mathfrak{s}$. We conclude that $ Z(\g)_{\bar 1} \cap [\g, \g]_{\bar 1}  =\{0\}$.
\end{proof} 

Let $\g$ be a quasireductive Lie superalgebra. We will denote $\g':=\g/Z(\g)$ and let
$\mathfrak{i}(\g)$ denote the sum of all the minimal ideals of $\g'$. Let $\mathfrak{r} := \g'/\mathfrak{i}(\g)$.
Then \cite[Theorem 6.9]{serganova2011quasireductive} states that $\g'$ is a semidirect product of $\mathfrak{i}(\g)$ and $\mathfrak{r}$ and that 
\begin{itemize}
    \item The ideal $\mathfrak{i}(\g)$ is a direct sum of ideals and each of these summands $\ft$ satisfies one of the following conditions:
    \begin{itemize}[label=$\star$]
        \item $\ft$ is a simple quasi-reductive Lie superalgebra.
        \item $\ft_{\bar 1}$ is an abelian ideal in $\ft$ and $\ft_{\bar 0}$ is either a simple Lie algebra (in which case $\ft_{\bar 1}$ is its adjoint module\footnote{In that case, \InnaA{there exists a Takiff sub-superalgebra of $\g'$ of the form $\ft\rtimes \kk \partial$}.}) or $\ft_{\bar 0}=0$.
    \end{itemize}
    \item The Lie superalgebra $\mathfrak{r} = \mathfrak{r}_{\bar 0}\oplus \mathfrak{r}_{\bar 1}$ has a reductive even part $\mathfrak{r}_{\bar 0}$ and $\mathfrak{r}_{\bar 1}$ is an abelian ideal with a trivial action of $[\mathfrak{r}_{\bar 0}, \mathfrak{r}_{\bar 0}]$.
    \item We have a decomposition of $\g'_{\bar 0}$-modules $\g'= \mathfrak{i}(\g) \oplus \mathfrak{r} $. In particular, $\mathfrak{r}_{\bar 0}$ is an ideal in $\g'_{\bar 0}$ and we have a decomposition $\g'_{\bar 0}= \mathfrak{i}(\g)_{\bar 0} \times \mathfrak{r}_{\bar 0} $ as a product of Lie algebras.
    \item The elements of $\mathfrak{r}_{\bar 1}$ in the decomposition $\g'= \mathfrak{i}(\g) \oplus \mathfrak{r} $ act by odd derivations on $\mathfrak{i}(\g) $; the restriction of these derivations to $\mathfrak{i}(\g)_{\bar 0}$ is zero. In particular, $[\mathfrak{r}_{\bar 1}, \mathfrak{r}_{\bar 1}]=0$.
\end{itemize}

We will denote by $\overline{\mathfrak{i}}(\g)$ the ideal (direct summand) of the Lie superalgebra $\mathfrak{i}(\g)$ which is the direct sum of those ideals $\fl\subseteq \mathfrak{i}(\g)$ which are simple Lie superalgebras.  Let $\overline{I}(G) \trianglelefteq G'$ be the connected sub-supergroup corresponding to the ideal $\overline{\mathfrak{i}}(\g)$. 

\subsection{Levi and parabolic sub-superalgebras}\label{ssec:levi}

\begin{definition}\label{def:levi_subalgebra}
Let $\g$ be a Lie superalgebra. A sub-superalgebra $\fl\subseteq \g$ is called a {\it Levi} subalgebra of $\g$ if there exists a commutative Lie subalgebra $\s\subset\g^{\mathbf{ss}}$ such that $\fl = \g^\s$. It is clear that $\fl$  is a quasi-reductive subalgebra in $\g$.

\end{definition}

\begin{remark}
As in \cite{collingwood1993nilpotent}, the subalgebra $\s$ in the definition above may be replaced by a single element $s\in \g^{\mathbf{ss}}$. 
\end{remark}
\begin{comment}
\begin{remark}
If $\g=\g_{\bar 0}$ is a semisimple Lie algebra, this definition of a Levi subalgebra coincides with the one in \cite{collingwood1993nilpotent}: it is the centralizer of some semisimple element of $\g$. However, if $\g=\g_{\bar 0}$ is a reductive Lie algebra with a non-trivial center, then its semisimple ideal $[\g, \g]$ will, in this definition, be considered a proper Levi subalgebra.
\end{remark}
\end{comment}

   Levi subalgebras of quasi-reductive Lie superalgebras are particularly nice, being also quasi-reductive. 
   \begin{lemma}
       Let $G$ be a quasi-reductive supergroup and $\g:=\mathrm{Lie} (G)$. Any Levi subalgebra of $\g$ is itself the Lie superalgebra of a quasi-reductive sub-supergroup of $G$.
   \end{lemma}
   \begin{proof}
Let $s\in \g^{\mathbf{ss}}$ and let $\fl:=\g^s$. Let $T$ be a minimal toral subgroup in $G_{\bar 0}$ such that $\operatorname{Lie}(T)$ contains $s$. 

The centralizer $C_{G_{\bar 0}}(T)$ of $T$ in $G_{\bar 0}$ is reductive and so is its connected component of the unit, which we denote by $L_0$. Clearly, $L_0=C_{G_{\bar 0}}(s)$ and $Lie(L_0)= \mathfrak{c}_{\g_{\bar 0}}(s)=\fl_{\bar 0}$.
 The pair $(\fl,L_0)$ is a Harish-Chandra pair, defining a quasi-reductive sub-supergroup $L\subset G$.

   \end{proof}

\begin{definition}\label{def:parabolic_subalgebra}
Let $\g$ be a Lie superalgebra. A sub-superalgebra $\p\subseteq \g$ is called a {\it parabolic} subalgebra of $\g$ if there exists $s\in \g^{\mathbf{ss}}$ defining a grading $\g=\bigoplus_{i\in \Z} \g^{i}$ such that $\p=\bigoplus_{i\geq 0} \g^{i}$. 

The sub-superalgebra $\g^{0}$ is then called the {\it Levi subalgebra} of $\p$. Any Levi subalgebra in $\g$ is the Levi subalgebra of some parabolic subalgebra $\p\subset \g$.
\end{definition}

\subsection{The \texorpdfstring{supergroups $OSp(1|2)$}{group OSp(1|2)} and \texorpdfstring{$\G^{(1|1)}$}{the odd additive supergroup}}\label{ssec:prelim_osp}
\subsubsection{The supergroup \texorpdfstring{$OSp(1|2)$}{OSp(1|2)}}

Let $V_k$ be the $(k+1)$-dimensional irreducible representation of $SL_2$ (we set $V_{-1}:={0}$).

The supergroup $OSp(1|2)$ is the group superscheme in $\sVect$ of 
automorphisms of the space $\kk^{1|2}$ respecting a fixed symmetric 
non-degenerate form $\kk^{1|2} \otimes \kk^{1|2} \to \kk$ and having Berezinian 
$1$ (so $OSp(1|2)$ is a connected group in our convention).

We have: $OSp(1|2)_{\bar 0} = SL_2$ and the Lie superalgebra $\osp(1|2) = \mathrm{Lie}(OSp(1|2))$ is $(3|2)$-dimensional. The even part $\osp(1|2)_{\bar 0}$ is isomorphic to $\mathfrak{sl}_2$ and the odd part $\osp(1|2)_{\bar 1}$ (as a module over $\osp(1|2)_{\bar 0}$) is isomorphic to the standard $2$-dimensional representation of $\mathfrak{sl}_2$. 

We denote by $h$ the generator of the Cartan subalgebra in $\osp(1|2)_{\bar 0}$ and by $X, Y$ the standard basis of $\osp(1|2)_{\bar 1}$. The elements $h, X, Y$ generate the superalgebra $\osp(1|2)$, with relations
$$[h, X]=-X,\, [h, Y]=Y, \, [Y, X]=2h.$$

Taking $f:=[X, X]$ and $e:=[Y, Y]$, we obtain a basis $(X, Y, e,f,h)$ of $\osp(1|2)$.

We denote by $\Rep(OSp(1|2))$ the category of finite-dimensional $OSp(1|2)$-modules. This category is semisimple, with isomorphism classes of simple objects (up to parity switch) numbered by even integers. We denote by 
$\MM(k)$ ($k\geq 0$) the $(k+1|k)$-dimensional irreducible representation such that $$ \MM(k) \downarrow^{OSp(1|2)}_{SL_2} \;\cong \;V_{k} \oplus \Pi V_{k-1}.$$
%and $\osp(1|2)_{\bar{1}}$ acts by odd morphisms accordingly.

\begin{example}
 We have: $\MM(0) = \triv$, $\Pi \MM(1)$ is the natural $(1|2)$-dimensional matrix representation of $\osp(1|2)$ and $\MM(2)$ is the adjoint representation of $\osp(1|2)$.
\end{example}
The Clebsh-Gordan coefficients for $\osp(1|2)$ were computed in \cite[Lemma 3.3.4]{entova2022jacobson}.
\begin{comment}
It is easy to compute the Clebsh-Gordan coefficients for $\osp(1|2)$ using the restriction to the even part $\sl_2 = \osp(1|2)_{\bar 0}$ (see \cite[Lemma 3.3.4]{entova2022jacobson}):

\begin{lemma}\label{lem:Clebsh_Gordan}
 In the category of $\osp(1|2)$-modules, the decomposition of tensor products of irreducibles into irreducible summands is given by

$$ \MM(k)\otimes \MM(m) \cong \bigoplus_{s=\abs{k-m}}^{k+m} \Pi^{k+m-s} \MM(s)
$$
\end{lemma}
\end{comment}
\subsubsection{The \texorpdfstring{supergroup $\G^{(1|1)}$}{odd additive supergroup}}\label{ssec:prelim_G_a_odd}

Let $\G^{(1|1)}$ be the $(1|1)$-dimensional additive algebraic supergroup with 
$$\left(\G^{(1|1)}\right)_{\bar{0}} = \G.$$ The corresponding 
$(1|1)$-dimensional Lie superalgebra $\g_a^{(1|1)}:=\Lie(\G^{(1|1)})$ has a basis $([x,x], x)$, where $x \neq 0$ is 
an odd element, with relation (the Jacobi identity) $[x,[x,x]]=0$. This defines 
a Harish-Chandra pair as in \cite{masuoka2012harish} and hence an algebraic supergroup. 
For any algebraic supergroup $G$, an odd $\ad$-nilpotent element $x\in \mathrm{Lie} (G)_{\bar 1}$ defines a homomorphism $\G^{(1|1)} \to G$. We will denote the image of $\g_a^{(1|1)}$ in $ \g $ by $\kk[x]$.

Let $\mathtt{M}_k$ be the indecomposable representation of $\G^{(1|1)}$ with a basis $a_0, a_1, \ldots, a_{k}$, such that $\overline{a_{j}} \equiv j \mod 2$ for any $j\geq 0$ and
$$xa_j = 
\begin{cases}
	a_{j+1} &\text{ if } j<k\\
	0 &\text{ if } j=k                                                                                                                                                                                   
\end{cases}.
$$ 
It is easy to see that these are all the indecomposable representations of $\G^{(1|1)}$ up to change of parity and isomorphisms. The tensor multiplication rules for the modules $\mathtt{M}_k$ as well as the following easy (yet important) result were proved in \cite{entova2022jacobson}: there exists a symmetric monoidal (non-exact) functor $S:\Rep(\G^{(1|1)})\to \Rep(OSp(1|2))$ making $\Rep(OSp(1|2))$ the semisimplification of the category $\Rep(\G^{(1|1)})$. This functor annihilates indecomposable $\G^{(1|1)}$-representations of even dimension; more explicitly,  
$S(\mathtt{M}_{2k})\cong\mathbf{M}(2k)$, $S(\mathtt{M}_{2k+1})=0$.

Let $X \in \osp(1|2)_{\bar 1}$ be such that the element $X^2=\frac{1}{2}[X,X] \in \osp(1|2)_{\bar 0}$ corresponds to 
$f:=\begin{pmatrix}
0 &0\\
1 &0                                                                                                       \end{pmatrix}
$ under the isomorphism $\osp(1|2)_{\bar 0} \cong \mathfrak{sl}_2$.
The element $X$ defines an embedding $\iota_X:\G^{(1|1)} \hookrightarrow OSp(1|2)$, which in turn induces a restriction functor $$(-)\downarrow_X~: \Rep(OSp(1|2)) \to \Rep(\G^{(1|1)}).$$ 
It is easy to see that $S\circ (-)\downarrow_X~\cong ~\id$, so $\MM(2k)\downarrow_X ~\cong~ \mathtt{M}_{2k}$ for any $k\geq 0$.

\subsubsection{Representations of the \texorpdfstring{Lie superalgebra $\g_a^{(1|1)}$}{(1|1)-dimensional superalgebra}}
We may also consider the category of all finite-dimensional representations of the Lie superalgebra $\g_a^{(1|1)}$. 

Such representations are given by a pair $(x\rvert_V, V)$ where $V$ is a finite-dimensional vector superspace and $x\rvert_V\in \End^{\bullet}(V)_{\bar 1}$ is an odd endomorphism (not necessarily nilpotent). Let us explain what the semisimplification of this category looks like.

The category $\Rep(\g_a^{(1|1)})$ contains $ \Rep(\G^{(1|1)}) $ as a full subcategory, whose objects are pairs $(x\rvert_V,V)$ where $x\rvert_V$ is a nilpotent operator.

An indecomposable object in $\Rep(\g_a^{(1|1)})$ is a pair $(x\rvert_V, V)$ where $[x\rvert_V,x\rvert_V]$ acts on each subspace $V_{\bar 0}, V_{\bar 1}$ by a single Jordan block, with the same eigenvalues. The operator $[x\rvert_V,x\rvert_V]$ has supertrace zero, so either $x\rvert_V$ is a nilpotent operator, or $\dim V_{\bar 0}=\dim V_{\bar 1}$ (so $\sdim V=0$.) 

We now consider the semisimplification of the category $ \Rep(\g_a^{(1|1)})$. The above considerations imply that we have a semisimplification functor $S:\Rep(\g_a^{(1|1)})\to \Rep(OSp(1|2))$ which fits into the commutative diagram  
$$\xymatrix{&\Rep(\G^{(1|1)}) \ar[rd]^S \ar@{^{(}->}[r] &\Rep(\g_a^{(1|1)}) \ar^S[d]\\ &{} &\Rep(OSp(1|2))}$$

\subsection{The Deligne filtration}\label{ssec:filtration}

Let $x$ be an odd nilpotent operator on a finite-dimensional superspace $M$. Then $x$ defines a canonical finite increasing filtration\footnote{In the case of an even operator $x$, this is the filtration which appears in the Hodge theory.}
$$\ldots\subset\mathcal F^{i}(M)\subset \mathcal F^{i+1}(M)\subset\ldots $$ satisfying the conditions
\begin{itemize}
  \item $x(\mathcal F^{i}(M))\subset\mathcal F^{i-2}(M)$;
  \item If $Gr^i(M):=\mathcal F^{i}(M)/\mathcal F^{i-1}(M)$ then $x^{{i}}: Gr^{i}(M)\to \Pi^iGr^{-i}(M) $ is an isomorphism for all $i\geq 0$.
  \end{itemize}

This filtration is explicitly given by $$\mathcal{F}_k(M)=\bigoplus_{j-i=k+1} \Ker x^j\rvert_M \cap \Im x^i\rvert_M.$$

Choose the standard set of generators $h, X, Y$ in  $\mathfrak{osp}(1|2)$ as in \cref{ssec:prelim_osp}.

\begin{lemma}\label{lem:ospstr}[See \cite{entova2022jacobson}]
For any vector superspace $M$ with an odd nilpotent endomorphism $x$ on $M$, the superspace $Gr^{ev}(M):=\bigoplus_{i\in\mathbb Z}Gr^{2i}(M)$ has a unique structure of $\mathfrak{osp}(1|2)$-module such that $h$ acts by grading and $X$ acts as $Gr(x)$.
\end{lemma}
\InnaA{Thus we have a functor $Gr^{ev}:\Rep(\G^{(1|1)}) \to \Rep(OSp(1|2))$. This functor is isomorphic to the semisimplification functor $S$.}

The following straightforward lemma connects the negligible morphisms in $\Rep(\G^{(1|1)})$ with the Deligne filtration.
\begin{lemma}\label{lem:negl_Deligne_filt}
    Let $f: M\to M'$ be a morphism of $\G^{(1|1)}$-modules and let $\mathcal{F}^{\bullet}(M), \mathcal{F}^{\bullet}(M')$ be the corresponding Deligne filtrations.
    \begin{enumerate}
        \item If for all $i\in \Z$ we have $f(\mathcal{F}^i(M))\subset \mathcal{F}^{i-1}(M')$, then $f$ is negligible.
        \item If $M, M'$ are indecomposable $\G^{(1|1)}$-modules of odd dimension and $f$ is negligible, then for all $i\in \Z$ we have $f(\mathcal{F}^i(M))\subset \mathcal{F}^{i-1}(M')$.
    \end{enumerate}
\end{lemma}

\subsection{Symmetric monoidal functors for arbitrary odd elements and the Duflo-Serganova functors}\label{ssec:DS_functors_prelim}

Let $G$ be a quasi-reductive supergroup, $\g:=\mathrm{Lie}~(G)$.

Any $x\in 
\g_{\bar 1}$ induces a homomorphism $ i_x:\g_a^{(1|1)}\to \g$. Define $$\Phi_x=S\circ (-)\downarrow_{i_x}:\mathrm{Rep}(\g)\to \mathrm{Rep}(OSp(1|2))$$ to be the composition of the restriction functor $(-)\downarrow_{i_x}:\Rep(\g)\to \Rep(\g_a^{(1|1)})$ and the semisimplification functor $S:\Rep(\g_a^{(1|1)})\to \Rep(OSp(1|2))$. We will usually consider the functor $\Phi_x$ as a functor from $\Rep(G) \to \mathrm{Rep}(OSp(1|2))$.

We will now give another description of $\Phi_x:\Rep(G) \to \mathrm{Rep}(OSp(1|2))$. Let $s\in \g_{\bar 0}$ be the semisimple part in the Jordan decomposition of $\frac{1}{2}[x,x]$ and let the supergroup $G^s$ be the centralizer of $s$ in $G$.

Let $M\in \Rep(G)$. Clearly, $x$ acts as a nilpotent operator on the space of invariants $M^s:=\Ker(s\rvert_{M})$. We may consider the associated Deligne filtration on $M^s$, which defines a symmetric monoidal functor
$$Gr^{ev}((\cdot)^s):\Rep(G)\to \Rep(OSp(1|2)),\;\;\; M\mapsto Gr^{ev}(M^s).$$ This functor is the composition of the following functors:
\begin{itemize}
    \item The functor $\Rep(G)\to \Rep(\G^{(1|1)})$ sending $M\in \Rep(G)$ to the $\G^{(1|1)}$-module given by action of $x$ on the $s$-invariants $M^s$, 
    \item the functor $Gr^{ev}:\Rep(\G^{(1|1)})\to \Rep(OSp(1|2))$, which is isomorphic to the semisimplification functor.
\end{itemize}
The above composition is isomorphic to the functor $\Phi_x$.

One may also consider it as a functor $\Phi_x: \Rep(G)~\longrightarrow~ \Rep(\Phi_x(\g))$ where $\Phi_x(\g)$ is a new Lie superalgebra, a semidirect product of $\osp(1|2)$ and a subquotient of $\g$.

We now describe an important special case of this construction. Assume that $x\in \g^{hom}$ (that is, $s:=\frac{1}{2}[x,x]$ is semisimple).
In this case, the functor $\Phi_x$ is called the Duflo-Serganova functor and denoted by $DS_x $ (see \cite{duflo2005associated, gorelik2022duflo}). It is given explicitly by $$DS_x: \Rep(G)~\longrightarrow~ \Rep(DS_x(\g)),\;\;\; M ~\longmapsto~ \frac{\Ker x\rvert_{M^s}}{\Im x\rvert_{M^s}} = \Phi_x(M^s).$$ 
(note that $Gr^{ev}(M^s)=Gr^0(M^s)$ in this case and the action of $OSp(1|2)$ on $\Phi_x(M)$ is trivial).
\begin{remark}
    The functor $\Phi_x$ is symmetric monoidal, but it is often not exact on either side (the elements $x$ for which it is exact are discussed in \cref{ssec:neat_def}).
\end{remark}

\subsection{Restriction functors for odd additive supergroups}

Let $\iota_X:\G^{(1|1)} \to  OSp(1|2)$ be the embedding we chose before and let $$(-)\downarrow_X~: \Rep(OSp(1|2)) \to \Rep(\G^{(1|1)})$$ the corresponding restriction functor. Consider the following commutative diagram of supergroups, where $\Delta$ stands for the diagonal embeddings:
$$\xymatrix{
 &\G^{(1|1)} \ar_{(i\times \id)\circ \Delta}[d] \ar^{i}[rr] &{} & OSp(1|2) \ar_{\Delta}[d]\\
    &OSp(1|2)\times \G^{(1|1)} \ar^{\id \times i}[rr] &{} &OSp(1|2) \times OSp(1|2)
}.$$

Consider the SM functors 
\begin{align*}
    T: \Rep(OSp(1|2))\boxtimes \Rep(\G^{(1|1)}) ~\longrightarrow~\Rep(\G^{(1|1)}), \;\;\; M\boxtimes M'\longmapsto M\downarrow_X~\otimes~ M',\\
    \overline{T}: \Rep(OSp(1|2))\boxtimes \Rep(OSp(1|2)) ~\longrightarrow~\Rep(OSp(1|2)), \;\;\; M\boxtimes M'\longmapsto M\otimes M',
\end{align*} The functors $T, \overline{T}$ correspond to the supergroup homomorphisms $$\G^{(1|1)} \xrightarrow{(i\times \id)\circ \Delta} OSp(1|2)\times \G^{(1|1)}, \;\;\; \Delta: OSp(1|2)\to OSp(1|2) \times OSp(1|2)$$ respectively. The following lemma gives a "semisimplification-categorical" analogue of the above commutative diagram.
\begin{lemma}\label{lem:aux_ss_functors}
    There exists a natural isomorphism making the following diagram of functors commutative: 
    $$\xymatrix{
        &\Rep(OSp(1|2))\boxtimes \Rep(\G^{(1|1)}) \ar_{T}[d] \ar^-{\id \boxtimes S}[rr] &{} &\Rep(OSp(1|2)) \boxtimes \Rep(OSp(1|2))\ar_{\overline{T}}[d]\\
 &\Rep(\G^{(1|1)})  \ar^-{S}[rr] &{} & \Rep(OSp(1|2)). }$$
\end{lemma}
\begin{proof}
    Denote: $\mathcal{U}:=\Rep(OSp(1|2))\boxtimes \Rep(\G^{(1|1)})$. Recall that the category $\Rep(OSp(1|2))$ is semisimple, so the functor $\mathcal{U} \xrightarrow{~\id \boxtimes S~} \Rep(OSp(1|2)) \boxtimes \Rep(OSp(1|2))$ is the semisimplification functor. This functor is given by taking the quotient of the category $\mathcal{U}$ by the ideal of negligible morphisms (see \cite{etingof2021semisimplification}). We will now show that any negligible morphism in $\mathcal{U}$ is sent to zero under the functor $S\circ T$; in other words, that $T$ sends negligible morphisms to negligible morphisms. This would imply the existence of a SM functor $\overline{T'}$ and a natural isomorphism making the diagram below commutative: 
    $$\xymatrix{
        &\mathcal{U} \ar_{T}[d] \ar^-{\id \boxtimes S}[rr] &{} &\Rep(OSp(1|2)) \boxtimes \Rep(OSp(1|2))\ar_{\overline{T'}}[d]\\
 &\Rep(\G^{(1|1)})  \ar^-{S}[rr] &{} & \Rep(OSp(1|2)). }$$ 

To show that $T$ sends negligible morphisms to negligible morphisms, it is enough to check this for negligible morphisms in $\mathcal{U}$ whose domain and target are indecomposable (see \cite[Exercise 3(ii), Subsection 2.18]{benson1984modular}, or \cite[Lemma 2.2]{etingof2021semisimplification}). Let  $f:M\to M'$ in $\mathcal{U}$ be a negligible morphism between two indecomposable $OSp(1|2)\times \Rep(\G^{(1|1)})$-modules. Since $f$ is negligible, at least one of the following two conditions holds (see {\it loc. cit.}): 
\begin{itemize}
    \item $f$ is not an isomorphism, 
    \item $\sdim M=0$.
\end{itemize}

Since $OSp(1|2)$ is semisimple, any indecomposable $OSp(1|2)\times \Rep(\G^{(1|1)})$-module is an external product of a simple $OSp(1|2)$-module and an indecomposable $\Rep(\G^{(1|1)})$-module. 

Let us first assume that $\sdim M=0$. Then $M\cong \MM(2k)\boxtimes\mathtt{M}_{2s+1}$ for some $k, s\in \Z_{\geq 0}$. By the definition of the functor $T$, we have an isomorphism of $\G^{(1|1)}$-modules $$T(M)\cong \mathtt{M}_{2k}\otimes \mathtt{M}_{2s+1}.$$ Applying the functor $S$, we get:
$$ S(T(M))\cong S\left(\mathtt{M}_{2k}\otimes \mathtt{M}_{2s+1}\right)\cong S(\mathtt{M}_{2k})\otimes S(\mathtt{M}_{2s+1})=S(\mathtt{M}_{2k})\otimes0=0.$$
This shows that $S(T(f))=0$ and $T(f)$ is negligible, as required.

Now let us assume that $f$ is not an isomorphism and that $\sdim M, \sdim M'\neq 0$. Then $M\cong \MM(2k)\boxtimes\mathtt{M}_{2s}$, $M'\cong \MM(2k')\boxtimes\mathtt{M}_{2s'}$, for some $k, k', s, s'\in \Z$. 

Under these isomorphisms $f$ corresponds to an element of $\Hom_{OSp(1|2)}(\MM(2k), \MM(2k'))\otimes \Hom_{\G^{(1|1)}}(\mathtt{M}_{2s}, \mathtt{M}_{2s'})$. Since $\dim \Hom_{OSp(1|2)}(\MM(2k), \MM(2k'))\leq 1$, we may assume that $f$ corresponds to some 
$f_{1}\otimes f_{2}$, where $f_1\in \Hom_{OSp(1|2)}(\MM(2k), \MM(2k'))$, $f_2\in \Hom_{\G^{(1|1)}}(\mathtt{M}_{2s}, \mathtt{M}_{2s'})$. Since $f$ is not an isomorphism, either $f_1=0$ (in which case $f=0$ and we are done), or $f_1$ is an isomorphism and $f_2$ is a negligible morphism.

By \cref{lem:negl_Deligne_filt}, we conclude that for any $i$, we have: $f_2(\mathcal{F}^i(\mathtt{M}_{2s}))\subset \mathcal{F}^{i-1}(\mathtt{M}_{2s})$. We conclude that
$T(f)$ sends the component $\mathcal{F}^i(M)\cong \sum_j\mathcal{F}^j(\MM(2k))\otimes  \mathcal{F}^{i-j}(\mathtt{M}_{2s})$ to the component $\mathcal{F}^{i-1}(M')\cong \sum_j\mathcal{F}^j(\MM(2k))\otimes  \mathcal{F}^{i-j-1}(\mathtt{M}_{2s})$ and \cref{lem:negl_Deligne_filt} implies that $T(f)$ is negligible.

This proves the existence of a SM functor $\overline{T'}$ and an isomorphism as above.

Any SM functor $\Rep(OSp(1|2)) \boxtimes \Rep(OSp(1|2)) \to \Rep(OSp(1|2))$ corresponds to some supergroup homomorphism $OSp(1|2)\to OSp(1|2) \times OSp(1|2) $. The explicit action of $\overline{T'}$ on the objects implies that $\overline{T'}$ is isomorphic to $\overline{T}$. This completes the proof of the lemma.
\end{proof}

\newpage

\section{Neat and ad-nilpotent elements in Lie superalgebras}\label{sec:neat_elems}

\subsection{Definitions}\label{ssec:neat_def}

Let $G$ be a quasi-reductive algebraic supergroup corresponding to a Harish-Chandra pair $(G_{\bar 0}, \g:=\mathrm{Lie} (G))$ (see \cref{ssec:supervec}).

The definition below was given in \cite{entova2022jacobson}.
\begin{definition}\label{def:neat_operator}
 Let $V$ be a vector superspace and $x \in \End(V)$ an odd nilpotent operator. The element $x$ defines an action of the supergroup $\mathbb{G}_a^{(1|1)}$ on $V$. The element $x$ acts {\it neatly} on $V$ if all indecomposable $\mathbb{G}_a^{(1|1)}$-summands of $V$ have non-zero superdimension.
\end{definition}

In other words, $x$ acts neatly on $V$ if and only if $V$ and $Gr^{ev}(V)$ (see \cref{ssec:filtration}) are isomorphic as vector superspaces.

\begin{definition}\label{def:nilpotent_elements}
   Let $x\in \g_{\bar 1}$. 
   \begin{itemize}
   \item We denote by $G_{\bar 0}.x$ the $G_{\bar 0}$-orbit of $x$.
       \item We will say that $x$ is {\it $\ad$-nilpotent} if $[x,x]$ is a nilpotent element in the reductive Lie algebra $\g_{\bar 0}$ (that is, $[x,x]\in [\g_{\bar 0}, \g_{\bar 0}]$ and $\ad_x\in \End(\g)$ is a nilpotent operator).
        \item We will say that an $\ad$-nilpotent element $x\in\g_{\bar 1}$ is {\it neat} (respectively, that $G_{\bar 0}.x$ is a {\it neat orbit}) if $x$ acts 
neatly in every finite-dimensional representation of $\g$ (see \cref{def:neat_operator}).
   \end{itemize}

   We denote by $\mathcal{N}^{\ad}(\g_{\bar 1})$ the set of all $\ad$-nilpotent elements on $\g_{\bar 1}$ and by $\g_{neat}$ the set of neat elements in $\g_{\bar 1}$. 
\end{definition}

\begin{remark}
 The element $0 \in \g_{\bar 1}$ is always neat.
\end{remark}
\begin{remark}

One may also consider the set of ``geometrically nilpotent'' elements $x$ for which $0\in \overline{G_{\bar 0}.x}$.
This is the set of elements $v\in \g_{\bar 1}$ annihilated by all the $G_{\bar 0}$-invariant homogeneous polynomials of positive degree on $\g_{\bar 1}$ (see for example \cite{gruson2010cones, jenkins2021nilpotent, motorin2023resolution}). This set is clearly contained in $ \mathcal{N}^{\ad}(\g_{\bar 1})$ but does not always coincide with $ \mathcal{N}^{\ad}(\g_{\bar 1})$; for example, when $\g=\sl(n|n)$, there are $\ad$-nilpotent elements which are not ``geometrically nilpotent''.
\end{remark}

The following theorem, proved in \cite[Theorem 4.2.1]{entova2022jacobson}, justifies our interest in neat elements:

\begin{theorem}[Super Jacobson-Morozov Lemma, see \cite{entova2022jacobson}]\label{thrm:JM}
Let $G$ be a quasi-reductive algebraic supergroup with Lie algebra $\g$. 

\begin{itemize}
\item Let $x\in\mathcal{N}^{\ad}(\g_{\bar 1})$ and let $V$ be a faithful representation of $G$, inducing a representation $\rho: \g \to \gl(V)$. Then $x\in \g_{neat}$ iff $\rho(x)\in \gl(V)_{neat}$.
    \item Let $x\in\g_{neat}$, $x\neq 0$ and let $i_x:\mathbb{G}_a^{(1|1)} \hookrightarrow G$ be the corresponding injective homomorphism. The inclusion $i_x$ can be extended to an injective homomorphism $\bar{i}_x: OSp(1|2) \hookrightarrow G$. This extension is unique up to conjugation by an element of $G_{\bar 0}$.
    \item Let $x\in\mathcal{N}^{\ad}(\g_{\bar 1})$, $x\neq 0$. The functor $\Phi_x$ is faithful (equivalently, exact) if an only if $x$ is neat; in that case, $\Phi_x$ is given by the restriction with respect to $\bar{i}$.
\end{itemize}
\end{theorem}
\subsection{The \texorpdfstring{$\osp(1|2)$}{osp} subalgebra associated to a neat element}\label{ssec:neat_notn}

Let $G$ be a quasi-reductive algebraic supergroup and $\g:=\mathrm{Lie} (G)$.

\begin{definition}\label{def:neat_element_aux_notions}
Let $x\in \g_{neat}$.
\begin{itemize}
\item We will denote by $\osp_x$ the image of a Lie superalgebra homomorphism $\bar{i}_x:\osp(1|2) \to \g$ associated to $x$ (see \cref{thrm:JM}). 

This image is isomorphic to $\osp(1|2)$ when $x\neq 0$ and to $0$ when $x=0$. By \cref{thrm:JM}, all such choices of $\osp_x$ are conjugate under the action of $G_{\bar 0}$, so we will often omit the explicit choice of $i_x$. We will call these subalgebras {\it $\osp(1|2)$-type subalgebras}.
\item We will denote by $\sl_x$ the even part $(\osp_x)_{\bar 0}$. When $x\neq 0$, $\sl_x\cong \sl_2$.

    \item Let $h\in (\osp_x)_{\bar 0}$ be such that $[h, x]=x$. We will call it {\it a Cartan element corresponding to $x$}.
   
\item Given a Cartan element $h$ corresponding to $x$, we will denote by $\g=\bigoplus_{i\in \Z} \g^i$ the eigenspace decomposition of $\ad_h$, with $\g^0 = \Ker \ad_h$ and $\bigoplus_{i\geq 0} \g^i$ being the parabolic subalgebra corresponding to $h$. 

We will also write $ \g^+:= \bigoplus_{i> 0} \g^i$.
\end{itemize}    
\end{definition}

\begin{lemma}\label{lem:conjugation_of_squares}

Let $x, x'\in \g_{neat}$. We have:
$x,x'$ are $G_{\bar{0}}$-conjugate if and only if $[x,x]$ and $ [x', x']$ are $G_{\bar{0}}$-conjugate.
\end{lemma}

\begin{proof}
One direction is clear: if $x,x'$ are $G_{\bar{0}}$-conjugate then so are $[x,x]$ and $ [x', x']$.

For the other direction, assume that $[x,x]$ and $ [x', x']$ are $G_{\bar{0}}$-conjugate.

We may reduce the problem immediately to the case when $[x,x]=[x', x']$. Denote this element by $e$. In this case, there exists an element $h\in \g_{\bar{0}}$, unique up to $G_{\bar{0}}$-conjugation, such that $[h, e]=2e$. 

Due to this uniqueness, $h$ must be a Cartan element corresponding to both $x$ and $x'$. So the eigenvalues of $\ad_h$ define a grading on $\g$ corresponding to the Deligne filtrations for both $x$ and $x'$, implying that $x, x'$ have the same Deligne filtration. By \cite[Lemma 4.2.3]{entova2022jacobson}, this implies that $x,x'$ are $G_{\bar{0}}$-conjugate.
\end{proof}
Not every nilpotent even element is the commutator square of some neat element, as the next example shows.
\begin{example}
Let $\g=\gl(2|2)$ and let $e=\begin{bNiceArray}{c|c}
    A &0\\
    \hline
    0 &A
\end{bNiceArray} \in \gl(2|2)_{\bar 0}$ where $A=\begin{bNiceArray}{c|c}
    0 &1\\
    \hline
    0 &0
\end{bNiceArray}$. Assume $e= [x,x]$ for some $x\in \gl(2|2)_{\bar 1}$. Then the action of $x$ on $\kk^{2|2}$ would make $\kk^{2|2}$ into an indecomposable $\G^{(1|1)}$-module, implying that $x$ is not neat. Thus $e\neq [x,x]$ for all $x\in \gl(2|2)_{neat}$.
\end{example}

\begin{corollary}\label{cor:fin_many_neat_orb}
There is an injective map $x\mapsto [x,x]$ from the set of neat nilpotent orbits in $\g$ to the set of nilpotent orbits in $\g_{\bar 0}$. This map preserves the closure order.

\InnaA{In particular, there are finitely many neat orbits in a quasi-reductive Lie superalgebra.}
\end{corollary}
The last result in this corollary was also proved in \cite{entova2022jacobson}.

\section{Attractors of an odd element}\label{sec:attractor}
\subsection{Definition}
Let $G$ be a quasi-reductive supergroup and $\g:=\mathrm{Lie} (G)$.

Let $x\in\g_{\bar 1}$ and let $\frac{1}{2}[x,x]=s+f$ be the Jordan decomposition of $\frac{1}{2}[x,x]$, where $s$ is the semisimple part and $f$ is the nilpotent part. Assume that $f\neq 0$. \InnaA{Since $f\in \g^s$ and the latter is quasi-reductive, we may} fix an $\sl_2$-triple $(e,h,f)$ such that $[s, h]=0$. Let us write 
\begin{equation}\label{eq:h_eigenvectors}
x=x_0+x_{-1}+\dots+x_{-k}
\end{equation} where $[h,x_i]=ix_i$, $[s,x_i]=0$. Note that $[f,x]=0$ so $[f, x_i]=0$ for all $i$.

 Given $M\in \Rep(G)$, the $s$-invariants $M^s$ are naturally a $\g^s$-module and $x$ acts as an odd nilpotent operator on $M^s$. In particular, $\ad_x$ is an odd nilpotent operator on $\g^s$.

We begin with the following easy but important observation:
 \begin{lemma}\label{lem:square_attractor}
We have:
$s=[x_0,x_0]$.
 \end{lemma}
\begin{proof}
The summand of $\frac{1}{2}[x,x]=s+f$ corresponding to the $0$-eigenvalue of $\ad_h$ is $s$. On the other hand, by \eqref{eq:h_eigenvectors} it is $[x_0,x_0]$. So $ s=[x_0,x_0]$.
\end{proof}

\begin{definition}
 Let $x\in \g_{\bar 1}$. If $[x,x]$ is not semisimple, consider the decomposition as in \eqref{eq:h_eigenvectors}. The element $x_0\in \g_{\bar 1}$ is called an {\it attractor} of $x$. If $[x,x]$ is semisimple, we say that $x$ is its own attractor.
\end{definition}
An straightforward property of the attractor is that its orbit lies in the closure of the orbit $G_{\bar 0}.x$:
\begin{lemma}
We have $x_0\in \overline{G_{\bar 0}.x}$.
\end{lemma}
\begin{proof}
    We have $\lim_{\substack{t\in \mathbb{R}, \\t\to \infty}} \operatorname{exp}(\ad_{th})(x)=\lim_{\substack{t\in \mathbb{R}, \\t\to \infty}} (\sum_i e^{-it} x_i)=x_0.$
\end{proof}

%The element $x$ is $\ad$-nilpotent in $\g^s/Z(\g^s)$, making all the constructions and proofs easier in this setting. 

Considering further the decomposition $x=\sum_i x_i$, we see that $[x_0,x_0]=s$, $[x_0, x_{-1}]=0$ and $f=[x_0, x_{-2}]+\frac{1}{2}[x_{-1}, x_{-1}]$. In particular, $x_0\in \g^{hom}$, so we may consider the Duflo-Serganova functor $DS_{x_0}$ (see \cref{ssec:DS_functors_prelim}). Since $x_{-1}\in \g^{x_0}$, we may consider the image $\overline{x}_{-1}$ of $x_{-1}$ in the Lie superalgebra $DS_{x_0}(\g)$.

By the definition of the Duflo-Serganova functor, we have: \InnaA{$$DS_{x_0}(\g):=\frac{\Ker \ad_{x_0}\cap \g^s}{\Im \ad_{x_0}\cap \g^s}=DS_{x_0}(\g^s)$$} so $  DS_{x_0}(\g)_{neat}=DS_{x_0}(\g^s)_{neat}$.

\begin{lemma}\label{lem:x_minus_1_is_neat_in_DS}
 Assume that $f\notin\Im \ad_{x_0}\rvert_{\g^s}$. Then $\overline{x}_{-1} \in DS_{x_0}(\g)_{neat}$.
\end{lemma}

\begin{proof}
By $h$-weight considerations, we have: $[f,x_0]=0$. Together with $[h,x_0]=0$, this implies that $[e,x_0]=0$. So $e,h,f \in \g^{x_0}\rvert_{\g^s}$ and their images $\overline{e}, \overline{h}, \overline{f}$ in $DS_{x_0}(\g^s) $ are either zero, or span a subalgebra of $DS_{x_0}(\g^s)$ isomorphic to $\sl_2$. Furthermore, since $f=\frac{1}{2}[x_{-1}, x_{-1}]+[x_0, x_{-2}]$, we have: $\overline{f}=\frac{1}{2}[\overline{x}_{-1}, \overline{x}_{-1}]$. By our assumption, $\overline{f}\neq 0$. Thus $(\overline{e},  \overline{h},\overline{f}, \overline{x}_{-1})$ generate a copy of $\osp(1|2)$ in $DS_{x_0}(\g^s)$, making $\overline{x}_{-1}$ a neat element in $DS_{x_0}(\g^s)=DS_{x_0}(\g)$.
\end{proof}
\begin{remark}\label{rem:x_minus_1_is_zero_in_DS}
The same argument shows that if $\overline{x}_{-1}=0$ then $\overline{f}=0$ and $f\in \Im \ad_{x_0}\rvert_{\g^s}$.
\end{remark}
\subsection{Uniqueness}
As one can see from the definition, there is usually more than one attractor for each $x\in \g_{\bar 1}$, since we may vary our choice of the $\sl_2$-triple $(e,h,f)$.
Yet it turns out that they are all conjugate.

\begin{lemma}\label{lem:selfcompart} Given $x\in \g_{\bar 1}$, let $\frac{1}{2}[x,x]=s+f$ be the Jordan decomposition of $\frac{1}{2}[x,x]$ as before. Let $(e,h, f)$, $(e',h', f)$ be two $\sl_2$-triples in $\g^s_{\bar 0}$ containing $f$. Suppose that we have decompositions of $x$ with respect to $\ad_h, \ad_{h'}$:
$$x=x_0+x_{-1}+\ldots=x_0'+x'_{-1}+\ldots$$
so that $[h, x_i]=ix_i$, $[h', x'_i]=ix'_i$ for each $i$ (and $[s,x_i]=[s, x'_i]=0$).
Then $x_0$ and $x_0'$ belong to the same $G_{\bar 0}$-orbit. 

Moreover, if $x_0=x_0'$ then $x_{-1}$ and $x'_{-1}$ belong to the same $G_{\bar 0}$-orbit. 
\end{lemma}
\begin{proof} 
Recall the notation $C_{G_{\bar 0}}(f)\subset G_{\bar 0}$ and $\mathfrak{c}_{\g_{\bar 0}}(f)\subset \g_{\bar 0}$ for the centralizers of $f$, as in \cref{ssec:centralizers}. By Kostant's Theorem, the $\mathfrak{sl}_2$-triple $(e',h', f)$ is obtained from $(e,h,f)$ via conjugation by an element of the connected centralizer subgroup $C_{G_{\bar 0}}(f)$.

Now, $\mathfrak{c}_{\g_{\bar 0}}(f)$ inherits the grading $\g=\bigoplus_{i\in \Z} \g^i$ defined by $\ad_h$. We denote $$\mathfrak{c}_{\g_{\bar 0}}(f)^i:=\mathfrak{c}_{\g_{\bar 0}}(f)\cap \g^i.$$ Clearly, $\mathfrak{c}_{\g_{\bar 0}}(f)^i=0$ for $i>0$.

Each element of $C_{G_{\bar 0}}(f)$ is a product of exponents $\exp(u)$ where $u\in \mathfrak{c}_{\g_{\bar 0}}(f)$.

For $u\in \mathfrak{c}_{\g_{\bar 0}}(f)^{<0}, u_0\in \mathfrak{c}_{\g_{\bar 0}}(f)^{0}$ we have:
$$\exp(u_0+u)=g'\exp(u)\exp(u_0)$$ where $g'$ is a product of exponents of elements in $\mathfrak{c}_{\g_{\bar 0}}(f)^{<0}$. \InnaA{Furthermore, $\mathfrak{c}_{\g_{\bar 0}}(f)^0$ centralizes $h, e$; thus} the adjoint action of $C_{G_{\bar 0}}(f)$ on the $\mathfrak{sl}_2$-triples containing $f$ is generated by $\Ad_g$ where $g=\exp(u)$, $u\in \mathfrak{c}_{\g_{\bar 0}}(f)^{<0}$.

Let $g=\exp(u_{-1}+\dots+u_{-n})\in C_{G_{\bar 0}}(f)$ where $u_i\in \mathfrak{c}_{\g_{\bar 0}}(f)^i$, and let $h'':=\operatorname{Ad}_{g}^{-1}(h), e'':=\operatorname{Ad}_g^{-1}(e)$. 

We will later prove that the decomposition of $x$  into $\ad_{h''}$-eigenvectors satisfies the following condition:

{\bf Condition ($\star$):}
$x=\sum_{i\leq 0} x''_i$ where $[h'', x''_i]=ix''_i$, $x''_0\in G_{\bar 0}.x_0$, and if $x''_0=x_0$ then $x''_{-1}\in G_{\bar 0}.x_{-1}$. 

Applying this argument several times, we will conclude that for $g\in C_{G_{\bar 0}}(f)$ we have: $Ad_g^{-1}(h)$ induces a decomposition of $x$ satisfying Condition ($\star$). This implies the statement of the Lemma.

We now prove that for $g=\exp(u_{-1}+\dots+u_{-n})$ as above and $h'':=\operatorname{Ad}_g^{-1}(h)$, the decomposition of $x$ into $\ad_{h''}$-eigenvectors satisfies Condition ($\star$).
Indeed, in the grading induced by $\ad_{h''}$ we have \InnaA{$x=\sum_{i\leq 0} x''_i$, where
$$x''_0=\Ad^{-1}_g(x_0), \;\;x''_{-1}=\Ad^{-1}_g(x_{-1}+[x_0,u_{-1}]), \;\;x''_{-2}=\Ad^{-1}_g(x_{-2}+[u_{-1},x_{-1}]+\frac{1}{2}[u_{-1},[u_{-1},x_0]])$$
%with $z_g$ being the sum of terms in degrees lower than $ -2$ with respect to $\ad_{h''}$. 
}
In particular, $x_0''=\Ad_g^{-1}(x_0)$. If $x_0=x_0''$ then $x_0=\Ad_g(x_0)$; that is, 
$$ x_0=\exp(\ad_{\sum_{i<0} u_i})(x_0)= x_0 + [u_{-1}, x_0]+z$$
where $x_0\in \g^0, ~[u_{-1}, x_0]\in \g^{-1}$ and $ z\in \g^{\leq -2}$. This implies that $[u_{-1}, x_0]=0=z$
and so $x''_{-1}=\Ad^{-1}_g(x_{-1}+[x_0,u_{-1}])=\Ad_g^{-1}(x_{-1})$.  
\end{proof}
\begin{remark}
    The same argument shows that whenever $x_0=x_0'$ and $x_{-1}=0$, we have: $x_{-2}$ and $x'_{-2}$ belong to the same $G_{\bar 0}$-orbit. 
\end{remark}

As a consequence, we have the following well-defined notation:

\begin{notation}\label{notn:attractor}
  \InnaA{Let $x\in \g_{\bar 1}$, and let $\frac{1}{2}[x,x]=s+f$ be the Jordan decomposition of $\frac{1}{2}[x,x]$. 
  
  If $[x,x]$ is semisimple (i.e. $f=0$), denote $att(x):=\{x\}$. If $[x,x]\in \g_{\bar 0}$ is not semisimple, denote by $att(x)$ the set of all the attractors of $x$. }
\end{notation}

\subsection{Criterion for neatness}\label{ssec:neat_criterion}

\begin{lemma}\label{lem:criterion_neatness_h_decomp}
Let $x\in \g_{\bar 1}$. 
If $att(x)=\{0\}$ then $x\in \g_{neat}$ and in any decomposition $x=\sum_i x_i$ as above, we have $x_{-1}\in G_{\bar 0}.x$.
\end{lemma}
The second part of the statement sits well with \cref{lem:selfcompart}: indeed, by \cref{lem:selfcompart} if $att(x)=\{0\}$ then any two decompositions of $x$ yield conjugate $(-1)$-components. 
\begin{proof}
Given $x\in \g_{\bar 1}$, let $\frac{1}{2}[x,x]=s+f$ be the Jordan decomposition of $\frac{1}{2}[x,x]$. 

Assume $att(x)=0$. By the definition of attractor, we either have $x=0\in \g_{neat}$ or $f\neq 0$. We will from now on assume that $f\neq 0$ and let $(e,h, f)$ be an $\sl_2$-triple in $\g^s_{\bar 0}$. Consider the $\ad_h$-grading $\g=\oplus_{i\in \Z}\g^i$ and let
$x=\sum_{i=1}^{n} x_{-i}$ be the decomposition of $x$ with $x_i \in \g^i$ (by assumption, $x_0=0$). 

Recall that $f=\frac{1}{2}[x_{-1}, x_{-1}]$ (since $x_0=0$), so $e,h,f,x_{-1}$ generate a copy of 
$\osp(1|2)$, which implies that $x_{-1}\in \g_{neat}$.

We will now show that there exists $g\in G_{\bar 0}$ such that
$\operatorname{Ad}_g(x)=x_{-1}$. This will complete the proof. Indeed, assume that
$n\geq 2$ (otherwise $x=x_{-1}$ and we are done). Recall that $\g^{<0}\subset \Im \ad_{x_{-1}}$ so there exists $w\in \g^{1-n}_{\bar 0}$ such that $\ad_{x_{-1}}(w)=x_{-n}$. Then
 $$\operatorname{exp}(\ad_w)(x)=x-x_{-n}+y=x_{-1}+\dots+x_{-n+1} +y,\quad y\in \oplus_{i<-n}\g^{i}.$$

Repeating this argument several times and using finiteness of the grading $\g=\oplus_{i\in \Z}\g^i$, we can find an element $g\in G_{\bar 0}$ such that $\operatorname{Ad}_g(x)=x_{-1}+x_{-2}+\dots+x_{-m}$ with $m<n$. We can now finish the proof by induction on $n$.
\end{proof}

\section{Balanced nilpotent elements}\label{sec:balanced}
Let $G$ be a quasi-reductive group and $\g:=\Lie(G)$.
\subsection{Definition of balanced elements}
\begin{definition}
    Let $V$ be a finite-dimensional vector superspace over $\kk$. An odd nilpotent operator $\varphi \in \InnaA{\End^{\bullet}(V)_{\bar 1}}$ is called {\it balanced} if the size of each of its Jordan blocks is either $1$ or even.
\end{definition}

\begin{definition}
    Let $x\in \g_{\bar 1}$ and let $s$ be the semisimple part in the Jordan decomposition of $\frac{1}{2}[x,x]$. The element $x$ is called {\it balanced} if for any $M\in \Rep(G)$, $x\rvert_{M^s}$ is a balanced operator.
\end{definition}
\begin{example}
    Let $x\in \g^{hom}$, that is, $s:=\frac{1}{2}[x,x]$ is semisimple. We claim that $x$ is balanced. Indeed, for any $M\in \Rep(G)$, $x\rvert_{M^s}$ is an odd nilpotent operator which squares to zero, so all its Jordan blocks have sizes $1$ or $2$. In particular, any self-commuting odd element is balanced.
\end{example}

The above example provides the motivation for the definition of a balanced element. Just like in the special case of a homological element, we may describe the property of being balanced nicely through the associated symmetric monoidal functors (see \cref{ssec:DS_functors_prelim}).

Let $x$ be an balanced element in $\g$ and let $s$ be the semisimple part of the Jordan decomposition of $\frac{1}{2}[x,x]$. For any $M\in \Rep(G)$, consider the associated action of $\G^{(1|1)}$ on $M^s$. Since $x$ is balanced, any indecomposable $\G^{(1|1)}$-summand in $M^s$ is a single Jordan block of $x$ which is either of even size or of size $1$. Thus $\Phi_x(M)$ is isomorphic, as a vector space, to the sum of the indecomposable $\G^{(1|1)}$-summands of dimension $1$. The corresponding $OSp(1|2)$-action on this space is trivial and we may consider the functor
   $\Phi_x$ as a functor $\Rep(G)\to \sVect$.

\begin{notation}
    Let $\g_{bal}$ denote the set of balanced elements in $\g_{\bar 1}$. 
\end{notation}

\begin{remark}\label{rmk:balanced_func_faithful}
    From the description above, the functor $\Phi_x$ is faithful if and only if $x=0$. In particular, we obtain: $\g_{neat}\cap \g_{bal}=\{0\}$.

    We also saw that $\g^{hom}\subset \g_{bal}$.
   \end{remark}

\begin{lemma}\label{lem:Phi_for_x_bal}
    Let $x\in \g_{bal}$. There exists a natural isomorphism $$\eta: \Phi_x \to \widetilde{DS}_x$$ where $\widetilde{DS}_x: \Rep(G)\to \sVect$ is the functor defined by $$\widetilde{DS}_x(M):=\frac{\Ker ~ x\rvert_{M^s}}{\Im ~ x\rvert_{M^s}\cap \Ker ~ x\rvert_{M^s}}$$ for every $M\in \Rep(G)$.
\end{lemma}
\begin{proof}
       Since we are only interested in the action of $x$ on $s$-invariant subspaces, we may assume for simplicity that $s=0$, $x$ is $\ad$-nilpotent and $M=M^s$.
       Consider the Deligne filtration of $M$ with respect to the nilpotent operator $x:M\to M$. This filtration is given by 
    $$\mathcal{F}_k(M)=\sum_{j-i=k+1} \Ker x^j\rvert_{M} \cap \Im x^i\rvert_{M}.$$
\InnaA{Here we use the convention $x^0\rvert_{M} := \id_M$, $x^i \rvert_{M} := 0$ for $i<0$.}

Consider the composition $$\eta_M:\mathcal{F}_k(M) ~\hookrightarrow ~\Ker ~x\rvert_M~\twoheadrightarrow~ \frac{M^x}{xM\cap M^x}.$$
This composition defines a natural transformation of functors \InnaA{$ \mathcal{F}_k\to \widetilde{DS}_x$.
By definition, $\eta_M$ is zero for $i<0$ since $\mathcal{F}_i(M)\subset \Im x\rvert_{M}$ for $i<0$.}

    Recall that $\Phi_x(M)=Gr^{ev}M = \bigoplus_k  \quotient{\mathcal{F}_{2k}(M)}{\mathcal{F}_{2k-1}(M)}$.
    Now the natural maps $\eta_M$ define a natural transformation $\eta: \Phi_x \to \widetilde{DS}_x$. To check that this is an isomorphism, it is enough to choose a Jordan basis for $x\rvert_M$ and compute $\eta_M$ explicitly: the trivial Jordan blocks yield a basis for $\Phi_x(M)$ and are sent bijectively to the corresponding basis in $M^x/(xM\cap M^x)$.
\end{proof}

Below, we will give a sufficient criterion for $x\in \g_{\bar 1}$ to be balanced, similar to the criterion in \cref{ssec:neat_criterion}. 
\subsection{Strongly balanced elements}\label{ssec:strongly_balanced}

Let $x\in \g_{\bar 1}$ and let $\frac{1}{2}[x,x]=:s+f$ be the Jordan decomposition of $\frac{1}{2}[x,x]$. 

\begin{comment}
Assume that $f\neq 0$ and let $(e,h,f)$ be an $\sl_2$-triple. Consider the $\ad_h$-eigenvector decomposition $x=\sum_{i\leq 0}x_i$ where $[h,x_i]=ix_i$, $[s,x_i]=0$ for any $i$ (see \cref{eq:h_eigenvectors}).
\end{comment}

\begin{definition}\label{def:strongly_bal}
    We say that $x\in\g_{\bar 1}$ is {\it strongly balanced} if one of the two conditions below hold:
    \begin{itemize}
        \item $x\in \g^{hom}$ (that is, $f=0$ and $\frac{1}{2}[x,x]=s$ is semisimple),
        \item $f\neq 0$ and there exists an $\sl_2$-triple $(e,h,f)$ so that in the $\ad_h$-eigenvector decomposition $x=\sum_{i\leq 0}x_i$ (see \cref{eq:h_eigenvectors}) we have: $x_{-1}=0$.
    \end{itemize}
  
    For a strongly balanced $x$, we denote by $\sl_x\subset \g_{\bar 0}$ an $\sl_2$-subalgebra spanned by appropriate $e,h,f$ when $f\neq 0$, or the zero subalgebra when $f=0$. 
\end{definition}
There are usually several choices of $\sl_x$ for a given $x$, but they are all $C_{G_{\bar 0}}(f)$-conjugate. We remind the reader that $[h,x_0]=[f, x_0]=0$ which implies that $[\sl_x, x_0]=0$.

\begin{remark}
By \cref{lem:balancedLevi}, we may choose $\sl_x\subset \g^s$.
\end{remark}

  \begin{remark}
    In the definition above, the condition $x_{-1}=0$ implies that $f=[x_0, x_{-2}]$.
    \end{remark}
    
\begin{proposition}\label{prop:x_1_can_be_eliminated}
    Let $x\in \g_{\bar 1}$. Assume that $f\neq 0$ (that is, $x\notin \g^{hom}$). Let $(e,h,f)$ be an $\sl_2$-triple in $\g^s_{\bar 0}$ and let $x=\sum_{i\leq 0}x_i$ be the $\ad_h$-decomposition as above. If $x_{-1}\in \Im \ad_{x_0}\rvert_{\g^s}$ then $x$ is strongly balanced in $\g$.
\end{proposition}

\begin{proof}
    Let $\g^s=\bigoplus_{i\in \Z} \g^i$ be the $\ad_h$-grading on the Levi sub-superalgebra $\g^s\subset \g$.  

    Let $\sl_x:=\mathrm{span}\{e,h,f\}$ and let $M$ be the simple $\sl_x$-module generated by $x_{-1}$. The $h$-weight of $x_{-1}$ is $-1$ and $[f,x_{-1}]=0$, so $\dim M=2$. Recall that $[x_0, \sl_x]=0$, so $\ad_{x_0}$ is an $\sl_x$-equivariant odd endomorphism of $\g^s$. Since $x_{-1}\in \Im \ad_{x_0}\rvert_{\g^s}$, this implies that there exists a simple $2$-dimensional $\sl_x$-submodule $M'\subset\g^s$ such that $\ad_{x_0}(M')=M$. In particular, there exists $u\in M'_{\bar 0}\subset \g^s_{\bar 0}$ such that $[x_0, u]=x_{-1}$ and $[f,u]=0$, $[h,u]=-u$.
    
    Let us consider the automorphism $\varphi:=\exp(\ad_u)$ of $\g^s$. This automorphism takes the $\sl_2$-triple $(e,h,f)$ to some $\sl_2$-triple $(\varphi(e),\varphi(h),f)$. 
    We claim that the $\ad_{\varphi(h)}$-decomposition of $x$ has no summand of $\varphi(h)$-weight $-1$, which would imply that $x$ is strongly balanced.

    Consider the decomposition of $\varphi^{-1}(x)$ with respect to the $\Z$-grading $\bigoplus_i \g^i$. We obtain:
    \begin{align*}
    \varphi^{-1}(x)&=x-x_{-1}-[u, x_{-1}]-[u, x_{-2}]-\ldots +\frac{1}{2}\ad_u^2(x_0)+\frac{1}{2}\ad_u^2(x_{-1})+\ldots  \\
    &=x_0+ ~(\text{ elements of degree (-2) or lower }).
    \end{align*}
    
    So in the decomposition of $x$ into $\ad_{\varphi(h)}$-eigenvectors, we obtain the summand $\varphi(x_0)$ (of $\ad_{\varphi(h)}$-weight zero) and summands of $\ad_{\varphi(h)}$-weights $(-2)$ and lower. This implies that $x$ is strongly balanced.
\end{proof}

\subsection{Strongly balanced implies balanced}\label{ssec:strong_bal_implies_bal}

From now and until the end of this subsection, assume that $x$ is strongly balanced. Let $\frac{1}{2}[x,x]=s+f$ be the Jordan decomposition of $\frac{1}{2}[x,x]$ and fix a subalgebra $\sl_x$ containing $f$.

We will prove that in this case, $x$ is balanced and the symmetric monoidal functors corresponding to $x$ and to its attractor $x_0$ are isomorphic; that is, $$\Phi_x\cong \Phi_{x_0}=DS_{x_0}.$$

We will start with the case when $x\in \g_{\bar 1}$ is an $\ad$-nilpotent strongly balanced element with $f:=[x,x]\neq 0$. Let $(e,h,f)$ be a suitable $\sl_2$-triple, so that the $\ad_h$-decomposition of $x$ is given by $x=x_0+x_{-2}+\dots$. For any $\g$-module $M$, we may consider the Deligne filtration $\mathcal{F}_f^\bullet( M)$ on $M$ associated with the nilpotent operator $f\rvert_M$. Its associated graded is given by the $h$-eigenspace grading $M=\bigoplus_k M_k$ on $M$. 

The action of $x$ preserves the filtration $\mathcal{F}_f^\bullet M$; taking the corresponding action on the associated graded, we obtain the action of $x_0$ on $\bigoplus_k M_k$ (recall that $x_0$ commutes with $\sl_x$).

\begin{lemma}\label{lem:strb} Let $x\in \mathcal{N}^{\ad}(\g_{\bar 1})$ be strongly balanced. Consider the Lie sub-superalgebra $\mathfrak k=\sl_x\times \kk[x_0]$ of $\g$. For any $\g$-module $M$ we have:
\begin{enumerate}
\item Every non-trivial indecomposable $\mathfrak k$-summand of $M$ is isomorphic to $V_k\otimes P$ where $V_k$ is an irreducible $\sl_x$-module with highest weight $k$ and $P$ is an indecomposable projective $\kk[x_0]$-module. In particular, $DS_{x_0}(M_i)=0$ for $i\neq 0$.
\item Every non-trivial indecomposable $\kk[x]$-summand of $M$ has superdimension $0$, hence $ x\in \g_{bal}$.
\end{enumerate}
\end{lemma}
\begin{proof} 
\mbox{}

\begin{enumerate}
\item Recall that $[\sl_x, x_0]=0$. By the complete reducibility of $\sl_2$-modules we can see that every indecomposable $\mathfrak k$-module is isomorphic either to $V_k$ or to $V_k\otimes P$, where $P$ is an indecomposable projective $\kk[x_0]$-module of dimension $2$ (such $P$ is unique up to isomorphism). 

Let us assume that there exists a indecomposable $\mathfrak k$-summand of $M$ which is isomorphic to $V_k$ for some $k>0$.
Let $v$ be a highest weight vector in this summand. 

By our assumption,
$x_0 v=0$ but $v\notin x_0 M$. Recall that $f=[x_0, x_{-2}]$, so $fv=x_0x_{-2}v$.
Note that $f^{k}v$ is a lowest weight vector and we have
$$f^k v=f^{k-1}x_0x_{-2} v=x_0 f^{k-1} x_{-2} v.$$
 
The
application of $e^k$ to both sides of the latter equality gives 
$$ x_0 (e^k f^{k-1} x_{-2}) v=e^k x_0 f^{k-1} x_{-2} v = e^kf^k v\in \kk^{\times} v$$ which implies $v\in x_0 M$. This is a contradiction.

\item Consider a $\mathfrak k$-decomposition of $M$:
$$M= \bigoplus_\alpha N_{\alpha} \oplus \widetilde{M}$$ where $\widetilde{M}$ is a trivial $\mathfrak{k}$-module, while each $N_{\alpha}$ is an indecomposable $\mathfrak k$-module which is isomorphic to
$V_k\otimes P$ for some $k$. Fix $\alpha$ and let $N_{\alpha}
\cong V_k\otimes P$. We may choose $v\in N_{\alpha}$ corresponding to $v'\otimes u\in V_k\otimes P$, where $v'\in V_k$ is a highest weight vector and $u\in P$ satisfies: $x_0u\neq 0$. Therefore $v$ is a highest weight vector in $N_{\alpha}$ with respect to the action of $\sl_x$ and $x_0 v\neq 0$.

Then $x_0f^k v\neq 0$. \InnaA{This is the $(-2k)$-eigenvector in the $h$-decomposition of $xf^kv$} so $ xf^kv=x^{2k+1}v\neq 0$. On the other hand, $f^{k+1}v=x^{2k+2}v=0$. Thus,
$v$ generates an indecomposable $\kk[x]$-submodule $N'_{\alpha}$ of dimension $(k+1|k+1)$. 
Consider the $\kk[x]$-submodule
\begin{equation}\label{eq:strongly_balanced_dir_sum}
 N=\sum_\alpha N'_{\alpha} \subset M.   
\end{equation} 
We claim that this sum is direct. Indeed, passing from the filtered space $\mathcal{F}_f^\bullet N\subset \mathcal{F}_f^{\bullet} M$ to the associated graded, the summands $N'_{\alpha}$ in Eq. \eqref{eq:strongly_balanced_dir_sum} are sent respectively to the direct summands in the decomposition $\bigoplus_\alpha N_{\alpha}$.

Recall that $h$-grading $M=\bigoplus_i M_i$. Let $m\in \widetilde{M}$. We have: $m\in M_0$ and $xm=\sum_{i\leq -2} x_im \in \oplus_{j<0} M_j$. This implies: $xm\in \bigoplus_{\alpha} N_{\alpha}$. Now, as linear subspaces of $M$, we have: $N=\bigoplus_\alpha  N_{\alpha}$. So $xm\in \mathcal{F}_f^{< 0} N$. By the definition of $N$, we conclude that $xm\in xN$.

Choose a basis $\{m_\beta\}_{\beta\in B}$ in $\widetilde M$. By the argument above, for every $k$ there exists $n_\beta\in \mathcal{F}_f^{< 0} N$ such that $xm_\beta=xn_\beta$.
Let $M':=\mathrm{span}\{m_\beta-n_\beta~|~\beta\in B \}$. Then $M'$ is a trivial $\kk[x]$-submodule of $M$ and $$M=N\oplus M'$$ as $\kk[x]$-modules (with $M'\cong \widetilde{M}$ as vector spaces).

\end{enumerate}
\end{proof}

\begin{proposition}\label{prop:isom_DS_funcs_balanced}
    Let $x\in \mathcal{N}^{\ad}(\g_{\bar 1})$ be strongly balanced, and let $x_0$ be its attractor.
    
    We have a natural isomorphism $\widetilde{DS} ~\xrightarrow{\sim}~DS_{x_0}$ where $$ \widetilde{DS}(M):=M^x / (xM\cap M^x), \;\;\; DS_{x_0}(M)=M^{x_0} / {x_0}M.$$
\end{proposition}
\begin{proof}
Let $(e,h,f=\frac{1}{2}[x,x])$ be an $\sl_2$-triple producing $x_0$ as an attractor of $x$. 
    Consider the $h$-grading $M=\bigoplus_i M_i$ and the projection
$$pr:M \twoheadrightarrow M_0=M^h.$$ 
Let $m\in M^x$. We have $fm=0$ so $m\in \bigoplus_{i\leq 0} M_i$. Let us write $m=\sum_{i\leq 0} m_i$ where $m\in M_i$. Then $pr(m)=m_0$ and $fm_i=0$ for all $i$. Now, $xm \in x_0m_0 + \bigoplus_{i< 0} M_i$, but since $xm=0$ we obtain: $$hm_0=fm_0=x_0m_0=0.$$

This shows that the $pr(M^x)\subset M^{\mathfrak{k}}\subset M^{x_0}$ (see the notation in \cref{lem:strb}). In particular, we obtain a natural transformation $(-)^x\to (-)^{x_0}$ given by the above map $pr:M^x\to M^{x_0}$. 

Now, let $m\in M$ such that $xm\in M^x$. Then $fm=0$, so again $m\in \bigoplus_{i\leq 0} M_i$. Let us write $m=\sum_{i\leq 0} m_i$, $m\in M_i$. Again, $xm \in x_0m_0 + \bigoplus_{i< 0} M_i$ and we have: $pr(xm)=x_0m_0$. Hence $ pr(xM\cap M^x)\subset x_0M_0\subset x_0M$.

So we obtain a natural transformation $\widetilde{DS} ~\longrightarrow~DS_{x_0} $ given by the morphism $$M^x / (M^x\cap xM) ~\longrightarrow ~DS_{x_0}(M)=M^{x_0} / {x_0}M.$$ 

It remains to check that this natural transformation is an isomorphism. Consider the subspaces $\widetilde{M}, M' \subset M$ constructed in the proof of \cref{lem:strb}. Clearly, we may identify (non-canonically) $$\widetilde{M}\cong M^{x_0} / {x_0}M, ~~ M' \cong M^x / (M^x\cap xM).$$ The explicit isomorphism $\widetilde{M}\to M', ~ m_k\mapsto m_k-n_k$ then gives the inverse to the map $M^x / (M^x\cap xM) ~\longrightarrow ~M^{x_0} / {x_0}M$ we previously defined.
\end{proof}

\begin{corollary}\label{cor:criterion_balanced}
Assume that $x\in \g_{\bar 1}$ is strongly balanced and let $x_0\in att(x)$.

Then $x\in \g_{bal}$ and there exists a natural isomorphism $\Phi_x \xrightarrow{\sim} DS_{x_0}=\Phi_{x_0}.$
\end{corollary}
\begin{proof}
Let $x\in \mathcal{N}^{\ad}(\g_{\bar 1})$ be strongly balanced. By \cref{lem:strb}, $x$ is balanced. The natural isomorphism in this case is provided by \cref{lem:Phi_for_x_bal} and \cref{prop:isom_DS_funcs_balanced}. More generally, given any strongly balanced element $x\in \g_{\bar 1}$ with Jordan decomposition $\frac{1}{2}[x,x]=s+f$, we may consider the action of the quasi-reductive Lie superalgebra $\g':=\g^s/\kk s$ on $M^s$ for any $M\in \Rep(G)$. The image $\widetilde{x}$ of $x$ in $\g'$ is $\ad$-nilpotent and strongly balanced, with attractor $\widetilde{x}_0$ (the image of $x_0$ in $\g$). The functors $\Phi_x$, $\Phi_{x_0}$ are naturally isomorphic to the compositions $\Phi_{\widetilde x}\circ (-)^s$, $\Phi_{{\widetilde x}_0}\circ (-)^s$ respectively. We now apply our previous results to prove the desired claim.
\end{proof}
We conclude this section with a proposition which allows us to identify balanced elements in the general linear Lie superalgebra, in a manner similar to the neatness criterion in \cref{thrm:JM}.
\begin{proposition}\label{prop:balanced_in_gl}
     Let $V$ be a finite-dimensional vector superspace and let $x\in \gl(V)_{\bar 1}$ be an element of the general linear Lie superalgebra. Let $\frac{1}{2}[x,x]=s+f$ be the Jordan decomposition of $\frac{1}{2}[x,x]$.
     
     Then $x$ is strongly balanced (thus, balanced) if and only if the operator $x\rvert_{V^s}$ is balanced. 
\end{proposition}

\begin{proof}
Let us first consider the ``only if'' direction: if $x$ is strongly balanced then \cref{cor:criterion_balanced} implies that $x$ is balanced, so the operator $x\rvert_{V^s}$ is balanced. 

We now assume that the operator $x\rvert_{V^s}$ is balanced. We need to prove that $x$ is strongly balanced. 
 
 If $f=0$ then we are done. Otherwise, $f\neq 0$ and we would like to define a suitable $\sl_2$-subalgebra $\sl_x\subset \gl(V^s)\subset \gl(V)^s$ (see \cref{def:strongly_bal}). 
 
We may reduce the problem to the case when $x$ is $\ad$-nilpotent (so $f=\frac{1}{2}[x,x]$): under this assumption, we will construct an $\sl_2$-action on $V$ extending $f\rvert_{V}$.  

It is enough to do so on each subspace of $V$ forming a single Jordan block of $x$, so we may assume for simplicity that $x$ has a just one Jordan block in $V$. If this block is trivial (i.e. $\dim V=1$) then the corresponding action of $\sl_2$ on it is defined to be trivial. 

Otherwise, this block has even size, so $\dim V\in 2\Z$. Let $(v_i)_{i=0}^{2n-1}$ be the Jordan basis for $x$, so that
$$ \forall\, i\leq 2n-2,\;xv_i=v_{i+1}, ~xv_{2n-1}=0.$$
Then for any $ i\leq 2n-3$ we have: $\;fv_i=v_{i+2}$ and $fv_{2n-1}=fv_{2n-2}=0$. That means that $f$ acts by a single Jordan block on each of the subspaces $span(v_{2i})_{i=0}^{n-1}$ and $span(v_{2i+1})_{i=0}^{n-1}$ and its action obviously extends to an $\sl_2$-action on each of these subspaces. Each of them becomes a simple $\sl_2$-module of dimension $n$ (one purely even and one purely odd). 
Let $h\in \End(V)$ be the Cartan element with respect to this action, so that for any $ i$ we have: $h.v_{2i}=h.v_{2i+1}=n-2i+1$.

We now have a decomposition $x=x_0+x_{-2}$, where $x_0, x_{-2}$ are operators defined on the basis via
$$x_0(v_i):=\begin{cases}
    v_{i+1} &\text{ if } i\in 2\Z\\
    0 &\text{ if } i\in 2\Z+1\\
\end{cases}, \;\;\;\;\;\;x_{-2}(v_i):=\begin{cases}
   0 &\text{ if } i\in 2\Z\\
     v_{i+1} &\text{ if } i\in 2\Z+1\\
\end{cases}.$$
Clearly, $[h, x_0]=0$ and $[h, x_{-2}]=-2x_{-2}$, implying that $x$ is strongly balanced.
\end{proof}

\subsection{Correspondence between strongly balanced elements and their attractors}\label{ssec:attract_corresp}

In this section, we study the correspondence between the set of strongly balanced elements in $\g$ and the set of cohomological elements. This correspondence is given by the attractor construction.

\begin{proposition}\label{lem:balancedfin} 
Let $x$ and $x'$ be two strongly balanced $\ad$-nilpotent elements.
Assume that $[x,x]=[x',x']$ and $att(x)=att(x')$. Then $G_{\bar 0}.x=G_{\bar 0}.x'$.
\end{proposition}
    \begin{proof} Let $f:=\frac{1}{2}[x,x]$ and $x_0:=att(x)=att(x')$. We may assume that $f\neq 0$, otherwise $x_0=x=x'$ and we are done. Since all the $\sl_2$-triples containing $f$ are conjugate, we may assume without loss of generality the existence of an $\sl_2$-triple $(e,h,f)$ such that the $\ad_h$-decompositions of $x, x'$ are
    $$x=x_0+\sum_{i\leq -2} x_{i},\quad  x'=x_0+\sum_{i\leq -2} x'_{i}.$$
    We denote the corresponding $\sl_2$-subalgebra of $\g$ by $\sl_x$, as before.
    Let $\g=\bigoplus_i \g^i$ be the $\ad_h$-grading on $\g$.
    
    Let us choose the maximal $i$ such that $x_i\neq x'_i$. Set $y_i:=x_i-x'_i$. We have $[h,y_i]=iy_i$, $[f,y_i]=0$
    and $[x_0,y_i]=0$. The last equality follows from the fact that the projection of $[x,x]-[x',x']=0$ on $\g^i$ is 
    $$\sum_{i\leq j\leq 0} \left([x_j, x_{i-j}] - [x'_j, x'_{i-j}]\right) =[x_0,y_i] $$ (here we rely on the equality $x_j=x'_j$ for $j>i$). 

    \InnaA{Consider the Lie sub-superalgebra $\g^{x_0} \subset \g $ and its ideal $[x_0,\g] $.
    
    Recall that $x_0$ commutes with $h, f$, meaning that $\sl_x\subset  \g^{x_0}$.
    Since $[x_0,\g] \subset \g^{x_0}$ is an ideal and $f = [x_0, x_{-2}]\in [x_0,\g]$ , we have: $\sl_x\subset [x_0,\g]$. So $\ad_h$ induces a $\Z$-grading $\g^{x_0} = \bigoplus_j \g^j\cap \g^{x_0}$, and a respective grading on $  [x_0,\g]$. 

    Next, consider the $\sl_x$-action on $[x_0,\g] \subset \g^{x_0} $.
    The subalgebra $\sl_x$ lies in the ideal $ [x_0,\g] \subset \g^{x_0}$ so $\sl_x$ acts trivially on the quotient $\g^{x_0}/[x_0,\g] $. Hence $ \g^i\cap \g^{x_0} \subset [x_0, \g]$ for any $i\neq 0$.     In our case, $y_i \in \g^i\cap \g^{x_0}$ with $i<0$ so $y_i\in [x_0, \g]$.
} 

Let $u_i\in \g_{\bar 0}$ such that $[x_0, u_i]=y_i$ and $[h,u_i]=iu_i$.
\InnaA{Note that $$[x_0, [f, u_i]]=[f, [x_0, u_i]]=[f, y_i]=0$$
so  $[f, u_i]\in \g^{x_0} \cap \g^{i-2}$. Furthermore, $\g^{x_0}  $ is an $\sl_x$-submodule, so there exists $z\in  \g^{x_0} \cap \g^i$ such that  and $[f,z]=[f, u_i]$. Replacing $u_i$ by $w_i:=u_i-z$, we get: $[f, w_i]=0, [x_0, w_i]=y_i$.
} 
    Let $x'':=\Ad_{\exp(-w_i)}(x')$. The pair $(x, x'')$ satisfies the conditions of the proposition and 
$$x''= x_0+\sum_{j\leq -2} x'_{j}-[w_i, x_0]+ z, \;\;\;\; z\in \bigoplus_{j<i} \g^j$$
So
$x_0+\sum_{j\leq {-2}}x''_i$ with $x_j=x_j''$ for $j\geq i$. Since the set of indices $i$ for which $x_i\neq 0$ is bounded from below, repeating this procedure several times leads to the required result.
        
      \end{proof}
\begin{corollary}\label{cor:balancednil} Let $x_0\in \g^{\mathbf{sc}}$ be a self-commuting element. There are finitely many, up to conjugation, strongly balanced $\ad$-nilpotent elements $x$ with $x_0\in att(x)$.
\end{corollary}
\begin{proof} Let $x$ be  a balanced $\ad$-nilpotent element with attractor $x_0$ and $f:=\frac{1}{2}[x,x]$. As it was shown in the proof of \cref{lem:balancedfin}, $f\in[x_0,\g]$ and $f$ can be embedded into an $\sl_2$-subalgebra of  $\mathfrak k=[x_0,\g]$.
Let $K\subset G_{\bar 0}$ be the adjoint group with the Lie algebra $\mathfrak{k}_{\bar 0}$. 
By \cref{lem:sl2} below, two embeddings $\sl_2\to\mathfrak k$ are conjugate if and only if the induced
embeddings $\sl_2\to \mathfrak{k}/\operatorname{rad}(\mathfrak k)$ are conjugate. Since $ \mathfrak{k}/\operatorname{rad}(\mathfrak k)$ is a reductive Lie algebra, there are finitely many conjugacy classes of embeddings $\sl_2\to \mathfrak{k}/\operatorname{rad}(\mathfrak k) $.
Hence there are finitely many (up to conjugation) choices of $f$.
Now the statement follows from \cref{lem:balancedfin}.
\end{proof}
\begin{lemma}\label{lem:sl2} Let $\mathfrak k$ be a Lie algebra. Two embeddings $\sl_2\to\mathfrak k$ are conjugate if and only if  the induced
embeddings $\sl_2\to \mathfrak{k}/\operatorname{rad}(\mathfrak k)$ are conjugate.     
\end{lemma}
\begin{proof} Let us fix a  Levi decomposition $\mathfrak{k}=\mathfrak s+\mathfrak r$ where 
$\s$ is a semisimple subalgebra of $\mathfrak k$ and $\mathfrak r$ is the radical. Let $\varphi:\sl_2\to\mathfrak{k}$ be some embedding and $\bar\varphi$ is the induced embedding $\bar \varphi:\sl_2\to\s$. It suffices to show that $\varphi$ is conjugate to $\bar\varphi$.

For any $u\in\sl_2$, denote $\psi(u):=\varphi(u)-\bar\varphi(u)$. 

Consider first the case when $\mathfrak r$ is abelian. Then $\psi$ is a $1$-cocycle in the Chevalley-Eilenberg complex computing $H^1(\sl_2,\mathfrak r)$, with the $\sl_2$-module structure on $\mathfrak r$ defined by $u\mapsto\ad_{\bar{\varphi}(u)}$. Since $H^1(\sl_2,\mathfrak r)=0$, there exists $v\in\mathfrak r$ such that $\psi(u)=[\bar{\varphi}(u),v]$ for all $u\in\sl_2$. Then $$(\exp\ad_v)(\varphi)=\varphi+[v,\varphi]=\bar{\varphi}.$$
In general, we prove the statement by induction on dimension of $\mathfrak r$ in the same way as it is done in the proof of Levi's theorem. Using the above argument for $\mathfrak k/[\mathfrak r,\mathfrak r]$ we may assume that the image of $\varphi$ lies in $\s+[\mathfrak r,\mathfrak r]$ and complete the proof using the induction assumption. 
    
\end{proof}
\begin{corollary}\label{cor:balanced} Let $x_0\in\g^{hom}$ be a homological element. There are finitely many (up to $ G_{\bar 0}$-conjugation) strongly balanced elements $x\in\g_{\bar 1}$ with $x_0\in att(x)$.
\end{corollary}
\begin{proof} Let $x$ and $x'$ be elements with a common attractor $x_0$.
Consider the Jordan decompositions $\frac{1}{2}[x,x]=s+n$ and $\frac{1}{2}[x',x']=s'+n'$. We have
$s=s'=\frac{1}{2}[x_0,x_0]$ and $x, x'\in \g^s$. So it suffices to check that there are finitely many, up to $C_{G_{\bar 0}}(s)$-conjugation, strongly balanced elements $x\in\g^s$ with the given attractor $x_0$. Moreover, we may pass to the quasi-reductive quotient $\g^s/\kk s$ of $\g^s$. Each conjugacy class in $\g^s_{\bar 1}$ corresponds to a single conjugacy class in $\g_{\bar 1}$. The images of the elements $x, x'$ in $\g^s/\kk s$ are strongly balanced $\ad$-nilpotent elements, so we reduce the statement to \cref{cor:balancednil}.
\end{proof}

\section{Distinguished orbits}\label{sec:distinguished}

\subsection{Definition of a distinguished element}

\begin{definition}\label{def:distinguished}
An element $x\in\g$ (respectively, the orbit $G_{\bar 0}.x$) is called {\it distinguished} if $\fc_{\g}(x)^{\mathbf{ss}}=Z(\g)_{\bar 0}$. 

\end{definition}

 This definition is analogous to the definition of a distinguished nilpotent element in a reductive Lie algebra. The above definition can be rephrased as follows: $x$ is distinguished iff $x $ is not contained in any proper Levi sub-superalgebra of $\g$ (see \cref{def:parabolic_subalgebra}).

\begin{remark}
    If $\g=\g_{\bar 0}$ is a semisimple (purely even) Lie algebra, this definition coincides with the definition of a distinguished element in \cite{collingwood1993nilpotent}: an element $x$ in a semisimple Lie algebra is called distinguished if its centralizer in $\g$ contains only nilpotent elements (and it's automatically nilpotent). 
    %This is due to the fact that $\Der(\g)\cong \g$ for any semisimple Lie algebra $\g$. However, if $\g=\g_{\bar 0}$ is a reductive (purely even) Lie algebra with a non-trivial center, our definition means that $\g$ has no distinguished elements at all.
\end{remark}

We will discuss examples of distinguished orbits in \cref{sec:orbits_in_classical}.

\begin{lemma}\label{lem:dist-jor} Let $x\in\g_{\bar 1}$ be a distinguished element in $\g$. Let
$\frac{1}{2}[x,x]=s+f$ be the Jordan decomposition of $\frac{1}{2}[x,x]$. Then $s\in Z(\g)$.

Furthermore, if $Z(\g)\cap[\g,\g]_{\bar 0}=0$ then $s=0$, hence $x\in \mathcal{N}^{\ad}(\g_{\bar 1})$.    
\end{lemma}
\begin{proof} The first assertion follows from the fact that $x$ commutes  with $\frac{1}{2}[x,x]$ so $\ad_{s}(x)=0$.

To prove the second assertion, it suffices to show that $s\in [\g,\g]$. Let $x_0$ be the attractor of $x$. By \cref{lem:square_attractor} we have:
$s=[x_0,x_0]$, as required.
\end{proof}

Just like in the Lie algebra case (see \cite[Theorem 8.1.1]{collingwood1993nilpotent}), any two minimal Levi subalgebras containing an $\ad$-nilpotent element $x$ are conjugate:

\begin{lemma}\label{lem:uniqlevi}
Let $x\in \g_1$ and let $\g', \g''$ be two minimal Levi subalgebras of $\g$ containing $x$. Then $\g'$ and $\g''$ are $C_{G_{\bar 0}}(x)$-conjugate,
and $x$ is distinguished in both $\g'$ and $\g''$.  
\end{lemma}
\begin{proof} Let $\g'=:\g^{s'}$ and $\g''=:\g^{s''}$ for  $s',s'' \in \g^{\mathbf{ss}}$. 
Since $C_{G_{\bar 0}}(x)$ (the centralizer of $x$ in $G_{\bar 0}$) is an algebraic group, all maximal toral subalgebras in its Lie algebra
are conjugate. Thus, without loss of generality we may assume that ${s'}, {s''}$ lie in the same maximal toral subalgebra $\mathfrak t\subset \mathrm{Lie} ~C_{G_{\bar 0}}(x)$. In particular, $[{s'}, {s''}]=0$. For a generic value of $\varepsilon\in \kk$
we have: $s'+\varepsilon s'' \in \mathfrak{t}\subset \g^{\mathbf{ss}}_{\bar{0}}$ and $\g^{\mathfrak{t}}\subset \g^{s'+\varepsilon s''}\subset \g^{s'}\cap \g^{s''}$. By the minimality of $\g^{s'}$, $\g^{s''}$ we conclude that $\g^{s'}=\g^{s''}=\g^\mathfrak{t}$.

To prove the second assertion, assume that $x$ is not distinguished in $\g^{s'}$. Then there exists a non-central semisimple $u\in\g^{s'}_{\bar 0}$ which lies in the centralizer of $x$. Then $x\in \g^{s'}\cap \g^u$ and that contradicts minimality of $\g^{s'}$.
 \end{proof}

The following straightforward statements will be useful for finding distinguished odd orbits in Lie superalgebras.

\begin{lemma}\label{lem:pg_distinguished_orbits}
    Given a quasi-reductive Lie superalgebra $\g$, let $\widetilde{\g}:=\g/Z(\g)_{\bar 0}$. Then distinguished orbits in $\widetilde{\g}_{\bar 1}$ are in bijection with distinguished orbits in $\g_{\bar 1}$.
\end{lemma}

\begin{lemma}
Let $M\in \Rep(\G^{(1|1)})$ and let $M=\bigoplus_i \mathtt{M}_i\otimes \widetilde{V}_i$ be its decomposition into a direct sum of indecomposable $\G^{(1|1)}$-modules, where $\mathtt{M}_i$ is the indecomposable $\G^{(1|1)}$-module with highest vector of even parity and dimension $i+1$ and $\widetilde{V}_i$ is the multiplicity superspace.
Then the quotient of the ring $\End_{\G^{(1|1)}}(M)$ modulo its radical is $ \prod_i \End(\widetilde{V}_i)$.  
\end{lemma}

\begin{proof}
This follows from the fact that $\End_{\G^{(1|1)}}(\mathtt{M}_i, \Pi^a \mathtt{M}_j)=0$ for $j<i$ and $a\in \{\bar 0, \bar 1\}$.
\end{proof}

\begin{lemma}\label{cor:determining_semisimples_comm_with_x}
   Let $G$ be a quasi-reductive supergroup, $\g:=\mathrm{Lie} (G)$ and $x\in \mathcal{N}^{\ad}(\g_{\bar 1})$. Let $M$ be a faithful indecomposable representation of $\g$ and let $M\rvert_x=\bigoplus_i \mathtt{M}_i\otimes \widetilde{V}_i$ be its decomposition as above.
   
   The element $x$ is distinguished if and only if the image of $\mathfrak{c}_{\g_{\bar 0}}(x)^{\ss} \hookrightarrow \prod_i \End(\widetilde{V}_i)$ lies in $\kk \id$. 
\end{lemma}

In particular, we obtain the following necessary condition for being distinguished: 
\begin{corollary}\label{cor:necessary_cond_distinguished}
Let $G$ be a quasi-reductive supergroup, $\g:=\mathrm{Lie} (G)$ and $x\in \mathcal{N}^{\ad}(\g_{\bar 1})$ be distinguished. 

Let $M$ be a faithful indecomposable representation of $\g$. \InnaA{Then for any $s\in \mathfrak{c}_{\g_{\bar 0}}(x)^{\ss} $, the operator $s\rvert_M $ is a scalar operator.}
\end{corollary}

\subsection{Criterion for distinguished neat orbits}

In this subsection we will prove the following proposition:
\begin{proposition}\label{prop:equiv_cond_distinguished}
    Let $\g$ be a quasi-reductive Lie superalgebra and let $x\in \g_{neat}$. Let $\osp_x\subset \g$ be a corresponding $\osp(1|2)$-type subalgebra and let $h \in \osp_x$ be the Cartan element. Let $\g=\bigoplus_{i\in \Z} \g^i$ be the corresponding decomposition into $\ad_h$-eigenspaces. Then the following conditions are equivalent:
    \begin{enumerate}
        \item\label{itm:1_neat_dist} $x$ is distinguished,
        \item\label{itm:2_neat_dist}  $\dim \g^0_{\InnaA{\bar 0}} = \dim \g^1_{\InnaA{\bar 1}}  + \dim Z(\g)_{\bar 0}$,
        \item\label{itm:3_neat_dist}  $\mathfrak{c}_{\g_{\bar 0}}(\osp_x)= Z(\g)_{\bar 0}$, where $\mathfrak{c}_{\g_{\bar 0}}(\osp_x)$ is the centralizer of $\osp_x$ in $\g_{\bar 0}$. 
    \end{enumerate}

\end{proposition}

\InnaA{Before we prove the proposition, let us prove an auxiliary lemma.}

\begin{lemma}\label{lem:centralizer_x_vs_osp_x_ss_elem}
Let $x\in \g_{neat}$ and let $\osp_x\subset \g$ be an $\osp(1|2)$-type subalgebra containing $x$. Then $\mathfrak{c}_{\g_{\bar 0}}(x)^{\mathbf{ss}}=Z(\g)_{\bar 0}$ if and only if $\mathfrak{c}_{\g_{\bar 0}}(\osp_x)^{\mathbf{ss}}=Z(\g)_{\bar 0}$.
\end{lemma}
\begin{proof}
Clearly $Z(\g)_{\bar 0}\subset \mathfrak{c}_{\g_{\bar 0}}(\osp_x)^{\mathbf{ss}} \subset \mathfrak{c}_{\g_{\bar 0}}(x)^{\mathbf{ss}}$ so we only need to check that the ``if'' implication. 

Now, fix any non-central semisimple element $s\in \mathfrak{c}_{\g_{\bar 0}}(x)$. We need to produce a non-central semisimple element in $\mathfrak{c}_{\g_{\bar 0}}(\osp_x)^{\mathbf{ss}}$ to complete the proof.

Let $h\in \osp_x$ be the Cartan element and let $\g=\bigoplus_{i\in \Z} \g^i$ be the $\ad_h$-decomposition of $\g$, with $x$ acting by an odd operator of degree $-1$. Recall that $ \g^{<0}:=\bigoplus_{i<0} \g^i$ is an ideal in $\g$ and all its elements are $\ad$-nilpotent. Hence $\mathfrak{c}_{\g_{\bar 0}}(x)\cap \g^{<0}$ is a nilpotent ideal in $\mathfrak{c}_{\g_{\bar 0}}(x)$.

It is a straightforward observation that $\mathfrak{c}_{\g_{\bar 0}}(\osp_x) = \mathfrak{c}_{\g_{\bar 0}}(x) \cap \g^0$, so $$ \mathfrak{c}_{\g_{\bar 0}}(x) = \mathfrak{c}_{\g_{\bar 0}}(\osp_x) \oplus \left( \mathfrak{c}_{\g_{\bar 0}}(x)\cap \g^{<0} \right).$$  

Consider the decomposition $s = s' +n$ where $s'\in \mathfrak{c}_{\g_{\bar 0}}(\osp_x)$ and $n\in \mathfrak{c}_{\g_{\bar 0}}(x)\cap \g^{<0} $. Let $s''$ be the semisimple part in the Jordan decomposition of $s'$. By \cref{lem:centralizer_Jordan_closed}, $s''\in \mathfrak{c}_{\g_{\bar 0}}(\osp_x)^{\mathbf{ss}}$. So it remains to check that $s''\notin Z(\g)_{\bar 0}$ to prove the assertion. 

Indeed, assume that $s''\in Z(\g)_{\bar 0}$. Then we would have $\ad_{s}=\ad_{(s'-s'')+n}$. But $s'-s''$ is an $\ad$-nilpotent element which lies in the Lie algebra $\mathfrak{c}_{\g_{\bar 0}}(x)$, while $n$ lies in the nilpotent ideal $ \mathfrak{c}_{\g_{\bar 0}}(x)\cap \g^{<0} $ of $\mathfrak{c}_{\g_{\bar 0}}(x)$; so $\ad_{(s'-s'')+n}$ is a nilpotent operator, contradicting the assumption that $\ad_s$ is a non-zero semisimple operator.

\end{proof}

\begin{proof}[Proof of \cref{prop:equiv_cond_distinguished}]

The last two conditions in \cref{prop:equiv_cond_distinguished} are clearly equivalent, since $\dim \g^0_{\InnaA{\bar 0}} - \dim \g^1_{\InnaA{\bar 1}}= \dim \mathfrak{c}_{\g_{\bar 0}}(x)$. 

To prove that \eqref{itm:3_neat_dist} implies \eqref{itm:1_neat_dist}, assume that the conditions in \eqref{itm:3_neat_dist} hold.  
By \cref{lem:centralizer_x_vs_osp_x_ss_elem}, $\mathfrak{c}_{\g_{\bar 0}}(x)$ contains no non-central semisimple elements, so $x$ is distinguished. 

Therefore it remains to check that \eqref{itm:1_neat_dist} implies \eqref{itm:3_neat_dist}. Assume that $x$ is distinguished. By \cref{lem:centralizer_x_vs_osp_x_ss_elem}, any semisimple element in $\mathfrak{c}_{\g_{\bar 0}}(\osp_x)$ lies in $Z(\g)_{\bar 0}$. 
%So we only need to check that $\mathfrak{c}_{\g_{\bar 0}}(\osp_x)$ doesn't contain non-central $\ad$-nilpotent elements.

Consider a decomposition of $\osp_x$-modules $\g=\mathfrak{c}_{\g_{\bar 0}}(\osp_x)\oplus \mathfrak m$. Here $ \mathfrak m$ is an $\osp_x$-submodule of $\g$ which contains no trivial $\osp_x$-submodules. We claim that
$$[\mathfrak{c}_{\g_{\bar 0}}(\osp_x),\mathfrak m]\subset \mathfrak m.$$
Indeed, the $\osp_x$-module $[\mathfrak{c}_{\g_{\bar 0}}(\osp_x),\mathfrak m]$ is isomorphic to a submodule of the tensor product $\mathfrak{c}_{\g_{\bar 0}}(\osp_x)\otimes \mathfrak m$. The latter contains no trivial $\osp_x$-submodules, so $[\mathfrak{c}_{\g_{\bar 0}}(\osp_x),\mathfrak m]\subset \mathfrak m$.

Taking the even part of the above decomposition, we obtain \begin{equation}\label{eq:distinguished_neat_aux}
    \g_{\bar 0}=\mathfrak{c}_{\g_{\bar 0}}(\osp_x)\oplus \mathfrak m_{\bar 0} \;\;\;\; \text{ and } \;\;\;\; [\mathfrak{c}_{\g_{\bar 0}}(\osp_x),\mathfrak m_{\bar 0}]\subset \mathfrak m_{\bar 0}.
\end{equation}

 \InnaA{Since $\g_{\bar 0}$ is reductive, we have $\mathfrak{c}_{\g_{\bar 0}}(\osp_x) \subset   Z(\g_{\bar 0})\oplus \left( \mathfrak{c}_{\g_{\bar 0}}(\osp_x)\cap [\g_{\bar 0}, \g_{\bar 0}]\right)$. 
 
 Let} $u\in \mathfrak{c}_{\g_{\bar 0}}(\osp_x)\cap [\g_{\bar 0}, \g_{\bar 0}]$. We will show that $u=0$, \InnaA{proving that $\mathfrak{c}_{\g_{\bar 0}}(\osp_x)\subset Z(\g_{\bar 0}) $.}
 
 By the Jacobson-Morozov theorem one can find a semisimple element $s \in \g_{\bar 0}$ such that $[s,u]=2u$. Write $s=s_1+s_2$
with $s_1\in\mathfrak{c}_{\g_{\bar 0}}(\osp_x)$ and $s_2\in\mathfrak m_{\bar 0}$. By \eqref{eq:distinguished_neat_aux}, we have: $[s_2,u]=0$ and $[s_1,u]=2u$. 
Since $s_1\in \mathfrak{c}_{\g_{\bar 0}}(\osp_x)$, it must be $\ad$-nilpotent, which implies $u=0$. 

Thus, $\mathfrak{c}_{\g_{\bar 0}}(\osp_x)\subset Z(\g_{\bar 0})$. All elements of $Z(\g_{\bar 0})$ are semisimple in a quasi-reductive superalgebra, so we conclude that $\mathfrak{c}_{\g_{\bar 0}}(\osp_x)\cap Z(\g_{\bar 0})\subset Z(\g)_{\bar 0}$ and $\mathfrak{c}_{\g_{\bar 0}}(\osp_x)\cap [\g_{\bar 0}, \g_{\bar 0}] =\{0\} $. Hence $\mathfrak{c}_{\g_{\bar 0}}(\osp_x)=Z(\g)_{\bar 0}$, as required.

\end{proof}

\subsection{Distinguished neat orbits and distinguished nilpotent even orbits}

Let $\mathcal{N}^{\ad}(\g_{\bar 0})$ be the nilpotent cone in the Lie algebra $\g_{\bar 0}$. 

The map $\kappa:\mathcal{N}^{\ad}(\g_{\bar 1}) \to \mathcal{N}^{\ad}(\g_{\bar 0})$, $x\mapsto [x,x]$ interacts nicely with the notion of distinguished elements (see \cref{def:distinguished})\footnote{\InnaA{An element $y\in \g_{\bar 0}$ is distinguished if and only if its semisimple part is central and its nilpotent part is a distinguished nilpotent element.}}, 

\begin{proposition}\label{prop:dist_square_implies_dist}
    
Let $\g$ be a quasi-reductive Lie superalgebra and let $x\in\g_{\bar1}$. 
\InnaA{Assume that at least one of the following conditions holds:
\begin{enumerate}
\item $Z(\g_{\bar 0}) = Z(\g)_{\bar 0}$,
    \item $\g_{\bar 0}\subseteq [\g, \g]$.
\end{enumerate}
 Under this assumption,} if $[x,x]$ is distinguished in $\g_{\bar 0}$ then $x$ is distinguished in $\g$.
\end{proposition} 
\begin{proof} 

\InnaA{Let $x\in\g_{\bar1}$ such that $[x,x]$ is distinguished in $\g_{\bar 0}$. Then $\mathfrak{c}^{\mathbf{ss}}(x)\subset \mathfrak{c}^{\mathbf{ss}}([x,x]) \subset Z(\g_{\bar 0})$; to show that $x$ is distinguished, we need to show that $\mathfrak{c}^{\mathbf{ss}}(x)\subset Z(\g)_{\bar 0}$. Clearly, if $Z(\g_{\bar 0}) = Z(\g)_{\bar 0}$ then we are done. So from now on we will assume that $\g_{\bar 0}\subseteq [\g, \g]$.}

Without loss of generality we may assume that $Z(\g)_{\bar 0}=0$. 

Let $\mathfrak{i}(\g)$ be the sum of the minimal ideals in $\g$ \InnaA{and $\rr:=\g/\mathfrak{i}(\g)$, as described in \cref{ssec:quasired}. Then $\rr_{\bar 1}$ is an abelian ideal in $ \rr$ and $\rr_{\bar 0}$ is a (purely even) reductive Lie algebra . Furthermore, we have an isomorphism $\g_{\bar 0}\cong \mathfrak{i}(\g)_{\bar 0} \times \rr_{\bar 0}$ of Lie algebras.

Since $x\in \g_{\bar 1}$ and $\rr_{\bar 1}$ is an abelian ideal, we have: $[x,x]\in \mathfrak{i}(\g)$. The element  $[x,x]$ is distinguished in $\g_{\bar 0}$, so the decomposition $\g_{\bar 0}\cong \mathfrak{i}(\g)_{\bar 0} \times \rr_{\bar 0}$ implies: $ \rr_{\bar 0}$ contains no non-central semisimple elements. Hence $\rr_{\bar 0}$ is abelian, which implies: $[\rr, \rr]_{\bar 0}=0$. Thus $[\g,\g]_{\bar 0}\subset\mathfrak{i}(\g)_{\bar 0}$. The assumption $\g_{\bar 0}\subseteq [\g, \g]$ implies that $ \mathfrak{i}(\g)_{\bar 0} = \g_{\bar 0}$, so we have a decomposition of $\g_{\bar 0}$-modules $\g=\mathfrak{d}\oplus\mathfrak{i}(\g)$, where $\mathfrak{d}$ is a purely odd abelian subalgebra acting trivially on $\g_{\bar 0}$.
}
\begin{comment}
  Let $\mathfrak{t}':=\g_{\bar 1}+\mathfrak{t}$. This is the preimage of $(\g/\ft)_{\bar 1}$ under the quotient map $\g\to \g/\ft$, so it is an ideal in $\g$.

Additionally, $\g/\mathfrak{t}'$ is a (purely even) reductive Lie algebra. The ideal $\mathfrak{t}':=\g_{\bar 1}+\mathfrak{t}$ has a $\g_{\bar 0}$-decomposition
$\mathfrak{d}\oplus\mathfrak{t}$ where $\mathfrak{d}$ is an odd abelian superalgebra acting trivially on $\ft'_{\bar 0}$ (see \cref{ssec:quasired}).

Since $x\in \ft'$, we also have: $[x,x]\in\mathfrak{t}'$. This element is distinguished in $\g_{\bar 0}$, so the semisimple part of  $\g/\mathfrak{t}'$ is trivial. Therefore the Lie algebra $\g/\mathfrak{t}'$ is abelian.
Then $[\g,\g]_{\bar 0}\subset\mathfrak{t}'_{\bar 0}$. The assumption $\g_{\bar 0}\subseteq [\g, \g]$ implies
that $\mathfrak{t}'=\g=\mathfrak{d}\oplus\mathfrak{t}$ where $\mathfrak{d}$ acts trivially on $\g_{\bar 0}$.  
\end{comment}
 Recall that $\mathfrak{i}(\g)=\bigoplus_i \mathfrak{t}_i$ where each $\mathfrak{t}_i$ is one of the following:
\begin{itemize}
    \item a simple Lie superalgebra,
    \item a purely odd abelian superalgebra,
    \item a superalgebra of the form $\s\otimes\kk[\xi]$ where $\s$ is a simple Lie algebra and $\xi$ is an odd variable.
\end{itemize} 

Let $s$ be a semisimple element in the centralizer $\fc_{\g_{\bar  0}}(x)$. In order to show that $x$ is distinguished, we wish to show that $s=0$. Since $s\in \fc_{\g_{\bar  0}}(x)$ we have: $[s,[x,x]]=0$. The assumption that $[x,x]$ is distinguished implies: $s\in Z(\g_{\bar 0})$. Furthermore, since $\mathfrak{d}$ acts trivially on $\g_{\bar 0}$, we have $[s,\mathfrak{d}]=0$.

Write $s=\sum_i s_i$, $x=u+\sum x_i$ with $s_i,x_i\in \mathfrak{t}_i$ and $u\in \mathfrak{d}$. The above arguments imply that $s_i\in Z((\mathfrak{t}_i)_{\bar 0})$ and $[s, u]=0$. Recall that $Z((\mathfrak{t}_i)_{\bar 0})\neq Z(\mathfrak{t}_i)_{\bar 0} =0$ only if $\mathfrak{t}_i\cong\sl (m|n)$ with $m\neq n$ or $\mathfrak{t}_i\cong\mathfrak{osp}(2|2n)$. So we may have $s_i\neq 0$ only in these cases.

Now, we have:
$$0=[s,x]=[s,u]+\sum_{i, j} [s_i, x_j]=0+\sum_{i} [s_i, x_i]=\sum_{i} [s_i, x_i].$$ Since $[s_i, x_i]\in \ft_i$ for every $i$, we have: $[s_i, x_i]=0$.

The condition that $[x,x]$ is distinguished implies that for each $i$, the element $[u+x_i,x_i]$ is distinguished in the ideal $\mathfrak{t}_i$.

Let us consider only the indices $i$ for which $\mathfrak{t}_i\cong\sl (m|n)$ ($m\neq n$) or $\mathfrak{t}_i\cong\mathfrak{osp}(2|2n)$. In particular, $x_i\neq 0$, and every non-zero element of $Z((\mathfrak{t}_i)_{\bar 0})$ acts invertibly on $(\mathfrak{t}_i)_{\bar 1}$, implying that $s_i=0$ for these indices as well.
%In other words, $s_i$ is a central element of $(\mathfrak{t}_i)_{\bar 0}$ commuting with a non-zero element of $(\ft_i)_{\bar 1}$. Yet if $\mathfrak{t}_i\cong\sl (m|n)$ ($m\neq n$) or $\mathfrak{t}_i\cong\mathfrak{osp}(2|2n)$, 
So $s_i=0$ for every $i$, implying that $s=0$ as required.

\end{proof}

\begin{example}
In general, an element $x\in \g_{neat}$ being distinguished does not imply that $f:=[x,x]$ is distinguished. In type $A$ this implication holds, as can be seen from \cref{ex:dist_elem_type_A}. 

Let $\g:=\osp(4|4)$ and let $V=\kk^{4|4}$ be the tautological representation of $\g$. % with the basis $B=(e_1, e'_1, e_2, e'_2, e_3, e_3', e_4, e'_4)$ . 
Let $x\in \gl(4|4)_{\bar 1}$ be given by the checkered diagram $\mathbf{D}^x$ below:
    
\[\ytableausetup{centertableaux}
\mathbf{D}^x = \begin{ytableau}
*(gray) & *(white) &*(gray)  &*(white) &*(gray)\\
*(white)  & *(gray) & *(white) 
\end{ytableau}\]

Then $V\rvert_{\osp_x} = \MM(2)\oplus \Pi \MM(1)$ (see \cref{ssec:prelim_osp} for notation),
and so $x$ is distinguished (cf. \cref{ssec:dist_in_osp_pe}). Now, $f\rvert_{V_{\bar 1}}$ is given by the Young diagram 

\[\ytableausetup{centertableaux}
 \begin{ytableau}
*(white) &*(white) \\
*(white)  & *(white) 
\end{ytableau}\]

This diagram has rows of equal length, \InnaA{so it represents a non-distinguished nilpotent orbit in $\mathfrak{sp}_4$}. This implies that $f$ is not distinguished in $\osp(4|4)_{\bar 0}$.
\end{example}

\section{Examples: odd orbits in simple quasi-reductive Lie superalgebras}\label{sec:orbits_in_classical}

%\mbox{}
\subsection{Dichotomy}
In this section, we will describe classifications of odd orbits in simple quasi-reductive Lie superalgebras (see the list of such superalgebras in \cref{ssec:classical_list}). In particular, we will prove (case-by-case) the following dichotomy:

\begin{theorem}\label{thrm:dist_is_neat_or_balanced}
    Let $\g$ be a simple quasi-reductive Lie superalgebra or a Takiff Lie superalgebra, such that $\g\not \cong \mathfrak{spe}(n)$ for any $n$. Let $x\in \g_{\bar 1}$ be a distinguished odd element. Then $x$ is either strongly balanced or neat.
\end{theorem}

We mention here a straightforward observation which we will use:
given a surjective homomorphism of Lie superalgebras $q: \g\twoheadrightarrow \g'$, we have
\begin{itemize}
    \item $x\in \g_{bal}$ (respectively, $x$ is strongly balanced) $ \; \Longrightarrow \;$ $ q(x)\in \g'_{bal}$ (respectively, $q(x)$ is strongly balanced).
    \item $x\in \g_{neat} \; \Longrightarrow \; q(x)\in \g'_{neat}$.
\end{itemize}
\subsection{Checkered Young diagrams}
To describe the nilpotent orbits in Lie sub-superalgebras of $\gl(m|n)$, we define a checkered Young diagram associated with a nilpotent odd operator $x \in \End(\kk^{m|n})_{\bar 1}$:

\begin{definition}\label{def:checkered_Young_diagram}
A {\it checkered Young diagram} is an equivalence class of Young diagrams whose cells are colored in black and white, the colors alternating in each row, under the equivalence relations: $\lambda \sim \mu$ if $\lambda$ is obtained from $\mu$ by a permutation of the rows. A checkered Young diagram is said to be of type $(m|n)$ if it contains $m$ black boxes and $n$ white boxes.
\end{definition}
\begin{comment}
    \begin{notation}
    Given a Young diagram $D$ (checkered or not), we denote by $c_i(D)$ ($i\geq 1$) the lengths of its $i$-th column and by $r_i(D)$ ($i\geq 1$) the length of its $i$-th row.
\end{notation}
\end{comment}

\begin{definition}\label{def:checkered_diagram}

Let $x\in \End(\kk^{m|n})_{\bar 1}$ be an nilpotent operator. The {\it checkered Young diagram $\mathbf{D}^{x}$} is a checkered Young diagram obtained by taking the sizes of the Jordan blocks of $x$. 

Each row of the diagram corresponds to a single Jordan block of $x$, with the leftmost box corresponding to the eigenvector annihilated by $x$. 

The coloring of the cells in the diagram is as follows: each cell of the diagram is colored either in black (if the corresponding vector in $\kk^{m|n}$ is even) or in white (if the corresponding vector in $\kk^{m|n}$ is odd). 
\end{definition}
Two adjacent boxes in the same row will always have different colors.

\begin{example}\label{ex:diagrams_for_gl_2_3}
Let $\{e_1, e_2, e'_1, e'_2, e'_3\}$ be the standard basis in the vector superspace $\kk^{2|3}$. We have:
\[x=\begin{bmatrix}
0 &0 &\vline&1 &0 &0\\
0 &0 &\vline&0 &0 &0\\
\hline
0 &0 &\vline&0 &0 &0\\
1 &0 &\vline&0 &0 &0\\
0 &0 &\vline&0 &0 &0\\
\end{bmatrix} \; \Longrightarrow \;\ytableausetup{centertableaux}
\mathbf{D}^x = \begin{ytableau}
*(white)  & *(gray) & *(white)  \\
*(gray) \\
*(white)
\end{ytableau}.\]
\begin{comment}
    $$\begin{bmatrix}
0 &0 &0 &0 &0\\
1 &0 &0 &0 &0\\
0 &1 &0 &0 &0\\
0 &0 &0 &0 &0\\
0 &0 &0 &0 &0\\
\end{bmatrix}$$
\end{comment}
\end{example}

\subsection{Nilpotent orbits in type A}\label{ssec:nilp_orbits_type_A}
\mbox{}

Let $G$ be the supergroup $GL(m|n)$ or $SL(m|n)$ and let $\g:=\Lie (G)$.

The checkered diagrams defined above parameterize nilpotent orbits in $\g_{\bar 1}$. That is, two elements $x, y\in \mathcal{N}^{\ad}(\g_{\bar 1})$ are $G_{\bar 0}$-conjugate iff they correspond to the same checkered Young diagram (as before, up to a permutation of rows).

The following lemma is  a straightforward corollary of \cref{thrm:JM,prop:balanced_in_gl}:
\begin{lemma}\label{lem:neat_and_balanced_diagrams_type_A}
Let $x\in \mathcal{N}^{\ad}(\g_{\bar 1})$. Then
\begin{itemize}
    \item $x$ is neat iff all the rows of $\mathbf{D}^{x}$ have odd length.
    \item $x$ is (strongly) balanced iff each row of $\mathbf{D}^{x}$ has either even length or length $1$.
\end{itemize}
\end{lemma}

\subsubsection{Distinguished orbits in \texorpdfstring{$\gl(m|n)$}{gl} and \texorpdfstring{$\sl(m|n)$, $m\neq n$}{sl}}\label{ex:dist_elem_type_A}

Let $\g$ be the Lie superalgebra $\gl(m|n)$ or $\sl(m|n)$ ($m\neq n$).
Let $x\in \g_{\bar 1}$ be a distinguished odd element. Let $V:=\kk^{m|n}$ be the tautological representation of $\gl(m|n)$ and consider the action of $\kk[x]$ on $V$ given by $x\rvert_V$.

 In this case, $Z(\g)\cap [\g, \g]_{\bar 0}\InnaA{=\{0\}}$, so by \cref{lem:dist-jor}, all the distinguished elements are $\ad$-nilpotent; so we assume from now on that $x$ is $\ad$-nilpotent. 

 Furthermore, by \cref{cor:determining_semisimples_comm_with_x}, we have:
 \begin{lemma}
  \InnaA{Let $x\in \mathcal{N}^{\ad}(\g_{\bar 1})$. The element $x$ is distinguished iff $V\rvert_{\kk[x]}$ is indecomposable; equivalently, $x$ is distinguished iff $\mathbf{D}^x$ has exactly one row.    }
 \end{lemma}
Using \cref{lem:neat_and_balanced_diagrams_type_A} we obtain:
 
 \begin{corollary}
    The Lie superalgebra $\gl(m|n)$ has a distinguished odd $\ad$-nilpotent orbit if and only if $|m -n|\leq 1$. 
    \begin{enumerate}
        \item This orbit is neat iff $|m-n|=1$; in that case, for any $x$ in this orbit, $f:=[x,x]$ is a regular nilpotent element in $\gl(m)\times \gl(n)$.

\item If $m=n$, there are precisely $2$ distinguished orbits, given by a checkered Young diagrams with one row (these diagrams are obtained one from another by a swap of colors). All the elements in these orbits are strongly balanced.
    \end{enumerate}
    The Lie superalgebra $ \sl(m|n)$ ($m\neq n$) has a distinguished odd $\ad$-nilpotent orbit if and only if $|m -n|= 1$, and this orbit is neat. 
 \end{corollary}

\subsubsection{Distinguished orbits in $\sl(n|n),\p\sl(n|n)$}
Let $\g:=\sl(n|n)$ and denote by $V=\kk^{n|n}$ the defining representation of $\sl(n|n)$. 
\begin{comment}
    Recall that $\mathrm{Der}^{\bullet}(\sl(n|n))=\gl(n|n)$.
\end{comment}

\begin{lemma}\label{lem:sl-dist} Let $x\in \g_{\bar 1}$ be distinguished
and let $\frac{1}{2}[x,x]=s+f$ be the Jordan decomposition of $\frac{1}{2}[x,x]$.
\begin{enumerate}
    \item $s=\lambda\id_V$ for some $\lambda \in \kk$.
    \item If $s\neq 0$, then $f$ is a principal nilpotent element of $\g_{\bar 0}$ embedded into a principal $\sl_2$-triple $(e,h,f)$ and we have a decomposition $x=x_0+x_{-2}$ such that $[h,x_i]=ix_i$ for every $i$. In particular, $x$ is strongly balanced.
    \item If $s=0$, there are two cases: 
    \begin{itemize}
        \item either $V$ is an indecomposable $\kk[x]$-module and $x$ is strongly balanced,
        \item or $V=U_1\oplus U_2$ where $U_1, U_2$ are indecomposable $\kk[x]$-modules of odd dimensions and $x$ is neat.
    \end{itemize} 
\end{enumerate}

\end{lemma}
 \InnaA{By \cref{prop:dist_square_implies_dist}, the above conditions on $s,f$ are also sufficient to ascertain that the element $x$ is distinguished.}
\begin{proof}
\begin{enumerate}

\item By \cref{lem:dist-jor} we have: $s\in Z(\g)_{\bar 0}$, so $ s=\lambda\id_V$.

   \item Write $x=\begin{pmatrix}0&B\\ C&0\end{pmatrix}$. \InnaA{Since $s\neq 0$, the matrices} $B,C$ are non-degenerate,
    therefore, using the action of $G_{\bar 0}$ we can reduce the situation to the case where $B=\mu 1_n$ for some $\mu\in \kk^{\times}$
    and $C$ coincides with its canonical Jordan form. 
    Clearly, if $C$ has more than one Jordan block,
    then it commutes with some semisimple non-central element $A\in \gl_n$; in that case, $x$ is not distinguished \InnaA{since it commutes with $\begin{pmatrix}
        A &0\\ 0 &A
    \end{pmatrix}$}.
    
    Hence we may write $C=\nu (1_n+J)$ where $J$ is the nilpotent Jordan block of size $n$ and $\nu \in \kk$. We then obtain:
    $\mu\nu=\lambda$. Thus we have $\frac{1}{2}[x,x]=\lambda\left(\begin{matrix}C&0\\ 0&C\end{matrix}\right)$.
    In particular, $f=\lambda J$ is a principal nilpotent element in $\g_{\bar 0}$. Taking the pricipal $\sl_2$-triple $(e,h,f)$  in $\g_{\bar 0}$, we have a decomposition $x=x_0+x_{-2}$ into $\ad_h$-eigenvectors, where $$x_0=\left(\begin{matrix}0&\mu 1_n\\ \nu 1_n&0\end{matrix}\right),\;\;\; x_{-2}=\left(\begin{matrix}0&0\\ \nu J&0\end{matrix}\right).$$
    
    \item Assume that $s=0$, so $x\in \mathcal{N}^{\ad}(\g_{\bar 1})$. Let $V\rvert_{\kk[x]} = \bigoplus_i M_i\otimes V_i $ be the decomposition of $V$ into a direct sum of indecomposable $\kk[x]$-modules, with $V_i$ denoting the multiplicity space (purely even) of the indecomposable $\kk[x]$-module $M_i$. 
    By \cref{cor:determining_semisimples_comm_with_x}, the elements of $ \fc_{\g_{\bar 0}}(x)^{\ss}$ act by a scalar on $V$. For any collection $s_i\in \End(V_i)$ of semisimple operators, we may consider the (even) semisimple endomorphism $\phi=\oplus_i s_i \id_{M_i}$ of $V$. If $str(\phi)=\sum_i tr(s_i)\sdim M_i=0$ then \cref{cor:determining_semisimples_comm_with_x} states that $\phi\in \kk\id_V$. In particular, we conclude:
    \begin{itemize}
        \item If $V$ is decomposable, then $V_i=0$ whenever $\sdim M_i=0$ (otherwise we could take $s_j:=\delta_{i, j}\id_{V_j}$ and obtain $\phi\not\in \kk\id_V$). 
   \item We have $\dim V_i\leq 1 $ for all $i$ (otherwise one can choose a traceless endomorphism $s_i \in \End(V_i)$, set $s_j:=0$ for $j\neq i$ and obtain $\phi\not\in \kk\id_V$). 
   \item If $V_i, V_j\neq 0$ and $\sdim M_i=\sdim M_j \neq 0$, then $i=j$ (otherwise we may take $s_i:=\id$, $s_j:=-\id$ and $s_k:=0$ for all $k\neq i, j$ and obtain $\phi\not\in \kk\id_V$). 
    \end{itemize}

    As a consequence, $V$ cannot have three distinct indecomposable direct summands $M_i, M_j, M_k$: indeed, these would have to be of superdimension $\pm 1$  due to the first conclusion; by the pigeonhole principle, two of them would have equal superdimensions, contradicting the last conclusion.

    We conclude that there are two possible cases:
    \begin{itemize}
        \item {\it $V$ is an indecomposable $\kk[x]$-module}. In this case, $x$ is conjugate either to the matrix 
$\left(\begin{matrix}0&\lambda 1_n\\ J&0\end{matrix}\right)$ or to $\left(\begin{matrix}0&J\\ \lambda 1_n&0\end{matrix}\right)$. Thus one can write a decomposition $x=x_0+x_{-2}$ into $\ad_h$-eigenvectors, with 
\begin{align*}
&x=\left(\begin{matrix}0&\lambda 1_n\\ J&0\end{matrix}\right) \;\; \Longrightarrow \;\;x_0=\left(\begin{matrix}0&\lam 1_n\\ 0&0\end{matrix}\right),\;\;\; x_{-2}=\left(\begin{matrix}0&0\\ J&0\end{matrix}\right), \\
&x=\left(\begin{matrix}0&J\\ \lambda 1_n&0\end{matrix}\right) \;\; \Longrightarrow \;\;x_0=\left(\begin{matrix}0&0\\ \lambda 1_n&0\end{matrix}\right),\;\;\; x_{-2}=\left(\begin{matrix}0&J\\0 &0\end{matrix}\right).
\end{align*} In particular, in this case $x$ is strongly balanced.
        \item {\it $V$ is a direct sum of two indecomposable $\kk[x]$-modules of odd dimensions (and opposite superdimensions).} In this case $x$ is neat by \cref{thrm:JM}.
    \end{itemize}

    \end{enumerate}
\end{proof}

\InnaA{By \cref{lem:pg_distinguished_orbits}, the odd distinguished elements in $\p\sl(n|n)$ are the same as in $\sl(n|n)$.}
\begin{corollary}\label{cor:psl-dist} Every odd distinguished element in $\sl(n|n) , \p\sl(n|n)$ is either strongly balanced or neat. 
    \end{corollary}
    
    \begin{remark}
        Let $\g=\sl(n|m)$ or $\p\sl(n|n)$. For each $x_0\in \g^{hom}$ there exists at most one distinguished orbit $G_{\bar 0}.x\subset \g_{\bar 1}$ such that $att(x)=G_{\bar 0}.x_0$.
    \end{remark}

\subsection{Nilpotent orbits in orthogonal and periplectic Lie superalgebras}
\subsubsection{Nilpotent operators in the presence of a symmetric bilinear form}\label{ssec:nilp_in_presence_bilinear_form}

In this subsection, we will consider a non-degenerate symmetric bilinear form $B$ on a finite-dimensional vector superspace $V$. The form $B$ might be even (so $B:V\times V\to \kk$) or odd (so $B:V\times V\to \Pi \kk$). 

Let $x\in \End^{\bullet}(V)$ be an odd nilpotent operator preserving (infinitesimally) the form $B$: that is, 
$$B(xu, v)=-(-1)^{\bar {u}} B(u, xv) $$ for any homogeneous $u, v \in V$.
The goal of this subsection is to explain how $V$ decomposes into a direct sum of mutually orthogonal $\kk[x]$-submodules.

\begin{lemma}\label{lem:form-nilp} 
Let $V$ be a vector superspace with a non-degenerate symmetric bilinear form $B$. Let $x\in\End^{\bullet}(V)$ be an odd nilpotent operator preserving the form $B$. 
The restriction of $V$ to $\kk[x]$ has an orthogonal decomposition
\begin{equation}\label{eq:orthog_decomp}
    V=\bigoplus_{i=1}^p U_i\oplus\bigoplus_{j=1}^q W_j,
\end{equation}
where every $U_i$ is an indecomposable $\kk[x]$-module and every $W_j=W_j'\oplus W_j''$ is a direct sum of two indecomposable $\kk[x]$-modules such that $W'_j\cong (W''_j)^*$ if $B$ is even and $W'_j\cong\Pi (W''_j)^*$ if $B$ is odd.

\end{lemma}

\begin{proof} The proof goes by induction on $\dim V$. Let $r$ be maximal such that $x^r\neq 0$. Since $B$ is $x$-invariant, we have $\Im x^r=(\Ker x^r)^{\perp}$. Therefore $B$ defines a non-degenerate pairing $(V/\Ker x^r)\times \Im x^r\to \kk$. 
Choose $v\notin\Ker x^r$ and let $c=B(v,x^r v)$. We will now explain how to ``split off" a $\kk[x]$-submodule of $V$ containing $v$. 

First, assume that $c\neq 0$. Then for all $0\leq k\leq r$ we have:
\begin{equation*}
    B(x^k v,x^{r-k}v)=(-1)^{k\bar v+\frac{k(k+1)}{2}}c.
\end{equation*}
If $B$ is even (resp., odd) then $r$ is even (resp., odd). Let $U$ denote the indecomposable $\kk[x]$-submodule generated by $v$. The restriction of $B$ on $U$ is non-degenerate: the radical of $B\rvert_{U\times U}$ must be a $\kk[x]$-submodule of $U$, but $B$ pairs $soc(U)=\mathrm{span}\{x^rv\}$ with $ \mathrm{span}\{v\}$.
%Indeed, let $u=\sum_{i=1}^r a_i x^iv$ and $s$ is minimal with $a_s\neq 0$. Then $B(u,x^{r-s}v)=\pm c a_s\neq 0$ since $B(x^p v,x^q v)=0$ for $p+q>r$. 
We now decompose $V=U\oplus U^\perp$ and apply the induction assumption to $U^\perp$.

Now assume that $c=B(v,x^r v)=0$. There exists $w\notin\Ker x^r$ such that $B(w,x^r v)=1$. We also can assume that $B(w,x^rw)=0$ because otherwise we can reduce to the previous case (taking $w$ instead of $v$).
We have the relations
\begin{equation*}
B(x^k w,x^{r-k}v)=(-1)^{k\bar v+\frac{k(k+1)}{2}}.
\end{equation*}
Let $W$ be the $\kk[x]$-submodule generated by $v$ and $w$. Again, the restriction
of $B$ on $W$ is non-degenerate: the radical of $B\rvert_{W\times W}$ must be a $\kk[x]$-submodule of $W$, yet $B$ pairs non-degenerately $soc(W)=\mathrm{span}\{x^rv, x^r w\}$ with $ \mathrm{span}\{v, w\}$. 

Letting $W':=\kk[x]v$, $W'':=\kk[x]w$ we obtain the non-degenerate pairing $B:W'\times W''\to\kk$. 
Therefore $W'$ is isomorphic to $(W'')^*$ if $B$ is even and to $\Pi (W'')^*$ if $B$ is odd.
%Indeed, let $u=\sum_{i=1}^r a_i x^iv+\sum_{i=1}^r b_i x^i w$, $k$ be minimal such that $a_k\neq 0$ and $l$ be minimal such that $a_l\neq 0$. Then if $k\leq l$ we have $B(u,x^{r-k}w)\neq 0$ and if $l<k$ then $B(u,x^{r-l}v)\neq 0$. 
As in the previous case, we write
$V=W\oplus W^\perp$ and apply the induction assumption to $W^{\perp}$.
\end{proof}

We now study the action of $\kk[x]$ on the two types of direct summands appearing in the decomposition \eqref{eq:orthog_decomp}.
\begin{lemma}\label{lem:form_B_indec_U}
Let $U$ be a vector superspace with a non-degenerate symmetric bilinear form $B$. Let $x\in\End^{\bullet}(U)$ be an odd nilpotent operator preserving the form $B$. Assume that $U$ is an indecomposable $\kk[x]$-module. Let $r:=\dim U-1$, so $x^r\neq 0$, $x^{r+1}=0$. 

Then we may choose a $\kk[x]$-generator $v\in U$ such that $B(x^av,x^b v)=0$ if 
$a+b\neq r$ and $B(x^av,x^{r-a} v)\neq 0$ for any 
$0\leq a\leq  r$.

As a consequence, we have:
\begin{itemize}
    \item If $B$ is even then $\dim U \in 2\Z+1$, $r\in 2\Z$, so $x$ is a neat operator. 
    
    Furthermore, $\bar{v}\equiv r/2 \mod 2$.
    \item If $B$ is odd then $\dim U \in 2\Z$, $r\in 2\Z+1$, so $x$ is a balanced operator. 
    
    Furthermore, $\bar{v}\equiv (r+1)/2 \mod 2$.
\end{itemize}
\end{lemma}
\begin{proof}

Let $S:=\{\text{homogeneous } v\in U~:~B(v, x^rv)=1\}$; by \cref{lem:form-nilp}, we have $ S\neq \emptyset$.
Each $v\in S$ is a $\kk[x]$-generator of $U$, since $x^rv\neq 0$ and $U$ is indecomposable. Moreover, since $x$ preserves the form $B$, we have: $$\forall ~v\in S, ~0\leq a\leq r, \;\;B(x^a v, x^{r-a}v)=\pm B(v, x^rv)\neq 0.$$ 

Next, for any $v\in S$, consider the set
$K(v):=\{0\leq k<r~:~ B(v,x^k v)\neq 0\}$ and denote $k(v):=
    \sup K(v)$ (we use the convention $\sup\emptyset:=-\infty$).

We will now explain how to construct a vector $u\in S$ for which $k(u)=-\infty$. Constructing such a vector will prove the lemma, since $u$ will satisfy: 
$$\forall  ~0\leq a,b\leq r, \;\;a+b\neq r\; \implies \;B(x^a v, x^bv)=\pm B(v, x^{a+b}v)= 0.$$ 

We start with $v\in S$ and denote: $k:=k(v)$. If $k=-\infty$, then we are done. 

From now on, assume that $k\geq 0$ and denote: $\beta:=B(v, x^kv)\neq 0$. Note that $\overline{x^r v}=\overline{x^k v}$ due to the homogeneity of $B$, so $r-k\in 2\Z$.

Let $\alpha\in \kk$ and set $v':=v-\alpha x^{r-k}v$. For any $ l\geq 0$ we have:
$$B(v, x^{l}v')=B(v,x^l v)-\alpha B(v, x^{r-k+l}v).$$
Thus $B(v, x^rv')=1$, $B(v,x^l v')=0$ for $k<l<r$ and $B(v,x^k v')=\beta -\alpha $.
It follows that for any $ l\geq 0$ we have:
$$B(v', x^{l}v')=B(v,x^l v')-\alpha B(x^{r-k}v, x^l v')=B(v,x^l v')-\alpha B(v, x^{r-k+l} v')$$
(in the last equality, we use that $r-k\in 2\Z$).
So $ B(v', x^{r}v')=1$, $ B(v', x^{l}v')=0$ for $k<l<r$ and $$ B(v', x^{k}v')=\beta-\alpha-(-1)^{(r-k)\bar {v}} \alpha B(v, x^r v')=\beta-2 \alpha.$$
Taking $\alpha:=\beta/2$, we obtain that $B(v', x^kv')=0$. Hence $v'\in S$ and $k(v')<k(v)$.
We may now replace $v$ by $v'$ and repeat this process.
In the end, we will obtain a new vector $u \in S$ such that $k(v)=-\infty$ as required.
\end{proof}

%Let $\g=\mathfrak{spe}(n)$ or $\osp(m|2n)$ and $V$ denote the defining representation of $\g$. Let $B$ denote a $\g$-invariant symmetric bilinear form on $V$ (odd for $\mathfrak{spe}(n)$ and even for $\osp(m|2n)$).
\begin{lemma}\label{lem:form-nil2} Let $W$ be a vector superspace with a non-degenerate symmetric bilinear form $B$. Let $x\in \End^{\bullet}(W)$ be an odd nilpotent operator preserving the form $B$. Assume that we have a  $\kk[x]$-module decomposition $W=W'\oplus W''$ where $W', W''$ are indecomposable $\kk[x]$-modules such that $B: W'\times W''\to \kk$ is a non-degenerate pairing. 

Let $r:=\frac{1}{2}\dim W-1$ (so $r+1=\dim W'=\dim W''$). Then
\begin{enumerate}
 \item One can choose  $W'$ and $W''$ to be isotropic with respect to $B$.
 \item One can choose generators $v\in W'$, $w\in W''$ such that $B(x^av,x^b w)=0$ if 
$a+b\neq r$ and $B(x^av,x^{r-a} w)\neq 0$ for any 
$0\leq a\leq  r$.
\end{enumerate}
In particular, we have: 
\begin{itemize}
    \item if $B$ is even then $\bar{v}\equiv\bar{w}+r \mod 2$,
    \item if $B$ is odd then $\bar{v}\equiv\bar{w}+r+1 \mod 2$.
\end{itemize}
\end{lemma}
\begin{proof} 
As a start, we note that $soc(W')\perp \Im x$, so $B$ must pair $ soc(W')$ with a subspace of $ cosoc(W'')$, whose parity is $\bar{r}+p$, where $p$ is the parity of $soc(W'')$. Thus $W'\cong \Pi^{(\bar B+\bar r)} W''$ as $\kk[x]$-modules.

{\bf Case 1: $r$ is even}. 

In this case, we have $x\in \g_{neat}$ where $\g:=\osp(W)$ if the form $B$ is even, 
and $\g:=\mathfrak{pe}(W)$ if the form $B$ is odd. We use the superanalogue of the Jacobson--Morozov theorem (see \cref{thrm:JM}): by this theorem, we may consider an $\osp(1|2)$-subalgebra $x\in\osp_x\subset\g$. The form $B$ is $\osp_x$-invariant and
 $W'$ and $W''$ become simple $\osp_x$-modules. 
 
%{\bf Case 1.1: $B$ is an even form}. 

Since $r$ is even, have an isomorphism of $\osp_x$-modules $W' \cong \Pi^{\bar B} W''$. So $W\cong W'\otimes U$ where $U$ a multiplicity space of dimension $(2|0)$ if $B$ is even and of dimension $(1|1)$ if $B$ is odd.
%is a purely even multiplicity space of dimension $2$.
 Decomposing the symmetric square of the tensor product $W^*=W'^*\otimes U^*$, we obtain the identity $$S^2(W^*)=S^2(W'^*)\otimes S^2(U^*)\oplus \Lambda^2(W'^*)\otimes \Lambda^2(U^*).$$
 This implies:
 $$S^2(W^*)^{\osp_x}=S^2(W'^*)^{\osp_x}\otimes S^2(U^*)\oplus \Lambda^2(W'^*)^{\osp_x}\otimes \Lambda^2(U^*).$$
 We have: $B\in S^2(W^*)^{\osp_x}_{\bar 0}$ if $B$ is even and $B\in S^2(W^*)^{\osp_x}_{\bar 1}$ if $B$ is odd.

The simple $\osp_x$-module $W'$ admits a unique, up to rescaling,  non-zero $\osp_x$-invariant even bilinear form $B_{W'}\in (W'^*\otimes W'^*)^{\osp_x}_{\bar 0}$, which is either symmetric or skew-symmetric. This implies that $B=B_{W'}\otimes B_U$, where $B_U$ is a nondegenerate bilinear form on $U$ (the form $B_U$ has the same parity as $B$). 

By the above decomposition of $S^2(W^*)^{\osp_x}$, the form $B_U$ is symmetric if $B_{W'}$ is symmetric and $B_U$ is skew-symmetric if $B_{W'}$ is skew-symmetric.

Whether $B_U$ is symmetric or skew-symmetric, we may choose a homogeneous basis of $U$ consisting of isotropic vectors.
Let us choose two copies of $W'$ in $W=W'\otimes U$ corresponding to this basis and denote these new copies again by $W', W''$. These new subspaces $W', W''$ will then be isotropic.

  \begin{comment}

{\bf Case 1.2: $B$ is an odd form}. In this case, $W'$ is isomorphic to $\Pi W''$ as an $\osp_x$-module and $W\cong W'\otimes U$ where $U$ is a $(1|1)$-dimensional multiplicity space. The previous argument implies that we again have a decomposition $B=B_{W'}\otimes B_U$, where $B_{W'}$ is a non-degenerate $\osp_x$-invariant form on $B_{W'}$, which is unique up to rescaling.
\end{comment}
In both cases (when $B$ is either an even or an odd form), the statement (2) follows from $h$-weight considerations, where $h$ is the Cartan element in $\osp_x$. Indeed, consider the $\osp_x$-actions on its simple modules $W',W''$. For any two $h$-eigenvectors $v \in W',w\in W''$ with respective $h$-eigenvalues $p,q$, we have: $B(v,w)\neq 0$ if and only if $p+q=0$.

{\bf Case 2: $r$ is odd}. 

In that case, $\sdim W'=\sdim W''=0$ and $W'\cong \Pi^{(\bar B+\bar 1)} W''$ as $\kk[x]$-modules.

 Let $f:=\frac{1}{2}[x, x]$. By the usual Jacobson--Morozov theorem, one can place
$f$ inside an $\sl_2$-triple $(e,h,f)$ so that the $\sl_2$-subalgebra $\sl_x$ generated by this triple preserves the form $B$. We will first prove that one can choose $W', W''$ to be isotropic.

{\bf Case 2.1: the form $B$ is even}. 

In this case, $W'$ is isomorphic to $\Pi W''$ as $\sl_x$-modules. Hence 
$\dim \operatorname{coKer}x=(1|1)$.
We decompose $W_{\bar 0}$ into a direct sum of two isomorphic irreducible $\sl_2$-submodules $W_{\bar 0}=Y_1\oplus Y_2$. Furthermore, we may assume that
$Y_1$ and $Y_2$ are isotropic and $Y_1 \not\subset\Im x$.
Let $v$ be the highest weight vector of $Y_1$. Then $v\notin \Im x$ and for any odd $k,l$ we have
$$B(x^k v, x^l v)=\pm B(x^{k-1}v,x^{l+1}v)=B(f^{(k-1)/2}v,f^{(l+1)/2}v)=0.$$
Thus, $W'=\kk[x]v$ is isotropic. Similarly, we may write $W_{\bar 1}=Z_1\oplus Z_2$ as a sum of two isotropic $\sl_2$-submodules, so that $Z_1 \not\subset\Im x$. The highest weight vector $w\in Z_1$ generates an isotropic $\kk[x]$-submodule $W''$. The vectors $v, w$ correspond to a homogeneous basis of $\operatorname{coKer}x$, so we conclude that $W'\oplus W''=W$.

{\bf Case 2.2: $B$ is an odd form}. In this case, $W'$ is isomorphic to $W''$  as $\sl_x$-modules. 
Without loss of generality we may assume that $\operatorname{coKer}x$ is a $2$-dimensional purely even subspace. 
The form $B$ defines a nondegenerate pairing between the highest $h$-weight subspace $\Ker(e \rvert_{W_{\bar 0}})$ and the lowest $h$-weight subspace $\Ker(f \rvert_{W_{\bar 1}})$.
Note that $\Ker(f\rvert_{W_{\bar 1}})=x^r \Ker(e\rvert_{W_{\bar 0}}) $, so the above pairing is determined by an even form $Q$ on $\Ker(e\rvert_{W_{\bar 0}})$, with $Q(v, w):= B(v,x^r w)$ where $v, w\in \Ker(e\rvert_{W_{\bar 0}})$. 

If $v,w\in \Ker(e\rvert_{W_{\bar 0}})$ then
$$Q(v, w)=B(v,x^r w)=(-1)^{(r+1)/2} B(w,x^r v)=(-1)^{(r+1)/2} Q(w,v).$$ Thus, the form $Q$ is either symmetric or skew-symmetric, depending on the parity of $(r+1)/2$. In both cases we can choose a basis $v_1,v_2$ of $\Ker(e\rvert_{W_{\bar 0}})$ which will be isotropic with respect to $Q$.

Let $W'=\kk[x]v_1$ and $W''=\kk[x]v_2$; since $\Ker(e\rvert_{W_{\bar 0}})\cong \operatorname{coKer}x$, we have: $W=W'\oplus W''$. We claim that $W', W''$ are isotropic with respect to $B$. These are isomorphic as $\sl_x$-modules, so they have the same highest weight $m$ (the $h$-weight of $v_1, v_2$). The $h$-weight of the lowest weight vector $x^r v_i$ is then $-m$.

To show that $W'$ is isotropic, it is enough to check that $B(v_i, x^k v_i)=0$ for any $k$. 
Recall that for any two $h$-eigenvectors $v, w \in W'$ with respective $h$-weights $p,q$, we have $B(v,w)\neq 0$ if and only if $p+q=0$. So it is enough to check that $B(v_i, w)=0$ where $w\in W'$ is any vector of $h$-weight $-m$. Recall that the space of odd lowest weight vectors  $\Ker(f\rvert_{W_{\bar 1}}) = \mathrm{span}\{x^rv_1, x^r v_2\}$ is two dimensional, so we just need to check that $B(v_i, x^r v_i)=0$. But this is true:
$$0=Q(v_i, v_i)=B(v_i,x^r v_i).$$
Hence $W'$ is isotropic and so is $W''$.

    It remains to check (2).

Let us choose a basis in $e_1,\dots,e_{r+1}$ in $W'$ such that $xe_i=e_{i+1}$ for all $i\leq r$ and $x e_{r+1}=0$. Let $f_1,\dots, f_{r+1}$ denote the dual basis in $W''$, i.e., $B(e_i,f_j)=\delta_{ij}$.
Then $$B(e_i,x f_j)=-(-1)^{\bar{e_i}}B(e_{i+1}, f_j)=\pm\delta_{i+1,j}. $$
From this we obtain $xf_j=\pm f_{j-1}$ for $j\geq 2$ and $x f_1=0$.
Take $v=e_1$, $w=f_{r+1}$. The statement follows.

\end{proof}

\subsubsection{Nilpotent orbits in \texorpdfstring{$\osp(m|2n)$}{}}
In this case, the nilpotent orbits were studied previously and classified in \cite{gruson2010cones}.
From the results of \cref{ssec:nilp_in_presence_bilinear_form}, we obtain the following description of nilpotent orbits in terms of checkered Young diagrams, recreating the classification of \cite{gruson2010cones}:
\begin{definition}

\mbox{}

\begin{itemize}
    \item A one-row diagram is called {\it basic $\osp$-admissible} if it has an odd length and the number of black boxes is odd.
    \item A two-row diagram is called {\it basic $\osp$-admissible} if the two rows have the same length $l$ and satisfy the following conditions: 
    \begin{itemize}
        \item if $l$ is odd then the two rows are identical, 
        \item if $l$ is even then one row is obtained from another by swapping the colors.
    \end{itemize} 
    \item A checkered Young diagram is called {\it $\osp$-admissible} if it is obtained by stacking one-row and two-row basic $\osp$-admissible diagrams. 
\end{itemize}
\end{definition}
In the following proposition only, we denote by $OSp(m|2n)$ the supergroup associated with the Harish-Chandra pair $(O(n)\times Sp(2m), \osp(m|2n))$ and by $SOSp(m|2n)$ its connected component of the unit.
\begin{proposition}\label{prop:osporbits} 

The {$\ad$}-nilpotent $OSp(m|2n)$-orbits
in $\osp(m|2n)$ are in bijection with the $\osp$-admissible checkered diagrams containing $m$ black boxes and $2n$ white boxes. 

The orbits under the action of the connected subgroup $SOSp(m|2n)$ are the same, unless $m\in 2\Z$ and the number of black boxes in each row is even. In the latter case, the corresponding $OSp(m|2n)$-orbit splits into two $SOSp(m|2n)$-orbits.
\end{proposition}

\subsubsection{Nilpotent orbits in periplectic Lie superalgebras}

We define a $\mathfrak{pe}$-admissible checkered diagram, used to classify nilpotent odd orbits in $\mathfrak{pe}(n)$. 
\begin{definition}

\mbox{}

\begin{itemize}
    \item A one-row diagram is called {\it basic $\mathfrak{pe}$-admissible} if one of the following conditions holds: \begin{itemize}
        \item The row has length $4m+2$ for some $m\in \Z_{\geq 0}$ and the leftmost box is black.
        \item The row has length
    $4m$ for some $m\in \Z_{\geq 1}$ and the leftmost box is white.
    \end{itemize} 
    \item A two-row diagram is called {\it basic $\mathfrak{pe}$-admissible} if the two rows have the same length $l$ and we have: 
    \begin{itemize}
        \item if $l$ is even then the two rows are identical, 
        \item if $l$ is odd then one row is obtained from another by swapping the colors.
        \end{itemize}
  
    \item A checkered Young diagram is {\it $\mathfrak{pe}$-admissible} if it is obtained by stacking one-row and two-row basic $\mathfrak{pe}$-admissible diagrams. 
\end{itemize}
\end{definition} 
From the results of \cref{ssec:nilp_in_presence_bilinear_form}, we obtain the following description of nilpotent orbits in terms of checkered Young diagrams:
\begin{proposition}\label{partitionsP}
The $\ad$-nilpotent orbits in $\mathfrak{pe}(n)$ are in bijection with $\mathfrak{pe}$-admissible checkered diagrams.
\end{proposition}

\begin{corollary}
    Neat orbits in $\mathfrak{pe}(n)$ are in bijection with (colorless) Young diagrams of size $n$ all of whose rows have odd length.  
\end{corollary}

\begin{proof}
Let $x\in \mathcal{N}^{\ad}(\mathfrak{pe}(n)_{\bar 1})$ and consider the $\mathfrak{pe}$-admissible checkered Young diagram $\mathbf{D}^x$ corresponding to $x$. The element $x$ is neat iff all the rows of $\mathbf{D}^x$ are of odd length. By the definition of a $\mathfrak{pe}$-admissible checkered Young diagram, we conclude that $x\in \mathfrak{pe}(n)_{neat}$ iff $\mathbf{D}^x$ is stacked out of basic $\mathfrak{pe}$-admissible two-row diagrams (without using one-row diagrams). The set of such diagrams is in bijection with the set of (colorless) Young diagrams of size $n$ all of whose rows have odd length: namely, each basic $\mathfrak{pe}$-admissible two-row checkered diagram $\mathbf{D}$ corresponds to a single-row Young diagram containing all the black boxes of $\mathbf{D}$; stacking of checkered diagrams corresponds to stacking of single-row Young diagrams.

\end{proof}

\subsubsection{Distinguished odd elements in $\osp(m|2n)$ and $\mathfrak{pe}(n)$, $\mathfrak{spe}(n)$}\label{ssec:dist_in_osp_pe}
\mbox{}

Let $(\g, V)$ be the pair $(\osp(m|2n), \kk^{m|2n})$ or the pair $(\mathfrak{pe}(n), \kk^{n|n})$. The space $V$ is the (faithful) matrix representation of $\g$ and it is equipped with a non-degenerate $\g$-invariant symmetric bilinear form $B$. The form $B$ is even if $\g=\osp(m|2n)$ and odd if $\g=\mathfrak{pe}(n)$.

The Lie superalgebras $\osp(m|2n)$ and $\mathfrak{pe}(n)$ have trivial centers, so all the distinguished odd elements in this setting are $\ad$-nilpotent.

\begin{lemma}\label{lem:form-dist} 
Let $x$ be an $\ad$-nilpotent element in $\g_{\bar 1}$.

The element $x$ is distinguished iff we have a decomposition $V=\bigoplus_{i=1}^p U_i$ where $U_i$ are mutually orthogonal non-isomorphic indecomposable $\kk[x]$-components. 

A distinguished element $x\in \g_{\bar 1}$ is neat 
    for  $\g=\osp(V)$ and balanced for $\g=\mathfrak{pe}(V)$.
\end{lemma}
\begin{proof}
    We use the results of \cref{ssec:nilp_in_presence_bilinear_form} and \cref{cor:determining_semisimples_comm_with_x}. 
    
    Let $x\in \mathcal{N}^{\ad}(\g_{\bar 1})$ and consider the orthogonal decomposition of the $\kk[x]$-module $V$ given in \cref{lem:form-nilp}:
$$
    V=\bigoplus_{i=1}^p U_i\oplus\bigoplus_{j=1}^q W_j.
$$
Here every $U_i$ is an indecomposable $\kk[x]$-module and every $W_j=W_j'\oplus W_j''$ is a direct sum of two isotropic indecomposable $\kk[x]$-modules paired by $B$.

If all the $U_i$'s are pairwise non-isomorphic and there are no summands of the form $W_j$, then by \cref{cor:determining_semisimples_comm_with_x}, the element $x$ is distinguished.

If a summand $W_j$ appears in the above decomposition, we may take an endomorphism of $V$ given by $\lambda\id_{W'_j}\oplus -\lambda\id_{W''_j}$ where $\lambda\neq 0$. This endomorphism lies in $\g\subset \End(V)$, so by \cref{cor:necessary_cond_distinguished}, this implies that $x$ is not distinguished.

Similarly, if $U_i$ appears in the above decomposition with multiplicity at least $2$, we may take the endomorphism $s$ of $V$ given by $\begin{bmatrix}
    0 &\lambda \\ -\lambda &0
\end{bmatrix} \in \End(U_i^{\oplus 2})$.
We obtain $s\in \mathfrak{c}_{\g_{\bar 0}}(x)^{\mathbf{ss}}$. By \cref{cor:determining_semisimples_comm_with_x}, $x$ is not distinguished.

  The remaining part of the lemma now follows from \cref{lem:form_B_indec_U}.
\end{proof}
We now consider distinguished elements in the Lie superalgebra $\g:=\mathfrak{spe}(n)$. Recall that  $\mathfrak{pe}(n)_{\bar 1}=\mathfrak{spe}(n)_{\bar 1}$ and let $V=\kk^{n|n}$ be the matrix representation of $\mathfrak{spe}(n)$. Again, the center of $\mathfrak{spe}(n)$ is trivial, so distinguished odd elements are $\ad$-nilpotent.

\begin{lemma}\label{lem:disting_in_spe}
    Let $\g=\mathfrak{spe}(n)$ and let $x\in \mathcal{N}^{\ad}(\g_{\bar 1})$. The element $x$ is distinguished in $\g$ if and only if the $\kk[x]$-module $V$ decomposes as $V=\bigoplus_{i=1}^p U_i \oplus W$ where 
    \begin{itemize}
        \item $U_i, W$ are mutually orthogonal, pairwise non-isomorphic $\kk[x]$-summands, 
        \item $U_i$ are indecomposable $\kk[x]$-modules with $\dim U_i\in 2\Z$, 
        \item $W=W'\oplus \Pi (W')^*$, with $\dim W'\in 2\Z+1$,  and both $W', \Pi (W')^*$ are isotropic indecomposable $\kk[x]$-modules.
    \end{itemize}
\end{lemma}
\begin{proof}

    Let $x\in \mathcal{N}^{\ad}(\g_{\bar 1})$ and consider the orthogonal decomposition of the $\kk[x]$-module $V$ given in \cref{lem:form-nilp}:
$$
    V=\bigoplus_{i=1}^p U_i\oplus\bigoplus_{j=1}^q W_j.
$$
Here every $U_i$ is an indecomposable $\kk[x]$-module of even dimension (see \cref{lem:form_B_indec_U}) and every $W_j=W_j'\oplus \Pi (W_j')^*$ is a direct sum of two isotropic indecomposable $\kk[x]$-modules paired by $B$.

If all $U_i$'s are pairwise non-isomorphic and there is at most one summand of the form $W_j$, then by \cref{cor:determining_semisimples_comm_with_x}, the element $x$ is distinguished.

Assume that some summands $U_i$ appears in the above decomposition with multiplicity at least $2$. We may then take the endomorphism $s$ of $V$ given by $\begin{bmatrix}
    0 &\lambda \\ -\lambda &0
\end{bmatrix} \in \End(U_i^{\oplus 2})$. We obtain $s\in \mathfrak{c}_{\g_{\bar 0}}(x)^{\mathbf{ss}}$. By \cref{cor:determining_semisimples_comm_with_x}, $x$ is not distinguished.

Now, for any direct summand $W_j$, let $\lambda_j \in \kk$ and define $$s_j:= \begin{bmatrix}
    \lambda &0 \\ 0&-\lambda
\end{bmatrix} \in \End(W_j)=\End(W_j'\oplus W_j'').$$ 
Let $s:=\bigoplus_j s_j$ be the corresponding endomorphism of $V$ (here $s\rvert_{U_i}=0$ for all $i$). If there exists $j$ such that $\sdim W'_j=0$, we may take $\lambda_k:=\delta_{j, k}$ and obtain: $s\in \g\subset \End(V)$; analogously, if there exist $j\neq k$ for which $\sdim W_j, \sdim W_k\neq 0$ we may take $\lambda_i:=0$ for $i\neq j, k$ and $\lambda_j, \lambda_k$ such that $\lambda_j\sdim W'_j= -\lambda_k\sdim W'_k$. In both these cases, we obtain $s\in \mathfrak{c}_{\g_{\bar 0}}(x)^{\mathbf{ss}}$. By \cref{cor:determining_semisimples_comm_with_x}, $x$ is not distinguished.

This completes the proof of the lemma.

\end{proof}
\begin{corollary}\label{cor:dist_orbits_spe}
    Let $x$ be a distinguished odd element in $\g=\mathfrak{spe}(n)$. Then $x$ is $\ad$-nilpotent and there exists $x_{neat}\in \g_{neat}$ and $x_{bal}\in \g_{bal}$ such that $x=x_{neat}+x_{bal}$ and we have:
    \begin{itemize}
        \item There exists an $\osp(1|2)$-type subalgebra $\osp_{x_{neat}}\subset \g$ commuting with $x_{bal}$.
\item There exists an $\sl_2$-type subalgebra $\sl_{x_{bal}}\subset \g$ commuting with $\osp_{x_{neat}}$.
\item \InnaA{The elements $x,x_{bal}$ have a common attractor.}
    \end{itemize}
\end{corollary}
\begin{proof}
Let $x\in \g_{\bar 1}$ be distinguished and consider the corresponding $\kk[x]$-decomposition $V=\bigoplus_{i=1}^p U_i \oplus W$ as in \cref{lem:disting_in_spe}. Let $x_{bal}:=x\rvert_{\bigoplus_i U_i}$ and $x_{neat}:=x\rvert_{W}$. Then $x_{bal}\in \g_{bal}$ and $x_{neat}\in \g_{neat}$. Clearly, we may choose an $\sl_2$-type subalgebra $\sl_{x_{bal}}\subset \gl(\bigoplus_i U_i)$ and an $\osp(1|2)$-type subalgebra $\osp_{x_{neat}}\subset \gl(W)$ which will commute with each other.

So it remains only to check that $att(x)=att(x_{bal})$. Indeed, let $(e,h, f:=\frac{1}{2}[x_{bal}, x_{bal}])$ and $(e',h', f':=\frac{1}{2}[x_{neat}, x_{neat}])$ be the $\sl_2$-triples in $\sl_{x_{bal}},\,\osp_{x_{neat}} $ respectively. Since the $\sl_{x_{bal}}, \osp_{x_{neat}}$ commute, the triple $(e+e', h+h', f+f')$ is also an $\sl_2$-triple, with $f+f'=\frac{1}{2}[x,x]$.

Let $x_{bal}=\sum_{i\leq 0} x_i$ be the $\ad_h$-decomposition of $x_{bal}$, with $[h, x_i]=ix_i$. The $\ad_{h+h'}$-decomposition of $x$ is then $x_0+(x_{-1}+x_{neat})+\sum_{i\leq 2} x_i$, with $[h+h', x_i]=ix_i$ when $i\neq -1$, $[h+h', x_{-1}+x_{neat}]=-(x_{-1}+x_{neat})$. We conclude that $x_0$ is an attractor of both $x_{bal}, x$ as required.
\end{proof}

\subsection{Nilpotent orbits in queer Lie superalgebras}

Let $\g=\q(n)$ and denote by $V=\kk^{n|n}$ the defining representation of $\q(n)$.
The $Q(n)_{\bar 0}$-orbits in $\g_{\bar 1}$ are just the $GL_n$-orbits in $\gl_n$, so we have:
\begin{lemma}
\mbox{}
\begin{enumerate}
    \item The $Q(n)_{\bar 0}$-orbits in $\mathcal{N}^{\ad}(\g_{\bar 1})$ are parameterized by (colorless) Young diagrams on $n$ boxes: for $x\in \mathcal{N}^{\ad}(\g_{\bar 1})$, the rows in the Young diagram for $x$ correspond to a pair of Jordan blocks of $x$ in $V$ having the same size, but with opposite parities of the basis vectors.
    \item   Neat orbits in $\g_{\bar 1}$ are parametrized by Young diagrams of size $n$ whose rows all have odd length.
\end{enumerate}
    
\end{lemma}

\begin{lemma}\label{lem:q-dist}
Let $x\in \g_{\bar 1}$ be a distinguished odd element
and $\frac{1}{2}[x,x]=s+f$ be the Jordan decomposition of $\frac{1}{2}[x,x]$. We have: 
\begin{enumerate}
    \item $s=\lambda \id_V$.
    \item If $s\neq 0$, then $f$ is a principal nilpotent element of $\g_{\bar 0}$. Considering a principal $\sl_2$-triple $(e,h,f)$ in $\g_{\bar 0}$, we have a decomposition $x=\sum_{i=0}^{n-1}x_{-2i}$ such that $[h,x_i]=ix_i$ for every $i$.
    \item If $s=0$, then $V$ decomposes into a direct sum of two indecomposable $\kk[x]$- modules $W\oplus \Pi W$. In that case, $x$ is neat if $n$ is odd, and $x$ is balanced if $n$ is even.
\end{enumerate}
    
\end{lemma}
 \InnaA{By \cref{prop:dist_square_implies_dist}, the above conditions on $s,f$ are also sufficient to ascertain that the element $x$ is distinguished.}

\begin{proof} 

\begin{enumerate}
    \item
By \cref{lem:dist-jor} we have: $s\in Z(\g)_{\bar 0}$, so $ s=\lambda\id_V$.

    \item Write $x=\left(\begin{matrix}0&B\\ B&0\end{matrix}\right)$ where $B\in \gl_n$. Since $x$ is distinguished, the Jordan normal form of $B$ consists of just one Jordan block. Write $B=\mu (1_n+J)$ where $J$ is the nilpotent Jordan block of size $n$ and $\mu\in \kk$. Then $\mu^2=\lambda$ and $f=\lambda\left(\begin{matrix}C&0\\ 0&C\end{matrix}\right)$ where $C=J^2+2J$. This implies that $f$ is a principal nilpotent element in $\g_{\bar 0}\cong \gl_n$ and we may embed it into a principal $\sl_2$-triple $(e,h,f)$ in $\g_{\bar 0}$. We write $h=:\left(\begin{matrix}H&0\\ 0&H\end{matrix}\right)$ where $H\in \gl_n$.
    
    Furthermore, $[f,x]=0$ implies that $B$ is a polynomial of $C$. The conditions $B=\mu(1_n+J)$, $C=J^2+2J$ imply that this polynomial is given by
    $$B=\mu\left(1_n+\frac{1}{2}C+\frac{1}{8}C^2+\dots+\binom{\frac{1}{2}}{n-1}C^{n-1}\right).$$
   From $[h,f]=-2f$ we have: $[H, C]=-2C$, so the $\ad_H$-decomposition of $B$ is given by $B=\sum_{k\leq 0} B_{2k}$, where $B_{-2k}=\mu \binom{\frac{1}{2}}{k}C^{k}$. This implies the $\ad_h$-decomposition $x=x_0+x_{-2}+\dots+x_{2-2n}$, where
    $[h, x_{i}]=ix_i$ and
    $$ x_0 = \left(\begin{matrix}0&\mu 1_n\\ \mu 1_n&0\end{matrix}\right) \;\;\; \text{ and }\;\;\; \forall ~k\geq 0, ~x_{-2k}=\left(\begin{matrix}0&\mu\binom{\frac{1}{2}}{k}C^{k}\\ \mu\binom{\frac{1}{2}}{k}C^{k}&0\end{matrix}\right).$$
    Clearly, in this case $x$ is strongly balanced.
    \item Consider the decomposition of $V$ into a direct sum of indecomposable $\kk[x]$-modules.

    We start by noticing that if $W$ is an indecomposable $\kk[x]$-module, then $W$ is not isomorphic to $\Pi W$ as a $\kk[x]$-module. Yet we have an isomorphism of $\g$-modules $V\cong \Pi V$, which implies that we have a $\kk[x]$-decomposition $V=\bigoplus_{i=1}^r (W_i\oplus \Pi W_i)$ where $W_i$ are all indecomposable. Note that $\g_{\bar 0}=\q(n)_{\bar 0}\subset \End(V)$ contains all endomorphisms of $V$ given by $diag(A_1, \ldots, A_r)$ where $A_i\in \End(W_i\oplus \Pi W_i)$ is of the form $\begin{bmatrix} A &0\\0 &A\end{bmatrix}
    $. But $x$ is distinguished, therefore we have (see \cref{cor:determining_semisimples_comm_with_x}): $r=1$ and $V=W\oplus \Pi W$, where $W$ is an indecomposable $\kk[x]$-module. 
    
    If $n$ is odd then $W, \Pi W$ are of odd dimension and so $x$ is a distinguished neat element. If $n$ is even then $x$ is a distinguished strongly balanced element \InnaA{(the subalgebra $\sl_x$ is constructed in $\gl(W)$ just as in the proof of \cref{prop:balanced_in_gl}, and corresponds to a subalgebra of $\q(n)_{\bar 0}$)}. A corresponding attractor of $x$ is then
    $$x_0=\left(\begin{matrix}0&B\\ B&0\end{matrix}\right) \;\; \text{ where } \;\; B = \diag\left(\left(\begin{matrix}0&1\\ 0&0\end{matrix}\right)~, ~ \left(\begin{matrix}0&1\\ 0&0\end{matrix}\right)~, \ldots,~ \left(\begin{matrix}0&1\\ 0&0\end{matrix}\right)\right).$$ 
    \end{enumerate}
\end{proof}
\begin{comment}
  %\begin{lemma}\label{lem:q-dist-nilp} Let $x\in \p\q(n)_{\bar 1}$ is distinguished nilpotent. Then it is either strongly balanced or neat.
%\end{lemma}
%\begin{proof} A distinguished nilpotent $x\in \p\q(n)_{\bar 1}$ is the image of the odd element satisfying the conditions of Lemma \ref{lem:q-dist} under the projection $\q(n)\to \p\q(n)$.
%If the preimage is not nilpotent the statement follows from Lemma \ref{lem:q-dist} (2). Thus, it remains to check the statement for the case when the preimage of $x$ is nilpotent.

%If $n$ is odd   Lemma \ref{lem:q-dist} (3) implies that $x$ is neat. If $n=2k$ is even one can choose bases $e_1,\dots,e_n$ in $V_{\bar 0}$ and $f_1,\dots, f_n$ in $V_{\bar 1}$ such that $\Pi(e_i)=f_i$ and the action of $x$ in this basis is given by $$e_1\mapsto f_2\mapsto e_3\mapsto\dots\mapsto f_n,$$ $$f_1\mapsto e_2\mapsto f_3\mapsto\dots\mapsto e_n.$$
%We choose $h\in\p\q(n)_{\bar 1}$ such that  $$h(e_{2i+1})=(2i-k+1) e_{2i+1}, h(e_{2i})=(2i+1-k) e_{2i},$$ $$h(f_{2i+1})=(2i-k+1) f_{2i+1}, h(f_{2i})=(2i+1-k) f_{2i}.$$ Then $x=x_0+x_{-2}$ where  $$x_0(e_{2i+1})=f_{2i+2},\ x_0(f_{2i+1})=e_{2i+2},\ x_0(e_{2i})=x_0(f_{2i})=0.$$
%\end{proof}
  
\end{comment}
By \cref{lem:pg_distinguished_orbits}, we have:
\InnaA{
\begin{lemma}
For $\p\q(n)$, the odd distinguished elements are the same as for $\q(n)$. For $\g=\s\q(n), \p\s\q(n) $, the distinguished orbits are the distinguished orbits in $\q(n)$ where $s=0$.
\end{lemma}
\begin{corollary}\label{cor:psl-dist} Every odd distinguished element in $\q(n), \p\q(n), \s\q(n), \p\s\q(n)$ is either strongly balanced or neat. 
    \end{corollary}
}
\subsection{Nilpotent orbits in exceptional Lie superalgebras}\label{sec:char_in_exceptional}

Let $\g$ be a basic simple Lie superalgebra of defect $1$. Simple Lie superalgebras of defect $1$ are $D(2,1;\alpha)$, $AG_2$ (also known as $G_3$, $G(1,2)$), $AB_3$ (also known as $F_4, F(1,3)$), $\mathfrak{sl}(1|n)$,
$\mathfrak{osp}(2|2n)$, $\mathfrak{osp}(3|2n)$, $\mathfrak{osp}(m|2)$. 

The odd self-commuting cone $\g^{\mathbf{sc}}:=\{x\in \g_{\bar 1}:[x,x]=0\}$ then has either one or two nontrivial $G_{\bar 0}$-orbits; in the latter case, the two orbits have the same dimension. These are the $G_{\bar 0}$-orbits in $\mathfrak g_{\bar 1}$ of minimal positive dimension. 

Given a self-commuting element $0\neq u\in \g^{\mathbf{sc}}$, the Lie superalgebra $DS_u(\g)$ is either a simple Lie algebra or $DS_u(\g)\simeq\mathfrak{osp}(1|2k)$.
The latter case occurs only for $\g=\mathfrak{osp}(3|2n)$.

A non-trivial $G_{\bar 0}$-orbit in $\g^{\mathbf{sc}}$ is an orbit of a root vector $x_{\alpha}\in\g_{\alpha}$ where $\alpha$ is an isotropic root.
The root vectors $x_{\alpha}$ and $x_{-\alpha}$ generate  an $\mathfrak{sl}(1|1)$-subalgebra with its center spanned by the Cartan element $h_\alpha$ corresponding to the root $\alpha$.
Note that $(h_\alpha,h_\alpha)=0$ and $DS_{x_\alpha}(\g)$ can be computed from the following relation:
\begin{equation}\label{eq:def1}
\g^{h_\alpha}=DS_{x_\alpha}(\g)\oplus \mathfrak{gl}(1|1)   
\end{equation}
Note that the above formula works for any finite-dimensional simple Kac-Moody superalgebra.

\begin{lemma}\label{lem:auxdef1} Let $\g$ be a simple finite-dimensional Kac-Moody superalgebra and $x_\alpha$ be a root vector for an isotropic root $\alpha$, then
$[x_\alpha,\g]_{\bar 0}$ is a solvable Lie algebra.
    \end{lemma}
    \begin{proof} 

We have $$[x_\alpha,\g]=\C h_{\alpha}\oplus \bigoplus_{\beta\in\Delta}[x_\alpha,\g_\beta],$$
where $\Delta$ is the set of roots of $\g$. Then from the property of root system we have for any root $\beta\neq\pm\alpha$:
\begin{itemize}
    \item If $(\beta,\alpha)=0$ then $\beta\pm\alpha$ is not a root.
    \item If $(\beta,\alpha)\neq 0$ then exactly one of $\beta+\alpha$, $\beta-\alpha$ is a root.    
\end{itemize}
Let $\Delta'$ be the set of roots $\gamma$ such that $\g_{\gamma}\subset[x_\alpha,\g]$. We claim that if $\gamma\in\Delta'$ then $-\gamma\notin\Delta'$. Indeed assume that $-\gamma$ is in $\Delta'$. We have $\gamma-\alpha,-\gamma -\alpha\in\Delta$. Then $\gamma+\alpha\in\Delta$ and we obtain a contradiction. 

Denote by $\Delta_0'\subset \Delta'$ the subset of even roots $\Delta'$. 
Let $\fl=\bigoplus_{\gamma\in\Delta'_0}\g_\gamma$. Then $\fl$ is an ideal in
 $\h\oplus\fl$. Consider the Killing form $C(\cdot,\cdot)$ on $\fl$. It is $\h$-invariant and therefore for any $\beta,\gamma\in \Delta'_0$ we have $C(\g_\beta,\g_\gamma)=0$ because $\beta+\gamma\neq 0$. In other words the Killing form is zero on $\fl$ and hence $\fl$ is solvable. This implies that $[x_\alpha,\g]_{\bar 0}=\C h_\alpha+\fl$ is solvable. 
 \end{proof}

\begin{lemma}\label{lem-defect1} Assume that $\g$ has defect $1$ and $DS_{x_\alpha}\g$ is a Lie algebra (purely even). Then every $x\in\mathcal{N}^{\ad}(\g_{\bar 1})$ is either self-commuting (thus balanced) or neat.
    \end{lemma}
\begin{proof} Assume that $f:=\frac{1}{2}[x, x]\neq 0$ (otherwise $x$ is self-commuting and we are done). 

Let $(e,h,f)$ be an $\sl_2$-triple and let $x_0\in \g^{\mathbf{sc}}$ be the attractor element of $x$ with respect to this triple (see \cref{sec:attractor}). If $x_0=0$ then $x$ is neat by \cref{lem:criterion_neatness_h_decomp} and we are done.

Now assume that $x_0\neq 0$. Then $x_{-1},e,h,f\in \g^{x_0}$. The condition that $DS_{x_0}\g$ is
purely even implies that the image of $x_{-1}$ in $DS_{x_0}\g$ is zero. Hence the image of $e,h,f$ in $DS_{x_0}\g$ is zero as well (see \cref{rem:x_minus_1_is_zero_in_DS}), so the $\sl_2$-subalgebra spanned by $e,h,f$ lies in $[x_0,\g]$.

Since $x_0$ is self-commuting, we can assume without loss of generality that $x_0=x_{\alpha}$ for an isotropic root $\alpha$. \cref{lem:auxdef1} implies that
$[x_0,\g]$ is solvable and so it cannot contain a copy of $\sl_2$. This contradicts our assumptions and completes the proof of the lemma. 
\end{proof}
\begin{remark} For $\g=\mathfrak{osp}(3|2n)$ and $n>1$, Lemma \ref{lem-defect1} is false. Indeed, one can consider a neat element $x_{-1}$
in $\mathfrak{osp}(1|2)$ and self-commuting element $x_0$ in $\mathfrak{osp}(2|2n-2)$  and $x=x_0+x_{-1}$ is neither neat nor self-commuting. 
\end{remark}

\cref{lem-defect1} reduces the classification of $\ad$-nilpotent orbits for exceptional Lie superalgebras to the classification of neat orbits and self-commuting orbits.
Moreover, to classify neat orbits, it suffices to classify nilpotent elements in $\g_{\bar 0}$ which are squares of neat elements (see \cref{lem:conjugation_of_squares}). We obtain the following dimensions of neat orbits, listed with multiplicities (a detailed classification will appear in a subsequent paper, using Kostant-Dynkin characteristic markings):
\begin{itemize}
    \item $D(2,1;a)$ has neat orbits of dimension $7,5,5,5$.
    \item $AG_2$ has neat orbits of dimensions $13,11,11,9,8$.
        \item $AB_3$ has neat orbits of dimensions $15,14,11,9$.
\end{itemize}

The complete classification of neat orbits will be published in a forthcoming paper. It follows from this classification that $D(2,1,a)$ and $AG_2$ have exactly one distinguished orbit, while $AB_3$ does not have distinguished orbits. 

Since the even parts of the exceptional Lie superalgebras are semisimple and have trivial centers, we conclude that any distinguished odd element in an exceptional Lie superalgebra is $\ad$-nilpotent; so it is either self-commuting (hence balanced) or neat.
\subsection{Distinguished orbits in Takiff superalgebras}

Let $\s$ be a simple Lie algebra and let $\g=\s\otimes \kk[\xi]\oplus \kk \partial$ be the corresponding Takiff superalgebra. Here $\xi$ is odd, $[\xi, \xi]=0$ and $\partial=\frac{\partial}{\partial_\xi}$ is an odd derivation of $\g$.
\begin{lemma}\label{lem:Takiff} Let
 $\g=\s\otimes \kk[\xi]\oplus \kk \partial$.
Let $x \in \g_{\bar 1}$. \InnaA{Then $x$ is distinguished if and only if
$x=\lambda \partial+\xi u$ where $\lambda\in \kk$ and $u$ is a distinguished nilpotent element in $\s$.}
\end{lemma}
\begin{proof} 
We can always write $x=\lambda \partial+\xi u$ for some
$u\in\s$. \InnaA{The condition that $x$ is distinguished is equivalent to $\s^u$ (the centralizer of $u$ in $\s$) having no semisimple element. This is equivalent to $u$ being nilpotent and distinguished as an element of $\s$.}
\end{proof}

\begin{corollary}\label{Cor:Takiff} Every odd distinguished element in $\g=\s[\xi]\oplus \kk \partial$ is balanced. 
\end{corollary}
\begin{proof} Consider the decomposition $x=\lambda \partial+\xi u$, where $u\in \s$ is a distinguished nilpotent element.
If $\lambda =0$ then $[x,x]=0$, so $x$ is balanced. If $\lambda \neq 0$ then $f=\frac{1}{2}[x,x]=\lambda\partial(\xi u)=\lambda u$ is nilpotent in $\s$ and may be placed inside an $\sl_2$ triple $(e,h,f)$. We have: $[h, \partial]=0$ and $[h, \xi u]=\xi [h,u]=2\xi u$ (since $[h,f]=-2f$). So letting $x_0:=\lambda \partial$, $x_{-2}:=\xi u$ we obtain: $x=x_0+x_{-2}$ is strongly balanced.
\end{proof}

\section{Decomposition of an odd \texorpdfstring{$\ad$}{ad}-nilpotent element}\label{sec:2step}
Let $G$ be a quasi-reductive supergroup.
\subsection{Definition: a 2-step decomposition}
\begin{definition}\label{def:2step}
Let $x\in \g_{\bar 1}$ and let $\frac{1}{2}[x,x]=s+f$ be the Jordan decomposition of $\frac{1}{2}[x,x]$.
  A {\it 2-step decomposition} of $x$ is a decomposition $x=x_{neat}+x_{bal}$, so that 
  \begin{itemize}
      \item $x_{neat}\in \g_{neat}$, $x_{bal}\in \g_{bal}$, 
  %    \item $[x_{bal},x_{neat}]=0$,
      \item There exists an $\osp(1|2)$-type subalgebra $\osp_{x_{neat}}\subset \g^s$ \InnaA{containing $x_{neat}$} such that $[x_{bal}, \osp_{x_{neat}}]=0$. In particular, $[x_{bal},x_{neat}]=0$.
   
  \end{itemize}
\end{definition}

\begin{example}
Let $V$ be a finite-dimensional vector superspace and let $x\in \gl(V)_{\bar 1}$.  We will construct a $2$-step decomposition of $x$ in $\gl(V)$.

Let $\frac{1}{2}[x,x]=s+f$ be the Jordan decomposition of $\frac{1}{2}[x,x]$.

It is enough to construct a $2$-step decomposition of $x$ in $\gl(V^s)\subset \gl(V)$, so we may assume from now on that $s=0$ and $x$ is $\ad$-nilpotent. 
 Let
$$V=\bigoplus_i V_i$$ where each $V_i$ corresponds to a single Jordan block of $x$. Let $x_{neat}\in \End(V)_{\bar 1}$ be the operator given as follows: for each $i$, let $x_{neat}\rvert_{V_i}:=0$ if $\dim V_i\in 2\Z$ and $x_{neat}\rvert_{V_i}:=x\rvert_{V_i}$ if $\dim V_i\in 2\Z+1$. We set $x_{bal}:=x-x_{neat}$, so that $x_{bal}\rvert_{V_i}=0$ if $\dim V_i\in 2\Z+1$, and $x_{bal}\rvert_{V_i}=x\rvert_{V_i}$ if $\dim V_i\in 2\Z$. It is straightforward to see that $x=x_{neat}+x_{bal}$ is indeed a $2$-step decomposition of $x$. 
\end{example}

\subsection{Existence of the 2-step decomposition}

In this section, we will show the existence of a $2$-step decomposition for every $x \in \g_{\bar 1}$. We will prove a slightly stronger and more technical statement (which will be useful later on) and derive the existence of the $2$-step decomposition at the end of this section.

\begin{theorem}\label{thm:2-step_exist} Let $x\in \g_{\bar 1}$ and let $\frac{1}{2}[x,x]=s+f$ the Jordan decomposition of $\frac{1}{2}[x,x]$. Assume that $f\neq 0$. There exists exists an $\sl_2$-triple $(e,h,f)$ in $\g_{\bar 0}^s$ so that the $\ad_h$-eigenvector decomposition $x=\sum_{i\leq 0} x_i$ of $x$ satisfies:
\begin{itemize}
    \item We have $x_{neat}:=x_{-1}\in \g_{neat}$, $x_{bal}:=x-x_{-1}\in \g_{bal}$.
    \item There exists an $\osp(1|2)$-type subalgebra $\osp_{x_{neat}} \subset \g^s$ containing $x_{neat}$ such that $[x_0, \osp_{x_{neat}}] =[x_{bal}, \osp_{x_{neat}}]=0$.
    \item \InnaA{$x_0 \in att(x) \cap att(x_{bal})$}.
\end{itemize}
\InnaA{In particular, $x=x_{neat}+x_{bal}$ is a $2$-step decomposition of $x$.}
  
\end{theorem}
 Let us briefly sketch the strategy of the proof. As usual, the first easy step consists of passing to the case when $x$ is $\ad$-nilpotent, by restricting ourselves to the quasi-reductive Levi subalgebra $\g^s$.
 
The main idea of the proof of  \cref{thm:2-step_exist} is to reduce the problem to the case when $x$ is distinguished by considering a small enough Levi subalgebra $\fl$ containing $x$. We then consider the structure of $\fl$ as a quasi-reductive Lie superalgebra, with minimal ideals which are simple Lie superalgebras. This allows us to (almost) reduce the problem to the case where $x$ is a distinguished element in a simple Lie superalgebra, in which case we may use the classification of such elements appearing in \cref{sec:orbits_in_classical}. The ``almost'' caveat refers to the fact that the original element $x$ does not have to lie in the direct sum of the minimal ideals of $\fl$; yet one may decompose $x$ as $x=x'+x''$, where $x''$ lies in the sum of the minimal ideals in $\fl$, while $\ad_{x'}$ acts as an odd derivation on these ideals. The $2$-step decomposition of $x$ then involves incorporating $x'$ in the balanced part of $x$.

We begin with a straightforward lemma.

\begin{lemma}\label{lem:commuting_balanced}
    
    \mbox{}
    \begin{enumerate}
        \item Let $x, x'$ be strongly balanced elements in $\g_{\bar 1}$. Let $\sl_x, \sl_{x'}\subset \g_{\bar 0}$ be as in \cref{def:strongly_bal}.
        
        Assume that $[x,x']=0,  [\sl_x,\sl_{x'}]=0$.
        Then $x+x'$ is strongly balanced and an attractor of $x+x'$ is given by the sum of attractors of $x,x'$ corresponding to the Cartan elements of $\sl_x, \sl_{x'}$ respectively.

    \item Let $x,x'\in \g_{neat}$ and let $\osp_x, \osp_{x'}\subset \g$  be as in \cref{def:neat_element_aux_notions}.
    Assume that $[\osp_x, \osp_{x'}]=0$.
        Then $x+x'\in \g_{neat}$.
    \end{enumerate}
\end{lemma}
Next, by the results in \cref{sec:orbits_in_classical}, we have:
\begin{lemma}\label{lem:2step_decomp_in_simples}
    If $\g$ is a simple Lie superalgebra and $x\in \g_{\bar 1}$ is distinguished, then there exists a decomposition of $x$ as in \cref{thm:2-step_exist}.
\end{lemma}
\begin{proof}
   In all the cases except $\mathfrak{spe}(n)$, distinguished orbits are either balanced or neat, as seen in \cref{thrm:dist_is_neat_or_balanced}. In the case of $\mathfrak{spe}(n)$, the appropriate decomposition was shown in \cref{cor:dist_orbits_spe}.
\end{proof}

\begin{proof}[Proof of \cref{thm:2-step_exist}] 

Let $\s$ be a maximal toral subalgebra in $\mathfrak{c}_{\g_{\bar 0}}(x)$ and consider the minimal Levi subalgebra $\g^\s$ containing $x$; by \cref{lem:uniqlevi}, $x$ is distinguished in $\g^\s$.
We may consider the image $x+Z(\g^\s)$ of $x$ in the quotient $\fl:=\g^\s/Z(\g^\s)$. By abuse of notation, we will denote this image by $x$ as well; it is distinguished in $\fl$. Repeating this step as needed, we reduce our problem to the case when $Z(\fl)=0$. 

In what follows, we will construct a decomposition of $x$ into a sum of commuting neat and balanced elements in in $\fl$. Recall from \cref{lem:aux_odd_central_commutator} that the sub-superalgebra $Z(\g^\s)_{\bar 1}$ splits off $\g^\s$ as a direct factor; so such a decomposition of $x$ in $\fl$ lifts naturally to an analogous decomposition of $x$ in $\g^\s\subset \g$.

We now use the description of the structure of $\fl$ as a quasi-reductive Lie superalgebra, as given in \cref{ssec:quasired}.
Let $\mathfrak{i}(\fl)=\bigoplus_{\alpha}  \mathfrak{t}_{\alpha}$ be the sum of the minimal ideals of $\fl$, so that for $\alpha\neq \beta$, we have $[\mathfrak{t}_{\alpha}, \mathfrak{t}_\beta]=0$, and each $\ft_{\alpha}$ satisfies: \begin{itemize}
        \item $\ft_{\alpha}$ is a simple Lie superalgebra, 
        \item $\ft_{\bar 1}$ is an abelian ideal while $\ft_{\bar 0}$ is either a simple Lie algebra (in which case $\ft_{\bar 1}$ is its adjoint representation) or $\ft_{\bar 0}=0$.
    \end{itemize} 
We have a decomposition of $\fl_{\bar 0}$-modules $\fl=\mathfrak{i}(\fl)\oplus \rr$ such that $[\rr_{\bar 1}, \rr_{\bar 1}]=0$, $[\mathfrak{i}(\fl)_{\bar 0}, ~\mathfrak{r}_{\bar 1}]=0$.

Let us write
$$x=:x'+\sum_{\alpha} y_{\alpha},\quad x'\in\rr_{\bar 1},\ y_{\alpha}\in (\mathfrak{t}_{\alpha})_{\bar 1}.$$
Note that $[x', x']=0$ since $[\rr_{\bar 1}, \rr_{\bar 1}]=0$.

First, we claim that each $y_{\alpha}$ is distinguished in $\mathfrak{t}_{\alpha}$. Indeed, assume otherwise: then there exists $\alpha$ and a semisimple element $s \in (\mathfrak{t}_{\alpha})_{\bar 0}\setminus Z(\mathfrak{t}_{\alpha})$ such that $[s,y_{\alpha}]=0$. Since $[\mathfrak{t}_{\alpha}, \mathfrak{t}_\beta]=0$ for $\alpha\neq \beta$, we have: $[s,\ft_{\beta}]=0$ for $\beta\neq \alpha$. Furthermore, $[s,x']=0$ since $$[(\mathfrak{t}_{\alpha})_{\bar 0}, x']\subset [\mathfrak{i}(\fl)_{\bar 0}, ~\mathfrak{r}_{\bar 1}]=0.$$ Hence $[s,x]=0$ and we obtain a contradiction with the fact that $x$ is distinguished. So for each $\alpha$, the element $y_{\alpha}$ is distinguished in $\mathfrak{t}_{\alpha}$.

Let $I:=\{\alpha~\mid~ [x',\mathfrak{t}_{\alpha}]\neq 0\}$ and $z:=x'+\sum_{\alpha \in I}y_{\alpha}$. We will show in \cref{lem:aux_2_step_decomp} below that
$z$ is strongly balanced. 

Next, we use 
\cref{lem:2step_decomp_in_simples}: for each $\alpha \notin I$, we have a decomposition $y_{\alpha} = y^{(\alpha)}_{neat}+y^{(\alpha)}_{bal}$ in $\ft_{\alpha}$ satisfying the requirements of the theorem. Let $\osp_{y_{\alpha}}$, $\sl_{y_{\alpha}}$ be the corresponding Lie sub-(super)algebras. We denote by $\sl_z$ the \InnaA{$\sl_2$-subalgebra} corresponding to $z$.

Now,
$[z, \mathfrak{t}_{\alpha}]=0$ for any $\alpha \notin I$, so \cref{lem:commuting_balanced} implies $x_{bal}:=z+\sum_{\alpha \notin I} y^{(\alpha)}_{bal}$ is strongly balanced.
By the same \cref{lem:commuting_balanced}, the element $x_{neat}:=\sum_{\alpha\notin I} y^{(\alpha)}_{neat}$ is neat.

Next, we choose the $\sl_2$-triple $(e,h,f)$ for $x$ by considering a diagonal embedding of $\sl_2$ into $\sl_z \oplus \bigoplus_{\alpha\notin I} \sl_{y_{\alpha}}\oplus \bigoplus_{\alpha \notin I} \osp_{y_{\alpha}}$. The $\ad_h$-decomposition of $x$ is then given by $x=\sum_{i\leq 0} x_i$, where $[h, x_i]=ix_i$. Clearly, $x_{neat}=x_{-1}$, $x_{bal}=\sum_{i\neq -1} x_i$ and $x_0$ is the attractor of both $x$ and $x_{bal}$ \InnaA{(note that $x_0$ is just the sum of the corresponding attractors of $z, y^{(\alpha)}_{bal}$ for $\alpha \notin I$)}. 

We conclude that $x=x_{neat}+x_{bal}$, $[x_{neat}, x_{bal}]=0$ and $x_0\in att(x)\cap att(x_{bal})$.

As part of our construction, we obtain an $\osp(1|2)$-type subalgebra $\osp_{x_{neat}}\subset \oplus_{\alpha \notin I} \mathfrak{t}_{\alpha}$ generated by $x_{neat}$ and the diagonal embedding of $\sl_2$ into $ \bigoplus_{\alpha \notin I} \osp_{y_{\alpha}}$. This sub-superalgebra satisfies: $[x_{bal}, \osp_{x_{neat}}]=[x_{0}, \osp_{x_{neat}}]=0.$

We have thus obtained a $2$-step decomposition of $x$ in $\g^\s$ (and thus in $\g$) satisfying the conditions of \cref{thm:2-step_exist}, as required. 
\end{proof}
To complete the proof of \cref{thm:2-step_exist}, we need two auxiliary lemmas:
\begin{lemma}
Let $\alpha\in I$, $y_{\alpha}$ and $x'$ as defined in the proof of \cref{thm:2-step_exist}. Let $\ft'_{\alpha}:=\kk x'\ltimes \mathfrak{t}_{\alpha} \subset \g$. Then $\widetilde{y}_\alpha:= x'+y_\alpha$ is strongly balanced in $\ft'_{\alpha}$, with $ x' \in att(\widetilde{y}_\alpha)$.
\end{lemma}
\begin{proof}
    First of all, $\ad_{x'}$ acts by an odd derivation on \InnaA{$\g$, with $\ad_{x'}(\g_{\bar 0})=0$. Recall that $\ft_{\alpha}$ satisfies: $[x', \ft_{\alpha}]\neq 0$. So the odd operator $\ad_{x'} \in \End^{\bullet}(\g) $ is injective on $(\mathfrak{t}_{\alpha})_{\bar 1}$ and acts by zero on $ (\mathfrak{t}_{\alpha})_{\bar{0}}$. This implies that $\ft_{\alpha}$ is isomorphic to either $\p\s\q(n)$ or to the maximal ideal $\s\oplus \xi\s$ in a Takiff superalgebra, with $\s$ a simple Lie algebra and $\xi^2=0$. }
    Consider the quasi-reductive Lie sub-superalgebra
$\ft'_{\alpha}:=\kk x'\ltimes \mathfrak{t}_{\alpha}\subset \g$. \InnaA{The above conditions on the ideal $\ft_{\alpha}$ imply that} $\ft'_{\alpha}$ is isomorphic either to a Takiff superalgebra or to $\mathfrak{pq}(n)$.

The element $\widetilde{y}_\alpha:=x'+y_\alpha$ is an odd distinguished element in $\ft'_{\alpha}$, for the same reason that $y_\alpha$ is distinguished in $\ft_{\alpha}$ (see the proof \cref{thm:2-step_exist}). 

Distinguished odd elements in the Takiff superalgebras and in $\mathfrak{pq}(n)$ are described in \cref{lem:Takiff} and \cref{lem:q-dist} respectively. In Takiff superalgebras, these elements are all strongly balanced, so we may assume from now on that $\ft'_{\alpha}\cong \mathfrak{pq}(n), ~\ft_{\alpha}\cong \mathfrak{psq}(n)$. In this case, we have: $$x'=\begin{pmatrix}
    0 &\mu I_n\\
    \mu I_n & 0
\end{pmatrix}$$ for some $\mu\neq 0$. The fact that ${y}_\alpha \in \ft_{\alpha}\cong \mathfrak{psq}(n)$ is distinguished implies that $y_{\alpha}=\begin{pmatrix}
    0 &A\\
    A & 0
\end{pmatrix}$ for a nilpotent $A\in \gl_n$. Hence $$ \widetilde{y}_\alpha=\begin{pmatrix}
    0 &\mu I_n +A\\
    \mu I_n+A & 0
\end{pmatrix}$$ and $[ \widetilde{y}_\alpha,  \widetilde{y}_\alpha]$ is not nilpotent in $\q(n)$. By \cref{lem:q-dist}, we conclude that $ \widetilde{y}_\alpha$ is strongly balanced. In both cases, \cref{lem:Takiff} and \cref{lem:q-dist} imply: $ x'\in att(\widetilde{y}_\alpha)$.
\end{proof}
\begin{lemma}\label{lem:aux_2_step_decomp}
    The element $z=x'+\sum_{\alpha \in I}y_{\alpha}$ defined in the proof of \cref{thm:2-step_exist} is strongly balanced.
\end{lemma}
\begin{proof}
Let $\alpha\in I$ and $x'$ as defined in the proof of \cref{thm:2-step_exist}. By the above lemma, 
$\widetilde{y}_\alpha:=x'+y_\alpha$ is a strongly balanced odd element in the quasi-reductive Lie superalgebra $\ft'_{\alpha}:=\ft_{\alpha}\rtimes \kk x'$, with $x'\in att(\widetilde{y}_\alpha)$.

Let
$$ f_{\alpha}:=\frac{1}{2}[\widetilde{y}_\alpha, \widetilde{y}_\alpha] = \frac{1}{2}[x', y_{\alpha}]+\frac{1}{2}[y_\alpha, y_\alpha].$$
Now, for every $\alpha\in I$, let us choose an $\sl_2$-triple $(e_{\alpha}, h_{\alpha}, f_{\alpha})$ in $(\ft_{\alpha})_{\bar 0}=(\ft'_{\alpha})_{\bar 0}$ such that
$$\widetilde{y}_\alpha=\sum_i (\widetilde{y}_\alpha)_{i} \;\;\; \text{ where }\;\;\;[h_{\alpha}, (\widetilde{y}_\alpha)_{i}]=i(\widetilde{y}_\alpha)_{i}\;\;\; \text{ and } \;\;\;(\widetilde{y}_\alpha)_{-1}=0.$$

Denote $$f:=\sum_{\alpha\in I} f_{\alpha}, \;\;\;h:=\sum_{\alpha\in I} h_{\alpha}, \;\;\;e:=\sum_{\alpha\in I} e_{\alpha}.$$ Since $[\mathfrak{t}_\alpha, \mathfrak{t}_\beta]=0$, the triple $(e,h,f)$ is an $\sl_2$-triple in $\bigoplus_{\alpha\in I} \ft_{\alpha}$.
For the same reason, we obtain: $$\frac{1}{2}[z,z] = \frac{1}{2}\left[x'+\sum_{\alpha \in I}y_{\alpha}, x'+\sum_{\alpha \in I}y_{\alpha}\right]=\sum_{\alpha \in I}\frac{1}{2}[x', y_{\alpha}]+\sum_{\alpha \in I}\frac{1}{2}[y_\alpha, y_\alpha]=\sum_{\alpha \in I} f_{\alpha}=f.$$

For any $i< 0$, let $z_{i}:=\sum_{\alpha\in I} (\widetilde{y}_\alpha)_{i}$, $z_{0}:= x'$. Recall that $\ad_{x'}:  (\mathfrak{t}_{\alpha})_{\bar{0}}\to  (\mathfrak{t}_{\alpha})_{\bar{1}}$ is zero. So for any $\alpha\in I$, we have: $[x', h_{\alpha}]=0$. From the calculations above, we have an $\ad_h$-decomposition $z=\sum_{i\leq 0} z_i$ with $z_{-1}=0$.
We conclude that $ z$ is strongly balanced.

\end{proof}
As a direct consequence of \cref{thm:2-step_exist}, we have:
\begin{corollary}\label{cor:2_step_exist}
    Let $x\in \g_{\bar 1}$. There exists a $2$-step decomposition $x=x_{neat}+x_{bal}$.
\end{corollary}

\subsection{Finiteness of attractor map}
\begin{theorem}\label{thm:finattractor} Consider the attractor map $att: \g_1/G_{\bar 0}\to \g^{hom}/G_{\bar 0}$, $G_{\bar 0}.x \to G_{\bar 0}.att(x)$. Then for any $y\in\g^{hom}$ the preimage $att^{-1}(G_{\bar 0}.y)$ is finite.
    \end{theorem}
\begin{proof} It suffices to prove that for any $y$, there are finitely many (up to conjugation) $2$-step decompositions $x=x_{bal}+x_{neat}$ such that $y\in att(x_{bal})$. Indeed, there are finitely many neat orbits \InnaA{(see \cref{cor:fin_many_neat_orb})}. We may only choose neat orbits that contain an element $x_{neat}$ centralizing $y$.
Let $\mathfrak{k}$ denote the centralizer of $\osp_{x_{neat}}$. Since $\mathfrak k$ is quasi-reductive (see \cref{prop:centralizer_reductive}), by \cref{cor:balanced} there are finitely many (up to conjugation) strongly balanced elements in $\mathfrak k$ with attractor $y$. 
This proves the statement.
\end{proof}

\begin{corollary}\label{cor:nilpnumber}
Let $\g^{\mathbf{sc}}:=\{x\in \g_{\bar 1}: [x,x]=0\}$ denote the cone of self-commuting odd elements in a quasi-reductive Lie superalgebra $\g$. Then $\mathcal{N}^{\ad}(\g_{\bar 1})$ has finitely many $G_{\bar 0}$-orbits if and only if $\g^{\mathbf{sc}}$ has finitely many $G_{\bar 0}$-orbits.
\end{corollary}

\subsection{2-Step decomposition and symmetric monoidal functors}
Let $x\in \g_{\bar 1}$, let $\frac{1}{2}[x,x]=s+f$ be the Jordan decomposition of $\frac{1}{2}[x,x]$ and fix a $2$-step decomposition $x=x_{neat}+x_{bal}$ of $x$ with properties as described in \cref{thm:2-step_exist}.

Let $\Phi_x, \Phi_{x_{neat}}: \Rep(G)\to \Rep(OSp(1|2))$, $\Phi_{x_{bal}}: \Rep(G)\to \sVect$ be the corresponding symmetric monoidal functors (see \cref{ssec:DS_functors_prelim}).

Recall that $[x_{bal}, \osp_{x_{neat}}]=0$ for some $\osp(1|2)$-type subalgebra $\osp_{x_{neat}}\subset \g^s$ containing $x_{neat}$. Fix such $\osp_{x_{neat}}$ and let $\varphi:OSp(1|2)\to G$ be the corresponding group homomorphism. We denote $(-)\downarrow_\varphi: \Rep(G)\to \Rep(OSp(1|2)\times G')$ where $G'$ is the quasi-reductive subgroup of $G$ centralizing $\Im(\varphi)$.

\begin{proposition}\label{prop:tensor_functors_2_step}
We have an isomorphism of functors
$$ (\id\boxtimes\Phi_{x_{bal}})\circ (-)\downarrow_{\varphi}\cong\Phi_x:  \Rep(G)\to \Rep(OSp(1|2)).$$
\end{proposition}

\begin{proof}
 We denote by $\psi: \G^{(1|1)}\to \InnaA{G'}$ the homomorphism corresponding to $x_{bal}\in \InnaA{\g^{x_{neat}}}$.

Recall the functors $T, \overline{T}$ from \cref{lem:aux_ss_functors}:
\begin{align*}
    T: \Rep(OSp(1|2))\boxtimes \Rep(\G^{(1|1)}) ~\longrightarrow~\Rep(\G^{(1|1)}), \;\;\; M\boxtimes M'\longmapsto M\downarrow_X~\otimes~ M',\\
    \overline{T}: \Rep(OSp(1|2))\boxtimes \Rep(OSp(1|2)) ~\longrightarrow~\Rep(OSp(1|2)), \;\;\; M\boxtimes M'\longmapsto M\otimes M',
\end{align*}
where $\downarrow_X: \Rep(OSp(1|2))\to \Rep(\G^{(1|1)})$ denotes the usual restriction functor. 

We claim that there exists a natural isomorphism making the following diagram of functors commutative:
        $$\xymatrix{  &\Rep(OSp(1|2)) \boxtimes \Rep(G') \ar^{\id\boxtimes \Phi_{x_{bal}}}[rr]  &{}&\Rep(OSp(1|2)) \boxtimes \sVect \ar_-{\id}[d]\\
        &\Rep(G) \ar_{\Phi_x}[rr] \ar^-{(-)\downarrow_{\varphi}}[u] &{}& \Rep(OSp(1|2)) .
 }$$

By the definition of $T$, there exists a natural isomorphism making the following diagram of functors commutative:
 $$\xymatrix{
        &\Rep(OSp(1|2)\times G') \ar^{(-)\downarrow_{\psi}}[rr] &{} &\Rep(OSp(1|2)\times \G^{(1|1)}) \ar^T[d] \\  &\Rep(G) \ar[rr]^{(-)\downarrow_x}\ar^-{(-)\downarrow_{\varphi}}[u]  &{} &\Rep(\G^{(1|1)})  }$$
        \InnaA{Here $(-)\downarrow_{\psi}$ denotes the restriction functor with respect to $\psi$ and the functor $(-)\downarrow_{x}: \Rep(G)\to \Rep(\G^{(1|1)})$ is the restriction functor corresponding to the element $x\in \g$.
        }
    By \cref{lem:aux_ss_functors}, we obtain a commutative diagram of functors
    $$\xymatrix{
        &\Rep(OSp(1|2)\times G') \ar^{(-)\downarrow_{\psi}}[r] \ar@/^2pc/^{\id\boxtimes \Phi_{x_{bal}}}[rr] &\Rep(OSp(1|2)\times \G^{(1|1)}) \ar^T[d] \ar^-{\id \boxtimes S}[r]  &\Rep(OSp(1|2))\boxtimes \Rep(OSp(1|2))\ar_{\overline{T}}[d]\\  &\Rep(G) \ar[r]^{(-)\downarrow_x} \ar@/_2pc/_{\Phi_{x}}[rr]  \ar^-{(-)\downarrow_{\varphi}}[u] 
 &\Rep(\G^{(1|1)})  \ar^-{S}[r]  & \Rep(OSp(1|2)). }$$

 Note that the functor $\overline{T}: \Rep(OSp(1|2)) \boxtimes \Rep(OSp(1|2)) \to \Rep(OSp(1|2))$ gives the identity functor when restricted to the full tensor subcategory $\Rep(OSp(1|2)) \boxtimes \sVect$, which gives us the required isomorphism.
\end{proof}

\begin{comment}

    Let $M\in \Rep(G)$ and consider the following three actions on $M^s$: 
    \begin{itemize}
        \item an action of $\G^{(1|1)}$ associated with $x$ (we will denote the group acting in this case by $\G^{(1|1), x}$), 
        \item an action of $OSp(1|2)$ associated with $x_{neat}$ (we will denote the group acting in this case by $OSp_{x_{neat}}$),
        \item an action of $\G^{(1|1)}$ associated with $x_{bal}$ (we will denote the group acting in this case by $\G^{(1|1), x_{bal}}$, to avoid confusion).
    \end{itemize}
    
   By \cref{thm:2-step_exist}, the groups  $OSp_{x_{neat}}$ and $\G^{(1|1), x_{bal}}$ commute, so $M^s$ decomposes into a direct sum of  indecomposable $OSp_{x_{neat}}\times \G^{(1|1), x_{bal}}$-submodules of the form $L\boxtimes L'$, where $L$ is a simple $OSp_{x_{neat}}$-module and $L'$ is an indecomposable $\G^{(1|1), x_{bal}}$-module. These indecomposable $OSp_{x_{neat}}\times \G^{(1|1), x_{bal}}$-submodules are clearly preserved by $x$; the element $x$ acts by $x_{neat}\boxtimes 1+1\boxtimes x_{bal}$. 

   Recall that $\Phi_x$ is defined as the composition of the $s$-invariants functor $$(-)^s:\Rep(G)\to \Rep(\G^{(1|1), x})$$ with the semisimplification functor $ S: \Rep(\G^{(1|1), x})\to\Rep(OSp(1|2))$.
   
Applying the functor $S$ to $L\boxtimes L'$, we obtain $0$ if $x_{bal}\rvert_{L'}$ is a Jordan block of 
even size, otherwise $\Phi_{x}(L\boxtimes L')=\Phi_{x_{neat}}(L) $.

This proves the required statement.
\end{comment}

\begin{corollary}\label{cor:decomp_for_neat_or_bal_elems}
    Let $x\in \g_{\bar 1}$ and let $x=x_{neat}+x_{bal}$ be its $2$-step decomposition as constructed in \cref{thm:2-step_exist}. We have:
    \begin{enumerate}
        \item $x\in \g_{neat}$ if and only if $x=x_{neat}$,
        \item The following conditions are equivalent:
        \begin{enumerate}
            \item\label{it:bal} $x\in \g_{bal}$,
            \item\label{it:strbal}  $x$ is strongly balanced,
            \item\label{it:strbal_suff} $x_{neat} \subset\Im \ad_{x_0}$, where $x_0$ is an attractor of $x_{bal}$. 
        \end{enumerate}
    \end{enumerate}
\end{corollary}
\begin{remark}
We remind the reader that both $\osp_{x_{neat}}$ and $x_0$ are defined up to conjugation.
\end{remark}
\begin{proof}
    \begin{enumerate}
        \item
    By the definition of a neat element, $x\in \g_{neat}$ if and only if the functor $\Phi_x$ is faithful. This happens if and only if $\Phi_{x_{bal}}$ is faithful, which in turn is equivalent to $x_{bal}=0$ (see \cref{rmk:balanced_func_faithful}).

       \item \InnaA{Recall that by by \cref{cor:criterion_balanced}, \eqref{it:strbal} implies \eqref{it:bal}. 
       
       Let us first show that \eqref{it:strbal_suff} implies \eqref{it:strbal}. Let $\frac{1}{2}[x,x]=s+f$ be the Jordan decomposition of $\frac{1}{2}[x,x]$.
       We use the explicit construction of $x_{neat}$ given in \cref{thm:2-step_exist}: there we constructed an $\sl_2$-triple $(e,h,f)$ in $\g_{\bar 0}$ so that $x_{neat}=x_{-1}$ in the $\ad_h$-decomposition $x=\sum_{i\leq 0} x_i$. We now use \cref{prop:x_1_can_be_eliminated}, which states: if $x_{-1}\in \Im \ad_{x_0}$ then $x$ is strongly balanced.

       Finally, we prove that \eqref{it:bal} implies \eqref{it:strbal_suff}. 
       Assume that $x\in \g_{bal}$. Let us show that $x_{neat} \subset\Im \ad_{x_0}$.
       By the definition of a balanced element,
    $x\in \g_{bal}$ implies that the action of $OSp(1|2)$ on $\Phi_x(V)$ is trivial for any $V\in \Rep(G)$. 
    
    Consider the adjoint representation $\g$ of $\g$ and the decomposition of $\g^s$ into indecomposable $\osp_{x_{neat}}\times \kk[x_0]$-summands. We may choose such a decomposition where an indecomposable summand $U$ contains the subspace $\osp_{x_{neat}}$ of $\g$. This summand will be isomorphic, as a $\osp_{x_{neat}}\times \kk[x_0]$-module, to $\osp_{x_{neat}}\boxtimes M$ for some indecomposable $\kk[x_0]$-module $M$. 
    
    Since $\osp_{x_{neat}}$ commutes with $x_0$, we have: $\osp_{x_{neat}}\subset \Ker \ad_{x_0}\rvert_{U}$. By \cref{prop:tensor_functors_2_step}, if the $OSp(1|2)$-action on
    $\Phi_x(\g)$ is trivial then $DS_{x_0}(M)=0$, implying that $DS_{x_0}(U)=0$. This means that 
    $ \Ker \ad_{x_0}\rvert_{U}=\Im \ad_{x_0}\rvert_{U}$, hence $\osp_{x_{neat}} \subset\Im \ad_{x_0}$.
}
       \end{enumerate}
\end{proof}

\begin{remark}
    From the proof of \cref{thm:2-step_exist}, we see that for any balanced (equivalently, strongly balanced) element $x\in \g_{\bar 1}$, there exists $\sl_x=\mathrm{span}\{e,h,f\}$ such that the $\ad_h$-decomposition of $x$ contains only eigenvectors with even eigenvalues.
\end{remark}
\begin{corollary}\label{cor:Phi_x_is_DS_att}
    Let $Forg:\Rep(OSp(1|2))\to \sVect$ be the forgetful functor and consider the functor
$\overline{\Phi}_x:=Forg\circ\Phi_x:\Rep(G)\to \sVect$.
    Then $$\overline{\Phi}_x\cong \overline{\Phi}_{x_{bal}}\cong\overline{\Phi}_{x_0}=DS_{x_0}.$$
\end{corollary}
\begin{proof} 
Let $i_{x_{neat}}: OSp(1|2) \to G$ be as in \cref{ssec:neat_notn}.
Consider the decomposition $$\InnaA{\Phi_x\cong (\id\boxtimes\Phi_{x_{bal}})\circ (-)\downarrow_{\varphi}}$$ as in \cref{prop:tensor_functors_2_step}. Composing $\Phi_x$ with the forgetful functor we obtain $$\InnaA{\overline{\Phi}_x\cong (Forg \boxtimes \id) \circ (\id\boxtimes\Phi_{x_{bal}})\circ (-)\downarrow_{\varphi} \cong \Phi_{x_{bal}}\circ (-)\downarrow_{\bar \varphi} \cong DS_{x_0}}$$ \InnaA{where $\bar \varphi: G'\to G$ is embedding of the centralizer $G'$ of $\Im(i_{x_{neat}})$} and the rightmost isomorphism follows from \cref{cor:criterion_balanced}.
\end{proof}

\begin{corollary}
    The functor $\Phi_x$ satisfies the Hinich property (``exact in the middle''): that is, given a short exact sequence $$0\to M'\xrightarrow{f'} M \xrightarrow{f''} M''\to 0$$ in $\Rep(G)$, the sequence $${\Phi}_x(M')\xrightarrow{{\Phi}_x(f')} {\Phi}_x(M) \xrightarrow{{\Phi}_x(f'')} {\Phi}_x(M'')$$ is exact and satisfies: $\Ker {\Phi}_x(f') \cong \Pi ~\mathrm{coker}\,{\Phi}_x(f'')$. 
\end{corollary}

\begin{proof}
    Let $x=x_{neat}+x_{bal}$ be a $2$-step decomposition of $x$ and let $x_0$ be the attractor of $x$ obtained from $x_{bal}$. The Duflo-Serganova functor $DS_{x_0}$ satisfies the Hinich property (see \cite[Lemma 2.7]{gorelik2022duflo}). \cref{cor:Phi_x_is_DS_att} then shows that the functor $\overline{\Phi}_x$ satisfies the same property. The functor $Forg$ is exact, implying that the sequence $${\Phi}_x(M')\xrightarrow{{\Phi}_x(f')} {\Phi}_x(M) \xrightarrow{{\Phi}_x(f'')} {\Phi}_x(M'')$$ is exact. Finally, the map $(f'')^{-1}\circ x_0\circ (f')^{-1}: \Pi \Phi_x(M'')\to \Phi_x(M')$ induces an isomorphism of vector superspaces $\Ker {\Phi}_x(f') \cong \Pi ~\mathrm{coker}\,{\Phi}_x(f'')$ (this is a direct computation, see \cite[Lemma 2.7]{gorelik2022duflo}). 
    %All the factors in this composition commute with $x_{neat}$ (see conditions of \cref{thm:2-step_exist}), so the above isomorphism is actually an isomorphism of $\osp_{x_{neat}}$-modules, completing the proof.
\end{proof}

\begin{onehalfspace}
\bibliographystyle{alpha}
\bibliography{biblio_nilp}
\end{onehalfspace}
\newpage
\appendix
\section{Neat and balanced elements in Levi sub-superalgebras}\label{app:detecting_neat_and_bal}
 
\begin{lemma}\label{lem:restr_neat_bal_to_levi}
Let $x\in \mathcal{N}^{\ad}(\g_{\bar 1})$ and let $\fl\subset \g$ be a Levi subalgebra such that $x\in \fl$. If $x\in \g_{neat}$ then $x\in \fl_{neat}$.
\end{lemma}

\begin{proof}
    
    Assume $x\in \g_{neat}$, $x\neq 0$ (for $x=0$ the claim obviously holds). 
    Let $\s\subset \mathfrak{c}_{\g_{\bar 0}}(x)$ be the toral subalgebra defining the Levi sub-superalgebra $\fl$.

    Let $x\in \osp_x\subset \g$ be an $\osp(1|2)$-type subalgebra and let $h\in \osp_x$ be the corresponding Cartan element, so that $[h,x]=-x$. Consider the $\ad_h$-grading $\g=\oplus_{i\in \Z} \g^i$. The operator $\ad_h$ preserves $\mathfrak{c}_{\g_{\bar 0}}(x)$, hence $\mathfrak{c}_{\g_{\bar 0}}(x)$ inherits the $\Z$-grading; since $\ad_x$ annihilates any homogeneous element in $\mathfrak{c}_{\g_{\bar 0}}(x)$, these all lie in grades less or equal $0$. Now, $\mathfrak{c}_{\g_{\bar 0}}(x)\cap (\bigoplus_{i<0} \g^i)$ is the nilradical of $\mathfrak{c}_{\g_{\bar 0}}(x)$, so $\s$ is $C_{G_{\bar 0}}(x)$-conjugate to a toral subalgebra of $ \mathfrak{c}_{\g_{\bar 0}}(x)\cap \g^0 = \Ker \ad_x\cap \Ker \ad_h$. \InnaA{Since the conjugate subalgebra lies in $\Ker \ad_x\cap \Ker \ad_h$, it commutes with $\osp_x$. In other words, $\osp_x$ is $C_{G_{\bar 0}}(x)$-conjugate to an $\osp(1|2)$-type subalgebra $\osp_x'$ which contains $x$ and commutes with $\s$. This implies that $\osp_x'\subset\fl$, so $x\in \fl_{neat}$.}
    \end{proof}

\begin{lemma}\label{lem:balancedLevi} Let $x\in \mathcal{N}^{\ad}(\g_{\bar 1})$ and let $\fl\subset \g$ be a Levi subalgebra such that $x\in \fl$. If $x$ is strongly balanced in $\g$ then $x$ is strongly balanced in $\fl$.
    \end{lemma}
    \begin{proof} Let $f:=\frac{1}{2}[x,x] \in \fl$. If $f=0$ then we are done, so we will assume that $f\neq 0$. Fix an $\sl_2$-triple $(e,h,f)$ such that $x_{-1}=0$ in the $\ad_h$-decomposition $x=\sum_{i\leq 0} x_i$. Any  $\sl_2$-triple containing $f$ is $C_{G_{\bar 0}}(f)$-conjugate to $(e,h,f)$. \InnaA{Since $\fl$ is quasi-reductive, we may choose an $\sl_2$-triple $(e',h',f)$ in $\fl_{\bar 0}$} and consider the $\ad_{h'}$-decomposition
    $x=\sum x'_i$. By \cref{prop:x_1_can_be_eliminated}, it is enough to prove that $x'_{-1} \in \Im \ad_{x'_0}\rvert_{\fl}$.
    
    Let $g\in C_{G_{\bar 0}}(f) $ be such that $e'=\operatorname{Ad}_g(e)$ and
     $h'=\operatorname{Ad}_g(h)$. Let $y:=\operatorname{Ad}_g(x)$. Then the $\ad_{h'}$-decomposition $y=\sum_i y_i$ satisfies: $y_{-1}=0$.
     
     Consider the $\ad_{h'}$-grading $\g=\bigoplus_{i\in \Z} \g^i$. The operator $\ad_{h'}$ preserves $\mathfrak{c}_{\g_{\bar 0}}(f)$ so we have an induced $\Z$-grading on $\mathfrak{c}_{\g_{\bar 0}}(f)$. Any element in $\g=\bigoplus_{i\in \Z} \g^i$ annihilated by $\ad_f$ must have a non-positive $h$-weight so $$\mathfrak{c}_{\g_{\bar 0}}(f)=\bigoplus_{i\leq 0}\mathfrak{c}_{\g_{\bar 0}}(f) \cap \g^i.$$
     Thus $\operatorname{Ad}_g=\exp{\ad u}$ 
     for some $u\in\bigoplus_{i\leq0}\mathfrak{c}_{\g_{\bar 0}}(f) \cap \g^i$. \InnaA{In fact, we may assume that $ u\in\bigoplus_{i<0}\mathfrak{c}_{\g_{\bar 0}}(f) \cap \g^i$, since $\g^f_0 = \fc_{\g_{\bar 0}}(h',f')$ commutes with the entire $\sl_2$-triple $(e',h',f)$.}
     
     Let $u=\sum_{i<0}u_i$ be the $\ad_{h'}$-decomposition of $u$. Then
     $y_{-1}=\ad_{u_{-1}} x'_0+x'_{-1}=0$. Thus, $x'_{-1}=[x'_0,u_{-1}]$ for some $u_{-1}\in\g$. 
     
     \InnaA{Consider the adjoint action of the reductive Lie algebra $\fl_{\bar 0}$ on $\g_{\bar 0}$. There exists an $\fl_{\bar 0}$-invariant decomposition 
     $\g_{\bar 0}=\fl_{\bar 0}\oplus\fm$. Let $u_{-1}=u'_{-1}+u''_{-1}$ with $u'_{-1}\in\fl_{\bar 0}$} and $u''_{-1}\in\fm$. Then $x'_{-1}=[x'_0,u'_{-1}]$, meaning that $x'_{-1} \in \Im \ad_{x'_0}\rvert_{\fl}$. By \cref{prop:x_1_can_be_eliminated} we get that $x$ is strongly balanced in $\fl$.
        
      \end{proof}

\end{document}